\documentclass[11pt,a4paper]{amsart}

\usepackage[T1]{fontenc}
\usepackage[utf8]{inputenc}
\usepackage{lmodern}

\usepackage[english]{babel}
\usepackage{amsmath,amsfonts,amssymb,amsthm}

\usepackage{graphicx} 

\usepackage[bottom=2.5cm, top=2.5cm, left=3cm, right=3cm]{geometry}

\usepackage[foot]{amsaddr}
\usepackage[alphabetic, abbrev]{amsrefs}

\usepackage{mathtools}
\mathtoolsset{showonlyrefs,showmanualtags}	

\usepackage{bm}	
\DeclareMathAlphabet{\mathbbold}{U}{bbold}{m}{n}	

\theoremstyle{plain}
\newtheorem{theorem}{Theorem}[section]
\newtheorem{corollary}[theorem]{Corollary}
\newtheorem{lemma}[theorem]{Lemma}
\newtheorem{proposition}[theorem]{Proposition}
\theoremstyle{definition}
\newtheorem{definition}[theorem]{Definition}

\theoremstyle{remark}
\newtheorem{remark}[theorem]{Remark}
\numberwithin{equation}{section}

\usepackage[bookmarksdepth=3]{hyperref}
\usepackage{xcolor}
\hypersetup{
    colorlinks,
    linkcolor={red!50!black},
    citecolor={blue!50!black},
    urlcolor={blue!80!black}
}

\definecolor{reviewteal}{RGB}{0,105,105}

\usepackage[normalem]{ulem}

\usepackage{todonotes}

\usepackage{enumitem}
\setlist[itemize]{leftmargin=1cm}

\usepackage{accents}

\DeclareMathOperator{\tr}{\mathrm{Tr}}

\DeclareMathOperator{\supp}{\mathrm{supp}}

\usepackage{comment}

\def\fh{\mathfrak{h}}
\def\fg{\mathfrak{g}}

\def\fv{\mathfrak{v}}

\def\fh{\mathfrak{h}}

\def\cE{\mathcal{E}}
\def\cF{\mathcal{F}}

\def\cL{\mathcal{L}}

\def\cP{\mathcal{P}}
\def\cR{\mathcal{R}}
\def\cS{\mathcal{S}}
\def\cU{\mathcal{U}}

\def\cW{\mathcal{W}}

\def\R{{\mathbb R}}			
\def\C{{\mathbb C}}			
\def\N{{\mathbb N}}			
\def\Z{{\mathbb Z}}		
\def \heis {\mathbb{H}} 
\def\D{\mathbb{G}}

\def\bDhat{\widehat{\mathbb{G}}}

\def\pr{{\rm pr}}

\def\mz{\setminus\{0\}}

\newcommand{\trace}[1]{\tr\big[#1\big]}

\newcommand{\vertiii}[1]{{\left\vert\kern-0.25ex\left\vert\kern-0.25ex\left\vert #1 
		\right\vert\kern-0.25ex\right\vert\kern-0.25ex\right\vert}}

\mathchardef\mhyphen="2D 

\let\oldtocsection=\tocsection 
\let\oldtocsubsection=\tocsubsection 
\let\oldtocsubsubsection=\tocsubsubsection

\renewcommand{\tocsection}[2]{\hspace{0em}\oldtocsection{#1}{#2}}
\renewcommand{\tocsubsection}[2]{\hspace{1em}\oldtocsubsection{#1}{#2}}
\renewcommand{\tocsubsubsection}[2]{\hspace{2em}\oldtocsubsubsection{#1}{#2}}

\author[Davide Barilari]{Davide Barilari$^{*}$}
\address{$^{*}$Dipartimento di Matematica ``Tullio Levi-Civita'', Universit\`a degli Studii di Padova}
\email{\href{mailto:davide.barilari@unipd.it}{davide.barilari@unipd.it}}
\author[Steven Flynn]{Steven Flynn$^{\sharp}$}
\address{$^{\sharp}$Department of Mathematics, University College London}
\email{\href{mailto:steven.flynn@ucl.ac.uk}{steven.flynn@ucl.ac.uk}}

\title[Refined Strichartz Estimates for 
$H$-type Carnot groups]{Refined Strichartz Estimates for sub-Laplacians \\ in Heisenberg
	and
	$H$-type groups}

\date{\today}
\begin{document}
	\begin{abstract}
 		We obtain refined Strichartz estimates for the sub-Rie\-man\-nian Schr\"{o}dinger equation on $H$-type Carnot groups using Fourier restriction techniques. 
		In particular, we extend the previously known Strichartz estimates obtained in \cite{BBG21}
		for the Heisenberg group also to non radial initial data.  The same arguments permits to obtain refined Strichartz estimates for the wave equation on $H$-type groups.

		Our  proof is based on estimates for the spectral projectors for sub-Laplacians and reinterprets Strichartz estimates as Fourier restriction theorems for nilpotent groups in the context of trace-class operator valued measures.

	\end{abstract}

	\maketitle

	\setcounter{tocdepth}{1}
	\tableofcontents

	
	\section{Introduction}
	The Schr\"{o}dinger equation in the Euclidean space is one the main prototype of dispersive equations. Indeed it is well-known that the solution of
	\begin{equation}\label{eq:sch-euclidean-i}
		\begin{cases}
			i\partial_t u -\Delta u=0 \\
			u|_{t=0}=u_0
		\end{cases}
	\end{equation}
	where $\Delta$ denotes the classical  Euclidean Laplace operator, satisfies the following a priori estimate for every $t\neq 0$
	\begin{equation}\label{eq:sch-heis-2}
		\|u(t)\|_{L^{\infty}}\leq \frac{C}{|t|^{n/2}}\|u_{0}\|_{L^{1}}.
	\end{equation}
	The inequality \eqref{eq:sch-heis-2} is at the core of the proof of the following Strichartz estimates, which are space-time estimates of the following type
	\begin{equation}\label{eq:sch-heis-3}
		\|u\|_{L^{q}_{t}L^{p}_{x}}\leq C \|u_{0}\|_{H^{\sigma}}
	\end{equation}
	where the pair $(p,q)$ satisfies the relation 
	\begin{equation}\label{eq:indici1}
		\sigma:=\frac{n}{2}-\frac{2}{q}-\frac{n}{p}\geq 0.
	\end{equation}
	Estimates \eqref{eq:sch-heis-3} are fundamental for applications to the non-linear Schr\"{o}dinger equation and passing from \eqref{eq:sch-heis-2} to \eqref{eq:sch-heis-3} is nowadays formalized in a classical abstract argument known as $TT^{*}$.

	More recently, the Schr\"{o}dinger equation in the sub-Riemannian setting started to attract some attention due to the following key observation contained in \cite{BGX00}: in the Heisenberg group $\mathbb{H}^{d}$ endowed with its sub-Laplacian $\Delta_{H}$ there exist a choice of the initial condition $u_{0}$ for which the solution
	\begin{equation}\label{eq:sch-heis-i}
		\begin{cases}
			i\partial_tu-\Delta_{H} u=0 \\
			u|_{t=0}=u_0
		\end{cases}
	\end{equation}
	behaves like a transport, thus a dispersive estimate like \eqref{eq:sch-heis-2} cannot hold. Nevertheless, it has been proved in \cite{BBG21} that for radial initial data $u_{0}$  (the terminology actually stands for a cylindrical symmetry with respect to the last coordinate), we have the following Strichartz type estimate 
	\begin{equation}\label{eq:sch-heis-4}
		\|u\|_{L^{\infty}_{\fv}L^{q}_{t}L^{p}_{\fh}}\leq C \|u_{0}\|_{H^{\sigma}}
	\end{equation}
	for pairs $(p,q)$ satisfying the relation
	\begin{equation}\label{eq:indici2}
		\sigma:=\frac{2d+2}{2}-\frac{2}{q}-\frac{2d}{p}\geq 0,
	\end{equation}
	where $\fg= \fh\oplus \fv$ denotes the splitting of the Lie algebra of Heisenberg group $\mathbb H^{d}$ into the 1-dim vertical and the $2d$-dim horizontal directions. 
	
	We stress that with respect to \eqref{eq:indici1}, in formula \eqref{eq:indici2} the topological dimension $n$ is replaced in the two occurrencies by $Q=2d+2$ the homogeneous dimension (which corresponds to the rescaling of the Lebesgue measure with respect to the non isotropic dilations) and by $2d$ the dimension of the horizontal distribution.
	
	\medskip
	The inequality \eqref{eq:sch-heis-4} for radial initial conditions in the Heisenberg group has been proved inspired by  the original approach of Strichartz, which is based on a Fourier restriction theorem. Indeed, even in the Euclidean case, the solution to \eqref{eq:sch-euclidean-i} 
	can be interpreted, setting  $y=(t,x)$ and $z=(\tau,\xi)$,
	as an inverse Fourier transform 
	\begin{equation} \label{eq:invintro}
		u(t,x)=\int_{\R^{n}}e^{i(x\cdot \xi + t|\xi|^{2})}\widehat{u}_{0}(\xi)d\xi = \int_{ \Sigma}e^{i y\cdot z} g(z)d\Sigma(z)
	\end{equation}
	of the function $g(z)=g(|\xi|^{2},\xi)=\widehat{u}_{0}(\xi)$ defined on the paraboloid 
	\begin{equation} \label{eq:parabintro}
		\Sigma = \{z=(\tau,\xi)\in   \R\times  \R^{n}\mid \tau=|\xi|^{2} \}\subset \R^{n+1}=\R\times  \R^{n},
	\end{equation}
	where $\R\times  \R^{n}$ is understood in the frequency space. Here $d\Sigma$ is the non-intrinsic  measure on $\Sigma$ defined by the pull-back of the Lebesgue measure by projection onto the first $n$ variables. The result then can be reduced to a restriction estimate for the Fourier transform on the paraboloid $\Sigma$.
	
	Performing a similar analysis in the Heisenberg group is less straightforward: the Heisenberg group is a non-commutative group so the ordinary Fourier transform should be replaced by the non-commutative one; moreover, the surface $\Sigma$ is actually an hypersurface of the space of frequencies, which in the case of a non-commutative group is parametrized by unitary irreducible and infinite dimensional representations.
	
	Such difficulties have been considered in \cite{BBG21} in which Strichartz estimates for the Schr\"{o}dinger/wave equation are proved in the Heisenberg group, for radial   initial condition.
	\subsection{Heisenberg group, no radial assumption}

	In this paper we remove the radial assumption thus extending the result of \cite{BBG21} to every initial condition. The estimates are stated in terms of Sobolev spaces. The norms $H^{\sigma}$ and $\dot B^{\sigma}_{2,1}$ are defined in \eqref{eq:sobolev-def}-\eqref{eq:besov-def}.
	
	\begin{theorem}\label{t:main1}
		Given $(p,q)$ belonging to  the admissible set
		\begin{equation}\label{eq:bbg0}
		\mathcal{A} = \Big\{(p,q)\in (2,\infty)^2\, \mid \, p \leq q \quad
		\mbox{and} \quad\frac2q+\frac{2d}p \leq  \frac Q2\Big\}
		\end{equation}
		the  solution to  the Schr\"odinger equation in the Heisenberg group 
		\begin{equation}\label{eq:scf}
			\begin{cases}
				i\partial_tu-\Delta_{H} u=f \\
				u|_{t=0}=u_0
			\end{cases}
		\end{equation}
		satisfies  
		\begin{equation}\label{eq:bbg1}
		\|u\|_{L^\infty_\fv L^{q}_t L^{p}_{\fh}} \leq C  (\|u_0\|_{H^{\sigma}} +\|f\|_{L^{1}_{t}H^{\sigma}})\,.
		\end{equation}
		where $\sigma = \frac Q2-\frac2q-\frac{2d}p$, recalling that $Q=2d+2$. If instead $q=\infty$ and $2<p\leq\infty$, then $$\|u\|_{L^\infty_\fv L^{\infty}_t L^{p}_{\fh}} \leq C  (\|u_0\|_{\dot B^{\sigma}_{2,1}} +\|f\|_{L^{1}_{t}\dot B^{\sigma}_{2,1}}).$$
	\end{theorem}
		\begin{remark}\label{r:pq-finite}
		 The set of admissible indices \eqref{eq:bbg0} in Theorem~\ref{t:main1} excludes $p=\infty$ and $q=\infty$. Notice, for instance, that $p\leq q$ and $p=\infty$ force $q=\infty$, where the estimates \eqref{eq:bbg1} are false, as this would imply a critical embedding of $H^{Q/2}$ in $L^{\infty}$ which is not true. 
		 
		 For the same reason, the endpoint estimates for infinite indices in \cite{BBG21} should be understood in terms of Besov spaces.
	
	\end{remark}

	It is crucial that the vertical variable has a $L^{\infty}$ norm due to the counterexample of Muller in \cite{mullerRestrictionTheoremHeisenberg1990}. We stress that $\sigma=0$  is not admitted here, since together with $p\leq q$ forces $p=2$. The endpoint case $p=2$ is excluded in our statement since our proof does not work for $(m,p)=(1,2)$, where $m$ is the vertical dimension. Notice that in \cite{BBG21} the endpoint case has been proved for radial initial data with a separate ad hoc argument.

	\subsection{$H$-type groups} Consider now a more general Carnot group of step 2 with Lie algebra $\fg=\fh\oplus \fv$ of dimension $n=2d+m$ where $\dim \fh=2d$, $\dim \fv=m$ and
	\begin{equation}
		[\fh, \fh]=\fv, \quad [\fh, \fv]=0.
	\end{equation}
	Fix a scalar product  $g$ on $\fg$ such that $\fh\perp \fv$. For any $\mu\in \fv^*$, consider the skew symmetric operator $J_\mu: \fh\to \fh$ defined by
	\begin{equation}
		\langle \mu, [X, Y]\rangle=g(X, J_\mu Y).
	\end{equation}
	A step 2 Carnot group is \textit{$H$-type} if for all $\mu\in \fv^*$ we have
	$J_\mu^2=-\|\mu\|_g^2I$,
	with $I$  the identity map on $\fh$. The sub-Laplacian $\Delta_{H}$ on the $H$-type group $G$ is given by
	\begin{equation}
		\Delta_H=\sum_{j=1}^{2d} X_j^2
	\end{equation}
	where $\{X_j\}_{j=1}^{2d}$ is an orthonormal basis for $\fh$. We consider the Schr\"{o}dinger equation
	\begin{equation}\label{eq:sch-heis-8i}
		\begin{cases}
			i\partial_tu-\Delta_{H} u=f \\
			u|_{t=0}=u_0
		\end{cases}
	\end{equation}
	We obtain the following result.
	(Recall that $H$-type groups of dimension $n=2d+m$ have homogeneous dimension $Q=2d+2m$.) 
	
	\begin{theorem}\label{t:main2}
		Let $G$ be an $H$-type group with $\dim \fh=2d$ and $\dim \fv=m$. Assume $(m,p)\neq (1,2)$. Given $(r,p,q)$ belonging to  the admissible set
		\begin{equation}\label{eq:bbg2}
		\mathcal{A} = \Big\{(r,p,q)\in [2,\infty]\times[2,\infty)^2\, /\, p \leq r,q; \ \frac1r\leq\frac{m-1}{2(m+1)}, \ \frac2q+\frac{2d}p + \frac{2m}{r}\leq  \frac Q2\Big\}
		\end{equation}
		the  solution to  the Schr\"odinger equation \eqref{eq:sch-heis-8i} on an $H$-type Carnot group satisfies
		$$
		\|u\|_{L^r_\fv L^{q}_t L^{p}_{\fh}} \leq C (\|u_0\|_{H^{\sigma}} +\|f\|_{L^{1}_{t}H^{\sigma}})\,.
		$$
		where $\sigma = \frac Q2-\frac2q-\frac{2d}p-\frac{2m}{r}$ and $Q=2d+2m$. If instead $p=\infty$ or $q=\infty$ but $(r, p, q)$ otherwise satisfies the  admissibility conditions in \eqref{eq:bbg2}, then $$\|u\|_{L^r_\fv L^{q}_t L^{p}_{\fh}} \leq C (\|u_0\|_{\dot B^{\sigma}_{2,1}} +\|f\|_{L^{1}_{t}\dot B^{\sigma}_{2,1}}).$$
	\end{theorem}
	 Notice that in this case, when the codimension $m>1$, we  also get the endpoint case $p=2$. Also, $m=1$ forces $r=\infty$   and $p> 2$, which is exactly the Heisenberg case. We stress that for $m>1$ we can remove the $L^{\infty}$ restriction in the vertical variable,  coherently with the fact that when $m>1$ the dispersion rate is non trivial, cf.\ Section~\ref{s:link}.
	
	\begin{remark}\label{r:endpoint}
	In the admissible set \eqref{eq:bbg2} we have $p,q<\infty$, which is used only in the proof of Lemma~\ref{l:recombination}, where the multiplier theorem is applied in $L^{q'}_tL^{p'}_\fh$ and requires $1<p',q'$.
		 For $q=\infty$ and $p<\infty$, Lemma~\ref{l:recombination}, which needs $p, q<\infty$, is unavailable. 
		
		 For completeness, details on the relation between Besov and Sobolev spaces in this context are contained in Appendix~\ref{a:estfalse}.

	\end{remark}

	\subsection{On the wave equation} The same approach permits to obtain Strichartz estimates on the wave equation associated with the sub-Laplacian on $H$-type groups
	\begin{equation}\label{eq:wave-heis-i}
		\begin{cases}
			\partial^{2}_tu-\Delta_{H} u=f \\
			(u,\partial_{t}u)|_{t=0}=(u_0,v_{0})
		\end{cases}
	\end{equation}
	Also in this case in the Heisenberg group we remove the radial assumption in the initial data, and we can generalize to $H$-type groups.
	\begin{theorem}\label{t:main3}
		Let $G$ be an $H$-type group with $\dim \fh=2d$ and $\dim \fv=m$.  Assume $(m,p)\neq (1,2)$. Given $(r,p,q)$ belonging to  the admissible set
		\begin{equation}\label{eq:bbg3}
		\mathcal{A} = \Big\{(r,p,q)\in  [2,\infty]\times[2,\infty)^2\, /\, p \leq r,q; \ \frac1r\leq\frac{m-1}{2(m+1)}, \ \frac1q+\frac{2d}p + \frac{2m}{r}\leq  \frac {Q-2}{2}\Big\}
		\end{equation}
		the  solution to  the wave equation \eqref{eq:wave-heis-i} on an $H$-type Carnot group satisfies
		$$
		\|u\|_{L^r_\fv L^{q}_t L^{p}_{\fh}} \leq C ( \|\nabla_{H}u_0\|_{ H^{ \sigma}}+\|v_0\|_{ H^{ \sigma}} +\|f\|_{L^{1}_{t} H^{ \sigma}})\,.
		$$
		where $\sigma = \frac {Q-2}{2}-\frac1q-\frac{2d}p-\frac{2m}{r}$ and $Q=2d+2m$. If instead $p=\infty$ or $q=\infty$ but $(r, p, q)$ otherwise satisfies the admissibility conditions \eqref{eq:bbg3},  then 
		$$\|u\|_{L^r_\fv L^{q}_t L^{p}_{\fh}} \leq C ( \|u_0\|_{\dot B^{\sigma+1}_{2,1}}+\|v_0\|_{\dot B^{\sigma}_{2,1}} +\|f\|_{L^{1}_{t}\dot B^{\sigma}_{2,1}}).$$ 
	\end{theorem}
	Again, we highlight that $m=1$ forces $r=\infty$ and $p>2$, recovering the Heisenberg case, for which the corresponding estimate has been proved for radial data in \cite{BBG21}. 	
		
We stress that the endpoint estimates for infinite indices in \cite{BBG21} should be understood in terms of Besov spaces,
 as for the Schr\"{o}dinger case, cf.\ Remark~\ref{r:pq-finite}.

	\subsection{Link with the Fourier restriction}

	In the Euclidean case one can observe that the equation \eqref{eq:invintro} can be rewritten as 
	$u=R^{*}_{\Sigma}g$
	where $R^{*}_{\Sigma}$ is the formal adjoint of the Fourier restriction operator
	$$R_{\Sigma}f=\cF(f)|_{\Sigma},$$
	where $\Sigma$ is the paraboloid in $\R^{n+1}$ defined in \eqref{eq:parabintro}. It is not difficult to prove that the continuity of $R^{*}_{\Sigma}$ from $L^{2}(\Sigma)$ to $L^{q'}_{t}L^{p'}_{x}(G)$ is equivalent to the continuity of the operator
	$$R^{*}_{\Sigma}R_{\Sigma}f=f*\kappa_{\Sigma}$$	
	from the space	$L^{q}_{t}L^{p}_{x}(G)$ to $L^{q'}_{t}L^{p'}_{x}(G)$, that is an estimate of the form.
	\begin{equation}	\label{eq:introgsigma}
		\|f*\kappa_{\Sigma}\|_{ L_t^{q'}L_x^{p'}}\leq C \|f\|_{L^q_tL^p_x},
	\end{equation}
	where $\kappa_{\Sigma}$ denotes the inverse Fourier transform of the pullback measure $d\Sigma$ induced on $\Sigma$ by the projection onto the first $n$ variables.
	
	\begin{remark} It is important here to observe that $\Sigma$ is a surface in the space of frequencies, which in the Euclidean case is isomoprhic to $\R^{n}$ itself. On general nilpotent groups $G$ the space of frequencies is replaced by the dual of the group $\widehat G$ which is the set of (equivalence classes of) strongly continuous unitary representations of $G$. For more details we refer to Sections~\ref{s:fheis} and \ref{s:proofH}.
	\end{remark}

	A crucial point in the Euclidean case is to observe that one can write, at least formally,
	\begin{equation}\label{eq:intro-kernel}
		f* \kappa_{\Sigma}(t, x)=\int_0^\infty e^{i\mu t}\cP_\mu(f^{\mu})(x)d\mu,
	\end{equation} 
	where $\cP_\mu=\cP^{-\Delta}_\mu$ is the spectral projector of the Euclidean Laplacian $-\Delta$ and $f^{\mu}$ denotes the Euclidean Fourier transform of $f$ with respect to the last variable and evaluated at $\mu$. One can then link the estimate \eqref{eq:introgsigma} to the properties of the spectral projectors.

	Actually to prove such an estimates, and to guarantee the convergence of the integral \eqref{eq:intro-kernel}, one needs compactness either of the surface or of the measure. For this reason one works with a truncation $\psi$ in the $\mu$ variable in the integral \eqref{eq:intro-kernel} and we denote $\kappa_{\Sigma_{\psi}}$ the  inverse Fourier transform of the corresponding compactly supported measure.
	
	\medskip
	The proof of our results is obtained by showing that the solution to the Schr\"{o}dinger Equation on an $H$-type group $G$:
	\begin{equation}\label{eq:sch-intro}
		\begin{cases}
			i\partial_tu-\Delta_{H} u=0 \\
			u|_{t=0}=u_0
		\end{cases}
	\end{equation}
	can be written as a inverse Fourier transform as in \eqref{eq:invintro}
	\begin{align}\label{eq:iftsigma-intro}
		u_{t}(x, z)=\cF_\mathbb{G}^{-1}\{\cF_G(u_0)d\Sigma\}(t, x, z).
	\end{align}
	where $\mathbb{G}=\R\times G$ is the extended space and $\Sigma$ is a suitable subset of the dual space endowed with the relevant measure $d\Sigma$.
	Considering 
	\begin{equation}\label{eq:newlab33}
		\kappa_{\Sigma_{\psi}}=\cF_\mathbb{G}^{-1}(\psi d\Sigma)
	\end{equation} (one actually need to multiply by a bump function $\psi$ to ensure compact support) the proof is then based on the analogue estimates of \eqref{eq:introgsigma} in our setting. We refer to Section~\ref{sec:dual_parab} and Section~\ref{s:frp} for all details.
	
	\smallskip
	To state these estimates we need to introduce a few more objects. The twisted Laplacian $\Delta_{H}^{\lambda }$ is defined by the action of the sub-Laplacian on the $z$-Fourier modes of a function $f=f(x, z)$ on the Heisenberg group:
	$$(\Delta_{H} f)^{\lambda}=\Delta_{H}^{\lambda }f^{\lambda}$$
 	Precisely, letting $y_i:=x_{i+d}$ for $i=1,\ldots d$, we have
	\begin{align*}
		\Delta_{H}^{\lambda}:=\Delta_{\R^{2d}}-i\lambda\sum_{j=1}^{d} \left(y_j\frac{\partial}{\partial x_j}-x_j\frac{\partial}{\partial y_j}\right)-\frac{1}{4}|\lambda|^2\sum_{j=1}^d(x_j^2+y_j^2).
	\end{align*}
	
	The twisted Laplacian $\Delta^\lambda_{H}$ is an essentially self-adjoint elliptic operator on the set $S(\mathbb{C}^d)\subset L^2(\mathbb{C}^d)$ with a pure point spectrum consisting of eigenvalues $\{-|\lambda|(2k+d)\mid k\in \N\}$.
	The following estimate involves the norm of the spectral projector $\Lambda_k$ for the twisted Laplacian, i.e., the projector on the $k$-th eigenspace of the operator $\Delta^{\lambda}_{H}$ when $\lambda=1$.

	\medskip
	For the readers convenience we state the results separately for the Heisenberg group and the $H$-type case, to highlight the differencies when the codimension $m$ is equal to one  and when it is greater.
	
	\begin{proposition}\label{prop:TTstarintro}
		Let $\mathbb{G}=\R\times\heis^{d}$, and $\kappa_{\Sigma_{\psi}}$ be as defined in  \eqref{eq:newlab33}. For any $p, q\in  [1, 2)$ satisfying  $q\leq p$   there exists $C>0$ such that  for all $f\in \cS(\mathbb{G})$   
		\begin{equation}
			\|f*_{\mathbb{G}} \kappa_{\Sigma_{\psi}}\|_{L_\fv^\infty L_t^{q'}L_\fh^{p'}}\leq C \left(\sum_{k=0}^\infty\frac{\|\Lambda_k\|_{L^p\to L^{p'}}}{(2k+d)^{1+d(\frac{1}{p}-\frac{1}{p'})}} \right)\|f\|_{L^1_\fv L^q_tL^p_\fh},
		\end{equation}
		and  the series in the right hand side is convergent. 
	\end{proposition}
	Notice that the endpoint case $p=2$ is excluded. Concerning $H$-type groups we prove the following result. Recall that  $m=\dim \fv$.
	\begin{proposition}\label{prop:TTstarintro2}
		Let $\mathbb{G}=\R\times G$, where $G$ is an $H$-type with $\dim \fh=2d$ and $\dim \fv=m$. Let $\kappa_{\Sigma_{\psi}}$ be defined in \eqref{eq:newlab33}. For any $p, q\in  [1, 2]$ and $1\leq r \leq 2(m+1)/(m+3)$ satisfying  $q\leq p$ and $r\leq p$  with $(m,p)\neq (1,2)$, there exists $C>0$ such that for all $f\in \cS(\mathbb{G})$     
		\begin{equation}\label{eq:strichartz_estimate}
			\|f*_{\mathbb{G}} \kappa_{\Sigma_{\psi}}\|_{ L^{r'}_\fv L^{q'}_tL^{p'}_\fh}\leq C \left(\sum_{k=0}^\infty\frac{\|\Lambda_k\|_{L^p\to L^{p'}}}{(2k+d)^{m(\frac{1}{r}-\frac{1}{r'})+d(\frac{1}{p}-\frac{1}{p'})}} \right)\|f\|_{ L^r_\fv L^q_tL^p_\fh},
		\end{equation}
		and the series in the right hand side is convergent. 
	\end{proposition}
	
	Given this result, proved in Section~\ref{s:proofH}, the proof of Strichartz estimates is then reduced to the following steps: understand the solution as the Fourier extension of a measure defined on a suitable subset, prove boundedness of the  Fourier restriction and extensions, prove the equivalence between  the Fourier restriction/extension estimates with estimates on the spectral projectors. This is the object of Sections \ref{s:fr} -- \ref{s:frp}.

	\subsection{Link with the literature} \label{s:link}
	
	Strichartz estimates for the Schr\"odinger equation in $H$-type groups have been considered in \cite{DH05}. These Strichartz estimates are consequence of dispersive estimates available in this setting when $m>1$, cf. also \cite{BFG16} for dispersive inequalities on more general step 2 Carnot groups. The result obtained here are refined estimates for two reasons: our results are available even in the case there is no dispersion, morever we obtain estimates in mixed Lebesgue norms in space.  
	
	We mention also similar Strichartz estimates in the Grushin space \cite{GMS23,BL25} and more recent works on smoothing effect for the sub-Riemannian Schr\"o\-din\-ger equation in Heisenberg group \cite{FMRS23,FMRS24}.

	Sharp $L^p\to L^2$ estimates for the spectral projector $\Lambda_k$ for the twisted Laplacian are given in \cite{kochSpectralProjectionsTwisted2004}. Recently Jeong, Lee and Ryu \cite{jeongSharpEstimateFof2022} proved sharp $L^p\to L^q$ estimates for $\Lambda_k$. This result is an improvement on \cite{kochSpectralProjectionsTwisted2004}, since it includes $L^p\to L^q$ estimates the case when $p\neq 2$ and $q\neq 2$. \
	
	In the main results of this article, the sets of admissible indices always require $p\leq q$ and $p\leq r$, with $p>2$ when $m=1$.  These requirements comes from technical points of the proof and the constraint $p\leq q$ was already present in the previous work \cite{BBG21}.

	\subsection{Plan of the paper} In Section~\ref{s:prel} we introduce $H$-type groups, while in Section~\ref{s:fheis} we introduce the elements of Fourier Analysis in  $H$-type groups needed to prove Propositions~\ref{prop:TTstarintro}, the proof being contained in Section~\ref{s:proofH}.
	Section~\ref{s:fr} contains the statement of the equivalence between the Fourier restriction problem and the estimates given in Propositions~\ref{prop:TTstarintro} (cf.\ also Proposition~\ref{prop:TTstarintro2} for the wave equation). Section~\ref{s:schro} proves Stricharts estimates as stated in Theorems~\ref{t:main1} -- \ref{t:main2}, and Section~\ref{s:wave} extends the results to the wave equations. Section~\ref{s:frs} proves the equivalence stated in Section~\ref{s:fr}. Appendices~\ref{s:app} -- \ref{s:app2} contains some technical results and comments. Appendix~\ref{s:appino} contains the proof of the inhomogeneous estimates.\\

	\section{Preliminaries}\label{s:prel}
	Let $G$ be a step 2 Carnot group with Lie algebra $\fg=\fh\oplus \fv$ of rank $2d$ and dimension $2d+m$. In particular, $\dim \fh=2d$ and $\dim \fv=m$, and
	\begin{equation}
		[\fh, \fh]=\fv, \quad [\fh, \fv]= [\fv, \fv]=0.
	\end{equation}
	 For any $\mu\in \fv^*$, there is a natural skew symmetric form $\omega_\mu :\fh\times\fh\to \R$ defined by
	\begin{equation}
		\omega_\mu(X, Y)=\langle \mu, [X, Y]\rangle.
	\end{equation}
	Fix a scalar product  $g$ on $\fg$ such that $\fh\perp \fv.$  The skew symmetric operator $J_\mu: \fh\to \fh$ is defined by
	\begin{equation}
		\omega_\mu(X, Y)=g(X, J_\mu Y).
	\end{equation}
	To the Riemannian metric $g$ on $\fg$ there is a cometric $\widetilde{g}$ on $\fg^*$. We will use the cometric implicitly in the sequel.  
	\begin{definition}\label{d:htype}
		A step 2 Carnot group is \textit{$H$-type} if for all $\mu\in \fv^*$, 
		\begin{equation}
			J_\mu^2=-\|\mu\|_g^2I,
		\end{equation}
		where $I: \fh\to \fh$ is the identity map and $\|\cdot\|_g$ is the norm on $\fv^*$ induced from $g$. 
	\end{definition}
	Notice that, since $J_\mu$ is invertible for every $\mu\neq 0$, no nonzero $X\in\fh$ commutes with all of $\fh$, so $\fv$ is the center of $\fg$.
	\begin{remark}
		In the sequel, we will write $|\mu|$ for  the norm of $\mu\in \fv^*$ instead of $\|\mu\|_{g}$.   
	\end{remark}

	The choice of an orthonormal basis $e_1, \ldots, e_{2d}$ of $\fh$ and $f_1, \ldots f_m$ of $\fv$ yields, through the group exponential map $\exp_G: \fg\to G$, coordinates $(x, z)\in \R^{2d}\times\R^{m}$ on $G$ such that $p=(x, z)$ if and only if 
	$$p=\exp_G\left(\sum_{i=1}^{2d} x_ie_i+\sum_{\alpha=1}^m z_\alpha f_\alpha\right).$$
	We will identify the coordinates $(x, z)$  with elements of the group $G$ or the Lie algebra $\fg$ 
	
	Since $[\fh, \fh]=\fv$ there exists a linearly independent family of skew-symmetric  $2d\times2d$ matrices $\{L^\alpha\}_{\alpha=1}^{m}$ such that the Lie bracket reads
	\begin{equation}
		[e_i, e_j]=\sum_{\alpha=1}^mL_{ij}^\alpha f_\alpha,\qquad i,j=1,\ldots,2d.
	\end{equation}
	In these coordinates, the left-invariant vector field $X_i$  corresponding to the Lie algebra element $e_i$ is given by
	\begin{equation}
		X_i=\partial_{x_i}-\frac{1}{2}\sum_{\alpha=1}^m\sum_{j=1}^{2d} L^\alpha_{ij}x_j\partial_{z_\alpha}, \qquad i=1,\ldots,2d.
	\end{equation}
	Furthermore, $Z_\alpha=\partial_{z_\alpha}$ for $\alpha=1,\ldots,m$, are the left invariant vector fields corresponding to the basis elements $f_\alpha$ of $\fv$. 
	The sub-Laplacian $\Delta_{H}$ on the $H$-type group $G$ is  
	\begin{equation}
		\Delta_{H}=\mathrm{div}(\nabla_{H}u)=\sum_{j=1}^{2d} X_j^2.
	\end{equation}
	where $\nabla_{H}u=\sum_{j=1}^{2d}(X_{j}u)X_{j}$ is the horizontal gradient and the divergence is computed with respect to the Haar measure.  
 	
	\subsection{The Heisenberg group} The main prototype of $H$-type group is $\heis^{d}$, that is the Lie group $\R^{2d+1}$ endowed with the group law (here  we set $y_i:=x_{i+d}$ for $i=1,\ldots d$)
	\begin{align*}
		(x, y, z)\cdot (x', y', z')=\left(x+x', y+y', z+z'+\frac{1}{2}\sum_{j=1}^d(x_jy_j'-x_j' y_j)\right).
	\end{align*}
	Its Lie algebra $\fg$ is spanned by the family of left invariant vector fields 
	$$\{X_1, \cdots, X_{d}, Y_1, \cdots Y_{d}, Z\}$$
	satisfying the commutation relations
	\begin{align*}
		[X_i, Y_i]= Z,  \qquad \forall\,  i=1, \ldots, d.
	\end{align*}
	In particular the Lie algebra can be written $\fg=\fh\oplus \fv$ with $\dim \fh=2d$ and $\dim \fv=1$ where
	$$\fh=\mathrm{span}\{X_1, \cdots, X_{d}, Y_1, \cdots Y_{d}\},\qquad \fv=\mathrm{span}\{Z\}.$$
	We refer to \cite{ABB19} for more details on Carnot groups.

		\section{Fourier Analysis and the dual space}\label{s:fheis}

		In this section $G$ is a step 2 Carnot group of $H$-type. We state some results about the Fourier Analysis on nilpotent groups. For more details we refer to \cite{CG}.
		\subsection{Dual Space of $G$} By $\widehat{G}$ we mean the set of (unitary equivalence classes of) strongly continuous unitary representations of $G$. It is given as a set by $\widehat{G}=\widehat{G}_\infty\coprod \widehat{G}_1$ where
		\begin{equation}\label{spec-res-central}
			\widehat{G}_\infty\cong\fv^*\mz, \quad \text{ and }\quad  \widehat{G}_1\cong \fh^*.
		\end{equation}
		Here $\cong$ denotes the identification given by Kirillov's orbit method, for which we refer to \cite{CG} (see also \cite[Appendix]{BBGL}). 
		Since the part $\widehat{G}_1$ has Plancherel measure zero, we will ignore it in our analysis, and simply use $\widehat{G}$ to denote the  part of the unitary dual with nonzero Plancherel measure.

		For  each $\lambda\in \widehat{G}=\fv^*\mz$, we denote $|\lambda|$ the norm induced from the inner product $g$ on $\mathfrak{g}$ used in the definition the $H$-type group. By classical linear algebra arguments, there exists an orthogonal map $T_\lambda:  \R^{2d}\to\fh$ such that 
		\begin{equation}\label{eq:diagJ}
			J_\lambda =|\lambda|T_\lambda\, J\, T_\lambda^{-1}, \quad J=\begin{pmatrix}
				0 & I_{d\times d} \\
				-I_{d\times d} & 0
			\end{pmatrix}.
		\end{equation}
		
		\begin{lemma}\label{lem:projec}
			The mapping $\alpha_\lambda: G\to \heis^{d}$ given by
			\begin{equation}
				\alpha_\lambda(x, z)=\left(T_\lambda^{-1}x, \frac{\lambda\cdot z}{|\lambda|}\right), \quad (x, z)\in \fh\oplus\fv\cong G,
			\end{equation}
			is a surjective homomorphism of Lie groups. In particular, $G/\ker\alpha_\lambda$ is isomorphic to $\heis^{d}$ where $\ker \alpha_\lambda=\{z\in \fv \mid \lambda\cdot z=0\}$ is the orthogonal compliment of $\lambda$ in the center $\fv$ of $G$. 
		\end{lemma}
\begin{proof}
Write $u=(a,b)$ and $v=(a',b')$ in $\R^{2d}$. By \eqref{eq:diagJ},
\[
 \frac{\lambda([T_\lambda u,T_\lambda v])}{|\lambda|}
 =u^\top Jv=a\cdot b'-b\cdot a'.
\]
Thus $\alpha_\lambda$ preserves the Lie bracket, hence the group law by the step-two Baker--Campbell--Hausdorff formula. Its surjectivity and kernel follow from directl verification.
\end{proof}

		Given $\lambda\in \fv^*\mz$ we can now define an irreducibe unitary representation $\pi_\lambda$ of $G$ on $L^2(\R^d)$ by the formula
		\begin{equation}
			\pi_\lambda:=\pi_{|\lambda|}\circ\alpha_\lambda
		\end{equation}
		where $\pi_{|\lambda|}$ is the Schr\"{o}dinger representation of $\heis^{d}$ defined as follows 
		 
		\begin{align*}
			\pi_{|\lambda|}(x, y, z)\Phi(\xi)=e^{i|\lambda|(z+\xi\cdot y+\frac{1}{2}x\cdot y)}\Phi(\xi+x), \quad \Phi\in L^2(\R^d).
		\end{align*}
		Note the abuse of notation of $\pi$ denoting a representation of $G$ or $\heis^d$. 
		
		For any left-invariant vector field $X$ on $G$, its infinitesimal representation is defined by
		\begin{align*}
			d\pi_\lambda(X):=\frac{d}{dt}\bigg|_{t=0} \pi_{\lambda}(\exp(tX))
		\end{align*}
		and we extend the action of $d\pi_\lambda$ to left-invariant differential operators in the natural way. In particular, by choosing a left-invariant frame for $\mathfrak{h}$ that diagonalizes $J_{\lambda}$ we can compute that $d\pi_{\lambda}(\Delta)=-H(\lambda)$,
		where $H(\lambda)$ is the \textit{Hermite operator} given by
		\begin{align}\label{eq:deltastev}
			H(\lambda)=\sum_{j=1}^{d}(-\partial_{\xi_j}^2+|\lambda|^2\xi_j^2)
		\end{align}
		acting on $\cS(\R^d)\subset L^2(\R^d)$.  The set of Hermite functions $\Phi_\alpha^\lambda$ for $\alpha\in \N^d$ defined in Equation \eqref{eq_Hermite_func} forms an orthonormal eigenbasis for $H(\lambda)$. 
		
		Let  $\cS(G)$ be the space of Schwartz functions on $G\cong\R^{2d+m}$. 
		For $\lambda\in\widehat{G}$ define the Fourier transform by
		\begin{align*}
			\cF_G(f)(\lambda)=\int_{G}f(x, z)\pi_\lambda(x, z)dxdz
		\end{align*}
		and the scalar Fourier transform by
		\begin{align*}
			\cF_{G}(f)(\lambda, \alpha, \beta):=\int_{G}f(x, z)E_{\alpha\beta}^\lambda(x, z)dxdz.
		\end{align*}
		where $E_{\alpha\beta}^\lambda(x, z)=\langle \pi_\lambda(x, z)\Phi_\alpha^{|\lambda|}, \Phi^{|\lambda|}_\beta\rangle$. Brackets in matrix coefficients and scalar orthogonal expansions are linear in the first entry. For Hilbert-space duality, including operator pairings of the form $\trace{A^*B}$, we use the conjugate-linear-first convention.
		\begin{remark}
			Note that we choose to use $\pi_\lambda(x, z)$ instead of $\pi_\lambda(x, z)^*$ in the definition of the Fourier Transform. This choice allows for conveniences in notation.
		\end{remark}Define the Schwartz space on the dual space $\widehat{G}$ by
		\begin{equation}\label{eq:schwartz_space}
			\cS(\widehat{G}):=\{\cF_G(f): f\in \cS(G)\}.
		\end{equation}
		\begin{remark}
			We may think of any element $\theta\in \cS(\widehat{G})$ either as a family of operators $\theta=\{\theta(\lambda): {\lambda\in \fv^*\mz}\}$ parametrized on $\fv^*\mz$ or as a scalar valued function  $\theta(\lambda, \alpha, \beta)=\langle \theta(\lambda)\Phi_\alpha^{|\lambda|}, \Phi_\beta^{|\lambda|}\rangle$ defined on $(\fv^*\mz)\times\N^{2d}$.
		\end{remark}
		The inverse Fourier transform for $\theta\in \cS(\widehat{G})$ is given by 
		\begin{align*}
			\cF_{G}^{-1}(\theta)(x, z)&=(2\pi)^{-d-m}\int_{\fv^*}\trace{\pi_\lambda(x, z)^*\theta(\lambda)}|\lambda|^dd\lambda\\
			&=(2\pi)^{-d-m}\sum_{\alpha, \beta\in \N^d}\int_{\fv^*}\theta(\lambda, \alpha, \beta)\overline{E^\lambda_{\alpha\beta}(x, z)}|\lambda|^dd\lambda
		\end{align*}
		 
		The space $L^{2}(\widehat{G})$ is a Hilbert space with the scalar product
		\begin{equation}
			\langle \theta_1, \theta_2\rangle_{L^2(\widehat{G})}:=\int_{\fv^*}\trace{\theta_1(\lambda)^*\theta_2(\lambda)}d\sigma_{\widehat{G}}(\lambda)
		\end{equation}
		where we set  $d\sigma_{\widehat{G}}(\lambda)=(2\pi)^{-d-m}|\lambda|^d d\lambda.$
		Then it holds
		\begin{align}
			 \langle \cF_G(f), \cF_G(g)\rangle_{L^2(\widehat{G})}=\int_G\overline{f(w)}g(w)\,dw.
		\end{align}
		Notice that the Fourier transform $\cF_G$ intertwines the sub-Laplacian and Hermite operator as follows:
		\begin{align*}
			\cF_G(\Delta_H f)(\lambda)=-\cF_G(f)(\lambda)\circ H(\lambda).
		\end{align*}
		
		\subsection{Extended Space}
		Let $\mathbb{G}=\R\times G$ be the extended space. Since $\widehat{G}\cong \fv^*\mz$, we define  $\widehat{\mathbb{G}}:=\widehat{\mathbb{R}}\times (\fv^*\mz)$. The set of Schwartz functions $\cS(\mathbb{G}$) is defined as usual. 
		
		We define the Fourier transform on $\mathbb{G}$ for all $f\in S(\mathbb{G})$ and $(\mu, \lambda)\in \bDhat$ by
		\begin{align*}
			\cF_{\mathbb{G}}(f)(\mu,  \lambda):=\int_{\R\times G}f(t, x, z)e^{-i\mu t}\pi_\lambda(x, z)\,dx\, dz\, dt
		\end{align*}
		and the scalar Fourier transform on $\mathbb{G}$ for $\alpha, \beta\in \N^d$ by
		\begin{align*}
			\cF_{\mathbb{G}}(f)(\mu, \lambda, \alpha, \beta):=\int_{\R\times G}f(t, x, z)e^{-i\mu t}E_{\alpha\beta}^{\lambda}(x, z)\,dx\, dz\, dt.
		\end{align*}
		\begin{remark}
		Thus the forward time transform uses $e^{-i\mu t}$ and the inverse uses $e^{i\mu t}$, while the spatial transform uses $\pi_\lambda$ and its inverse uses $\pi_\lambda^*$.
		\end{remark}
The respective Fourier inversion formulas on $\mathbb{G}$, for all $f\in\cS(\mathbb{G})$, are 
		\begin{align*}
			f(t, x, z)&=(2\pi)^{-d-m-1}\int_{\R\times\fv^*}e^{i\mu t}\tr\left(\pi_\lambda(x, z)^*\cF_{\mathbb{G}}(f)(\mu, \lambda)\right)\, |\lambda|^dd\lambda\, d\mu\\
			&=(2\pi)^{-d-m-1}\sum_{\alpha, \beta\in \N^d}\int_{\R\times\fv^*}\cF_{\mathbb{G}}(f)(\mu, \lambda, \alpha, \beta)e^{i\mu t}\overline{E_{\alpha\beta}^{\lambda}(x, z)}\, |\lambda|^dd\lambda\, d\mu.
		\end{align*}
		 
		\subsection{Heisenberg fan and the dual surface}\label{sec:dual_parab}
		Consider the Schr\"{o}dinger operator $i\partial_t-\Delta_{\R^n}$ in the Euclidean space. Denoting by $\cF_{\R\times\R^n}$  the Euclidean Fourier transform, we have
		\begin{align*}
			\cF_{\R\times\R^n}((i\partial_t-\Delta_{\R^n})f)(\tau, \xi)=\cF_{\R\times\R^n}f(\tau, \xi)(-\tau+|\xi|^2).
		\end{align*}
		Then the \textit{dual paraboloid} $\Sigma_{\R^n}$ for the Schr\"{o}dinger operator associated to $\Delta_{\R^n}$ is the zero locus in $\R\times\R^n$ of the symbol $-\tau+|\xi|^2$, that is the subset 
		\begin{equation}\label{eq:sigmarn}
			\Sigma_{\R^n}:=\{(\tau, \xi)\in \R\times\R^n: -\tau+|\xi|^{2}=0\}.
		\end{equation}
		\begin{remark} \label{rem:pistar}
			One can  endow $\Sigma_{\R^n}$ with the measure $d\Sigma_{\R^n}$ given by the pull-back measure of $\R^{n}$ through the projection $\pr:\Sigma_{\R^n}\to \R^{n}$ where  $\pr(\tau,\xi)=\xi$. Given $f:\R^{n}\to \R$, setting $F=\pr^{*}f:=f \circ \pr$ we have the identity 
			$$\|f\|_{L^{2}(\R^{n})}=\|F\|_{L^{2}(\Sigma_{\R^{n}},d\Sigma_{\R^{n}})}.$$ 
		\end{remark}
		To build the analogue object in the $H$-type group case, observe that the operator-valued symbol of the Schr\"{o}dinger operator $i\partial_t-\Delta_{H}$ is the operator $-\mu I +H(\lambda)$ acting on $L^2(\R^d)$, where $H(\lambda)$ is the Hermite operator defined in \eqref{eq:deltastev}. That is the identity
		\begin{equation}
			\cF_{\D}((i\partial_t-\Delta_{H})f)(\mu, \lambda)=\cF_{\D}f(\mu, \lambda)(-\mu I+H(\lambda)).
		\end{equation}
		The set \eqref{eq:sigmarn} in this context is replaced by the \emph{Heisenberg fan} defined by
		\begin{align*}
			\Sigma:=\{(\mu, \lambda)\in \widehat{\R}\times (\fv^*\mz)\mid  \mu=|\lambda| (2k+d), \ \text{for some }k\in \N\},
		\end{align*}
		that is the set of pairs $(\mu, \lambda)$ for which $\Pi(\mu,\lambda)=\ker (-\mu I+H(\lambda))$ is non-trivial. The subspace 
		$\Pi(\mu,\lambda)\subset L^2(\R^d)$ is finite-dimensional and explicitly described by
		\begin{equation}\label{eq:emulambda}
			\Pi(\mu,\lambda)=\mathrm{span}\{\Phi^{\lambda}_{\alpha}\mid \alpha\in \N^{n}\ \text{s.t.}\ \mu=|\lambda| (2|\alpha|+d)\}.
		\end{equation}
		Notice that  we can decompose
		$$\Sigma=\bigcup_{k\in \N} \Sigma _{k},\qquad \Sigma_{k}:=\{(\mu, \lambda)\in \widehat{\R}\times (\fv^*\mz)\mid  \mu=|\lambda| (2k+d)\}.$$

		\begin{remark}
			We may think of any element $\Theta\in \cS(\widehat{\mathbb G})$, defined analogously as in \eqref{eq:schwartz_space}, either as a family of operators $\Theta=\{\Theta(\mu,\lambda): (\mu,\lambda)\in\widehat{\R}\times  \fv^*\mz\}$ parametrized on $\widehat{\R}\times(\fv^*\mz)$ or as a scalar valued function  $\Theta(\mu,\lambda, \alpha, \beta)$ defined on $\R \times (\fv^*\mz)\times\N^{2d}$.
		\end{remark}

		We now introduce $ d\Sigma_k$ and $ d\Sigma$ as continuous linear functionals $\cS(\widehat{\mathbb G})\to \C$.
		
		\begin{align}\label{eq:dSigmak}
			\langle d\Sigma_k, \Theta\rangle
 			=\frac{1}{(2\pi)^{d+m}}\sum_{|\alpha|=k}\int_{\fv^*} \Theta\left(|\lambda|(2k+d), \lambda, \alpha, \alpha\right)|\lambda|^d d\lambda.
		\end{align}
		where given $\Theta: \bDhat\to \cL_1$ we denote $\Theta\left(\mu, \lambda, \alpha, \beta\right)=\langle \Theta\left(\mu, \lambda\right)\Phi_\alpha^\lambda, \Phi_\beta^\lambda\rangle $. This implies
		\begin{align}\label{eq:dSigmak2}
 			\langle d\Sigma, \Theta\rangle
			:=
			\frac{1}{(2\pi)^{d+m}}\sum_{\substack{\alpha\in \N^d }}\int_{\fv^*} \Theta\left(|\lambda|(2|\alpha|+d), \lambda,  \alpha, \alpha\right)|\lambda|^{d} d\lambda
   		\end{align}
		
		\begin{remark}\label{r:fan-measure}
Formulas \eqref{eq:dSigmak} and \eqref{eq:dSigmak2} define locally finite operator-valued measures, acting on compactly supported continuous trace-class-valued tests by the trace pairing. The pairing with $d\Sigma$ extends to the chosen Schwartz test space as a tempered distribution. Indeed, with $D_k=\binom{k+d-1}{d-1}$, polar coordinates give
\[
 |\langle d\Sigma,\Theta\rangle|
 \leq C\left(\sum_{k\geq0}\frac{D_k}{(2k+d)^{d+m}}\right)
 \int_0^\infty \mu^{d+m-1}\sup_{\lambda\neq0}
       \|\Theta(\mu,\lambda)\|_{\mathrm{op}}\,d\mu.
\]
For $\Theta=\cF_\D f$, $f\in\cS(\D)$, the supremum decays faster than any power of $\mu$ by integration by parts in time. The series is finite. Multiplication by $\psi\in C_c^\infty((0,\infty))$ gives the cut off measure $d\Sigma_\psi:=\psi d\Sigma$ finite total trace mass. 
\end{remark}
		
		We now motivate the definition of $d\Sigma$. 
		If $\pr: \bDhat\to \fv^*\mz$ denotes the projection $(\mu, \lambda)\mapsto \lambda$ onto the second factor, then we can define the scalar measure on $\Sigma_k$ given by $
		d{\sigma_k}:=\pr|_{\Sigma_k}^*(d\sigma_{\widehat{G}})$
		where $d\sigma_{\widehat{G}}(\lambda)=(2\pi)^{-d-m}|\lambda|^d d\lambda$ is the Plancherel measure on $\widehat{G}\cong \fv^*\mz$. That is 
		\begin{equation}\label{eq:parab_measure_k_scalar}
			d\sigma_k(\mu, \lambda)=\frac{1}{(2\pi)^{d+m}}\delta_0(\mu-|\lambda|(2k+d))|\lambda|^d d\lambda d\mu
		\end{equation}  where $\delta_0$ is the Dirac delta function, and $d\mu$ is the Lebesgue measure on $\widehat{\R}$. 
		Equivalently, $d\sigma_k$ is defined so that, for  $F: \Sigma_k\subset \bDhat\to \C$, 
		\begin{align*}
			\int_{\Sigma_k}F(\mu, \lambda)d\sigma_k(\mu, \lambda)=\frac{1}{(2\pi)^{d+m}}\int_{\fv^*}F(|\lambda|(2k+d), \lambda)|\lambda|^d d\lambda.
		\end{align*}
		
		The operator-valued measure $d\Sigma_k$ we are defining on $\Sigma_k$ encodes the restriction to those eigenvalues of $H(\lambda)$ of height equal to $k$: more precisely if $P_k(\lambda): L^2(\R^d)\to \Pi(\mu, \lambda)$ is the eigenprojector for the Hermite operator $H(\lambda)$, we can represent $d\Sigma_k$ as follows
		\begin{align*}
			d\Sigma_k(\mu, \lambda)&:=P_k(\lambda)\otimes d\sigma_k(\mu, \lambda).\\
			&=\frac{1}{(2\pi)^{d+m}}P_k(\lambda)\delta_0(\mu-|\lambda|(2k+d))|\lambda|^dd\lambda d\mu 
		\end{align*}
		we may then rewrite formula \eqref{eq:dSigmak} as
		\begin{align*}
			\langle d\Sigma_k, \Theta\rangle=\int_{\bDhat}\trace{\Theta(\mu, \lambda)d\Sigma_k(\mu, \lambda)}.
		\end{align*}
		This formula requires some interpretation: Since $d\Sigma_k(\mu, \lambda)=P_k(\lambda)\otimes d\sigma_k(\mu, \lambda)$ the integrand is the concatenation of the trace-class operator $\Theta(\mu, \lambda)P_k(\lambda)$ with the measure $d\sigma_k$. Taking the trace of the operator-valued integrand produces a scalar valued measure which is then integrated over $\bDhat=\widehat{\R}\times(\fv^*\mz).$
		
		\begin{remark}
			Let $\pr: \bDhat\to \widehat{G}$ given by $(\mu, \lambda)\mapsto \lambda$ and let $\theta: \widehat{G}\to \cL(L^2(\R^d))$, . If $\Theta: \Sigma\subset \bDhat\to \cL(L^2(\R^d))$ is defined such that  $\Theta =\theta\circ\pr|_{\Sigma}$, then 
			\begin{equation}
				\|\Theta\|_{L^2(d\Sigma)}=\|\theta\|_{L^2(\widehat{G})}.
			\end{equation}
			This is the analogue property of Remark~\ref{rem:pistar} and is proved in Lemma~\ref{l:postotto2}.
		\end{remark}
		
		\smallskip

		The (inverse) Euclidean Fourier transform $\cF_{\R^n}^{-1}$ of a compactly supported Borel measure $\mu$ on $\R^n$ is defined by
		\begin{align*}
			\cF_{\R^n}^{-1}(\mu)(\xi)=\int_{\R^n}e^{ix\cdot\xi}d\mu(x).
		\end{align*}
		Given a bump function $\psi\in C^\infty_c((0, \infty), [0, 1])$ set 
		$d\Sigma_{\psi}:=\psi(\mu)d\Sigma$. 
		By analogy, we may define the group Fourier transform of the operator valued measure $d\Sigma_{\psi}$ as follows:
		\begin{align*}
			\cF_{\D}^{-1}(d\Sigma_\psi)=\int_{\R\times\fv^*}e^{i\mu t}\trace{\pi_\lambda(x, z)^*d\Sigma_\psi(\mu ,\lambda)}.
		\end{align*}
 		As described in Section~\ref{s:frs},  the group Fourier transform $\cF_{\D}$ extends by duality to a map between the spaces tempered distributions $\cS'(\D)\to \cS'(\bDhat)$.   We therefore write
		\begin{equation}\label{eq:gsigmastev}
			\kappa_{\Sigma_{\psi}}:=\cF_{\mathbb{G}}^{-1}(d\Sigma_{\psi})
		\end{equation}
		for the corresponding tempered distribution. In fact, we can say more about $\kappa_{\Sigma_\psi}$:

 		\begin{proposition}\label{p:newlab}
			The tempered distribution $\kappa_{\Sigma_\psi}$ is a bounded function and it satisfies the representation formula
			\begin{equation} \label{eq:gsigmapsi}
				\kappa_{\Sigma_\psi}(t, x, z)
				=\int_0^\infty e^{i\mu t}\kappa_\mu(x, z)\psi(\mu)d\mu,
			\end{equation}
			where $\kappa_\mu$ is the right convolution kernel of the spectral projector $\cP_\mu=\cP^{-\Delta}_\mu$, i.e. the tempered distribution for which
			\begin{equation} \label{eq:gsigmapsi2}
				\cP_\mu f(x, z)= f*_G \kappa_\mu(x, z).
			\end{equation}  
		\end{proposition}
		We stress that $\cP_\mu$ denotes the generalized spectral projector defined in Equation \eqref{eq:spec_proj_alt} for generalized eigenvalue $\mu$ of the sub-Laplacian $-\Delta_{H}$ on the $H$-type group $G$. The expression of $\kappa_\mu$ is given explicitly in Proposition \ref{def:GSlice}. We observe that \eqref{eq:gsigmapsi} follows
		comparing the expression for $\kappa_{\Sigma_\psi}$ in Proposition \ref{prop:GSigma} to that of $\kappa_\mu$ in Proposition \ref{eq:GSlice} and apply Fubini's Theorem. 
		
		If we denote the convolution on $\D$ is given by 
		\begin{equation}
			f*_\D g(t, w)=\int_{\D}f(t-s, w\cdot_G v^{-1})g(s, v)dsdv=\int_{\D}f(s, v)g(t-s, v^{-1}\cdot_G w)dsdv
		\end{equation}
		where $\cdot_G$ is the group product on $G$ we obtain combining \eqref{eq:gsigmapsi} and \eqref{eq:gsigmapsi2}  we obtain the following result.
		\begin{corollary} \label{cor:GP_link}For any $f\in \cS(\D)$, we have
			\begin{equation}
				f*_{\mathbb{G}} \kappa_{\Sigma_\psi}(t, x, z)=\int_0^\infty e^{i\mu t}\cP_\mu(f^{\mu})(x, z)\psi(\mu)d\mu,
			\end{equation} 
			where $\cP_\mu=\cP_\mu^{-\Delta}$ and where we denote
			\begin{equation}\label{eq:fourier_trans_central0}
				f^{\mu}(x,z)=\int_{\mathbb{R}}f(t,x, z)e^{-i\mu t}dt.
			\end{equation}
 		\end{corollary}
		\begin{proof} Let $w=(x, z)$ and $v=(x', z')$ represent elements of $G$. Also let $\psi_+\in C_c^\infty(\R)$ be the continuation of $\psi\in C_c^\infty((0, \infty))$ by zero to $\R$. 
			\begin{align*}
				f*_\D \kappa_{\Sigma_\psi}(t, w)
				&=
				\int_{\D}f(s, v)\kappa_{\Sigma_\psi}(t-s, v^{-1}\cdot_G w)dsdv\\
				&=
				\int_{\D}\int_\R e^{i\mu (t-s)}f(s, v)\kappa_\mu(v^{-1}\cdot_G w)\psi_+(\mu)d\mu dsdv\\
				&=
				\int_{\R}\int_{G}e^{i\mu t}f^\mu(v)\kappa_\mu(v^{-1}\cdot_Gw)\psi_+(\mu) dv d\mu\\
				&=
				\int_\R e^{i\mu t}f^\mu*_G \kappa_\mu(w)\psi_+(\mu)d\mu \\
				&=\int_\R e^{i\mu t}\cP_\mu(f^\mu)(x, z)\psi_+(\mu)d\mu \qedhere
			\end{align*}
		\end{proof}
		
		In what follows we prove the main result and the applications assuming Proposition~\ref{p:newlab}, whose details are contained in Section~\ref{s:frs}.
		
		\section{$H$-type groups: Proof of Propositions \ref{prop:TTstarintro} and \ref{prop:TTstarintro2}}\label{s:proofH}
		Since Propositions \ref{prop:TTstarintro} is a particular case of Proposition \ref{prop:TTstarintro2}, it is sufficient to prove the latter. This is the goal of the rest of the section.
		\begin{proof}[Proof of Proposition \ref{prop:TTstarintro2}]

			By Corollary \ref{cor:GP_link} we have
			\begin{equation}\label{eq:schro-kernelH} 
				f*_\mathbb{G} \kappa_{\Sigma_{\psi}}(t, x, z)=\int_0^\infty e^{i\mu t}\cP_\mu(f^{\mu})(x, z)\psi(\mu)d\mu.
			\end{equation}
			where  
			$\cP_\mu=\cP_\mu^{-\Delta}$ is the generalized spectral projector given in Equation \eqref{eq:spec-proj-Htype}  for eigenvalue $\mu$ of the (minus) sub-Laplacian $-\Delta_{H}$ of $G$. (For simplicity we remove the $H$ in the notation of the Laplacian when writing the spectral projector $P_\mu^{-\Delta}$.)
			
			We can rewrite identity \eqref{eq:schro-kernelH} as 
			\begin{equation}\label{eq:schro-kernel3}
				f*_\mathbb{G} \kappa_{\Sigma_{\psi}}(t, x, z)=\mathcal{F}_{\mu\mapsto -t}\left[\cP_\mu(f^{\mu})(x, z)\psi_+(\mu)\right]
			\end{equation}
			where $\cF_{\mu\mapsto -t}$ denotes the Euclidean Fourier transform in the $\mu$ variable evaluated at $-t$. 
			Suppose  now $(m,p)\neq (1,2)$ and
			\begin{itemize}
				\item[(i)] $1\leq q, r\leq p\leq 2$,
				\item[(ii)] $1\leq r\leq 2(m+1)/(m+3)$. 
			\end{itemize} 
			 
		Since $q\leq p\leq 2\leq p'$, we may apply the anisotropic Hausdorff-Young inequality (Lemma~\ref{l:an-haus-young}) to Equation \eqref{eq:schro-kernel3} to  obtain   
		\begin{equation}\label{eq:estimate1}
			\|f*_{\mathbb{G}} \kappa_{\Sigma_{\psi}}\|_{L^{r'}_\fv L^{q'}_tL^{p'}_\fh}
			\lesssim 
			\|\cP_\mu(f^{\mu})\psi_+(\mu)\|_{ L^{r'}_\fv L^{q}_\mu L^{p'}_\fh}.
		\end{equation}
		where $\widehat{\R}_\mu$ is the dual of $\R_t$.
		Then by an application of the Minkowski Inequality,  the term on the right hand side in \eqref{eq:estimate1} is less than or equal to 
		\begin{equation*}
			\|\cP_\mu(f^{\mu})\psi_+(\mu)\|_{L^{q}_\mu L^{p'}_\fh L^{r'}_\fv}.
		\end{equation*}
		By Theorem \ref{th:bound_spec_proiec_Htype}, there exists $C>0$ such that  
		\begin{align}\label{eq:spec-proj-estimate}
 			\|\cP_\mu(f^{\mu})\psi_+(\mu)\|_{L^{q}_\mu L^{p'}_\fh L^{r'}_\fv}
			&\leq
			C\bigg\|\|f^\mu\|_{L^p_\fh L^r_\fv}\psi_+(\mu) \mu^{m(\frac{1}{r}-\frac{1}{r'})+d(\frac{1}{p}-\frac{1}{{{p'}}})-1}\bigg\|_{L^q_\mu} \\
			&\times\left(\sum_{k=0}^\infty(2k+d)^{-m(\frac{1}{r}-\frac{1}{r'})-d(\frac{1}{p}-\frac{1}{{{p'}}})}\|\Lambda_k\|_{L^p_\fh\to L^{p'}_\fh}\right).
 		\end{align}
		whenever the right side is finite. Assume the series in \eqref{eq:spec-proj-estimate} is convergent. 
		Then, applying the H\"{o}lder Inequality on the first term on the right, for $\frac{1}{b}+\frac{1}{q'}=\frac{1}{q}$
		\begin{equation}\label{eq:holder1}
			\bigg\|\|f^\mu\|_{L^{p}_{\fh} L^{r}_{\fv}}\psi_+(\mu)\mu^{m(\frac{1}{r}-\frac{1}{r'})+d(\frac{1}{p}-\frac{1}{{{p'}}})-1}\bigg\|_{L^q_\mu}
			\leq
			\|f^\mu\|_{L^{q'}_\mu L^p_\fh L^r_\fv }\|\psi_+(\mu)\mu^{m(\frac{1}{r}-\frac{1}{r'})+d(\frac{1}{p}-\frac{1}{p'})-1}\|_{L^b_\mu}.
		\end{equation}

		The second term on the right hand side of \eqref{eq:holder1} is clearly finite. Regarding the other term, Lemma~\ref{l:an-haus-young} together with a couple applications of the Minkowski Inequality implies
		\begin{align*}
			\|f^\mu\|_{L^{q'}_\mu L^p_\fh L^r_\fv}
			&\leq 
			\|f^\mu\|_{L^p_\fh L^r_\fv L^{q'}_\mu}, \quad \text{since } q'\geq p, r,\\
			&\leq
			C_q\|f\|_{L^p_\fh L^r_\fv L^{q}_t}, \quad \text{by Hausdorff-Young and } q\leq 2, \\
			&\leq 
			C_q\|f\|_{L^r_\fv  L^q_tL^p_\fh}, \quad \text{since } p\geq r, q. \qedhere
		\end{align*}
	\end{proof}
	The result then follows once the convergence of the series in in \eqref{eq:spec-proj-estimate} is proved, which is the content of the next proposition.

	\begin{proposition}\label{tm:TT*-Htype}
		Let $m\geq 1$ and $p\in [1, 2]$ with $(m, p)\neq (1, 2)$. Given $r$ such that  $1\leq r\leq 2(m+1)/(m+3)$,  
		the folllowing series is convergent
		\begin{equation}\label{eq:series_Htype1ante}
			\sum_{k=0}^\infty\frac{\|\Lambda_k\|_{L^p\to L^{p'}}}{(2k+d)^{m(\frac{1}{r}-\frac{1}{r'})+d(\frac{1}{p}-\frac{1}{{p'}})}}.
		\end{equation}
	\end{proposition}
	\begin{proof}
		Applying Corollary~\ref{c:twisted-estimate-comb}, the series \eqref{eq:series_Htype1ante} is bounded by a constant multiple of
		\begin{equation}\label{eq:series_Htype2}
			\sum_{k=0}^\infty\frac{(2k+d)^{\varrho(p')}}{(2k+d)^{m(\frac{1}{r}-\frac{1}{r'})+d(\frac{1}{p}-\frac{1}{p'})}} .
		\end{equation}
		where $\varrho(p')$ is defined in \eqref{eq:rho}. To prove convergence in \eqref{eq:series_Htype2}, we  preliminary observe  that given the assumption on $r$ we have 
		\begin{equation}\label{eq:m_bound}
			m\left(\frac{1}{r}-\frac{1}{r'}\right)\geq \frac{2m}{m+1}\geq 1.
		\end{equation}
		with equality if and only if $m=1$. Indeed $r\leq 2(m+1)/(m+3)$ implies 
		$$\frac{1}{r}-\frac{1}{r'}=\frac{2}{r}-1\geq \frac{m+3}{m+1}-1=\frac{2}{m+1}.$$
		We now split the proof into two cases. Denote $p_*(d):=2(2d+1)/(2d+3)$.
		
		\medskip		
		\paragraph{\emph{Case 1:} Suppose $1\leq p\leq p_*(d)$. } 
		Then $2(2d+1)/(2d-1)\leq p'\leq \infty$. In this case, using \eqref{eq:rho}, we have
		\begin{equation}
			\varrho(p')=2d\left(\frac{1}{2}-\frac{1}{p'}\right)-1=d\left(\frac{1}{p}-\frac{1}{p'}\right)-1.
		\end{equation} The sum \eqref{eq:series_Htype2} becomes
		\begin{equation}
			\sum_{k=0}^\infty\frac{1}{(2k+d)^{m(\frac{1}{r}-\frac{1}{r'})+1}} 
		\end{equation}
		which converges by \eqref{eq:m_bound}. \\[-0.3cm]
		\paragraph{\emph{Case 2:} Suppose $p_*(d)\leq p\leq 2$. Then $2\leq p'\leq 2(2d+1)/(2d-1)$. }
		In this case, again using \eqref{eq:rho}, we have
		\begin{equation}
			\varrho(p')=\frac{1}{p'}-\frac{1}{2}=-\frac{1}{2}\left(\frac{1}{p}-\frac{1}{p'}\right).
		\end{equation} The sum  \eqref{eq:series_Htype2} becomes
		\begin{equation}
			\sum_{k=0}^\infty\frac{1}{(2k+d)^{m(\frac{1}{r}-\frac{1}{r'})+(d+\frac{1}{2})(\frac{1}{p}-\frac{1}{p'})}} .
		\end{equation}
		If $m=1$, and $p<2$, then $m(\frac{1}{r}-\frac{1}{r'})=1$ and  $(d+\frac{1}{2})(\frac{1}{p}-\frac{1}{p'})>0$. 
		If $m>1$  then $m(\frac{1}{r}-\frac{1}{r'})>1$ by \eqref{eq:m_bound}. Therefore, the series converges as long as   $(m, p)\neq (1, 2)$. 
	\end{proof}
	
	\begin{remark} It is interesting to observe that this argument excludes the endpoint estimate $p=2$ only when $m=1$. Notice that in the Heisenberg case, for which $m=1$, in \cite{BBG21} the endpoint estimate $p=2$ has been proved for radial data by an ad hoc argument (inspired from \cite{mullerRestrictionTheoremHeisenberg1990}).
	\end{remark}

	\section{The Fourier restriction and the $TT^*$ equivalence} \label{s:fr}
	
	As explained in the introduction, Strichartz estimates for the wave and Schr\"odinger equation can be interpreted as boundedness of the Fourier restriction and extension operators.  
	By the $TT^*$ argument, boundedness of the Fourier restriction and extension operators is equivalent to boundedness of convolution by $\kappa_\Sigma$ on the extended space $\D$.
	
	When $\R^n$ is replaced by an $H$-type group $G$, the dual paraboloid for the Schr\"{o}dinger equation is replaced by the Heisenberg fan $\Sigma$ defined in Section \ref{sec:dual_parab}. The (trace-class operator-valued) measure $d\Sigma_\psi$ defined in Section~\ref{sec:dual_parab} allows for the definition of the Hilbert space $L^2(d\Sigma_\psi)$ of square integrable families of operators on $\Sigma$. See formula~\eqref{def:lebesgue_parab}. 
	
	Finally the restriction $\Theta|_\Sigma$ of the operator family $\Theta\in \cS(\bDhat)$ to $\Sigma$ means the restriction of $(\mu, \lambda)\mapsto \Theta(\mu, \lambda)$ to $(\mu, \lambda)\in  \Sigma$ followed by the restriction of the operator $\Theta(\mu, \lambda): L^2(\R^d)\to L^2(\R^d)$ to the finite dimensional eigenspace $\Pi(\mu,\lambda) \subset L^2(\R^d)$ defined in \eqref{eq:emulambda}.

	\begin{theorem}\label{tm:TTstar} Let $p, q, r\in [1, \infty]$ with dual exponents $p', q', r'\in [1, \infty]$. The following are equivalent.
		\begin{enumerate}
			\item 	For all Schwartz functions $f\in \cS(\D)$,
			$$\|\cF_{\mathbb{G}}(f)|_{\Sigma}\|_{L^2(d\Sigma_{\psi})}\leq C_{q, p} \|f\|_{L^r_\fv L^q_tL_\fh^p(\D)}.$$
			
			\item For any $\Theta\in L^2(d\Sigma_\psi)$, 
			$$\| \cF_{\mathbb{G}}^{-1}(\Theta\,d\Sigma_\psi)\|_{L^{r'}_\fv L_t^{q'}L_\fh^{p'}(\D)}\leq C_{p, q} \|\Theta|_{\Sigma}\|_{L^2(d\Sigma_{\psi})}.$$
			\item For all $f\in \mathcal{S}(\D)$,
			\begin{align*}
				\|f*_{\mathbb{G}}\kappa_{\Sigma_{\psi}}\|_{L^{r'}_\fv L^{q'}_t L^{p'}_\fh(\D)}\leq C \|f\|_{L^r_\fv L_t^q L_\fh^p(\D)}.
			\end{align*}
		\end{enumerate}
	\end{theorem}
	Notice that Proposition \ref{prop:TTstarintro} establishes a set of admissible $(p, q, r)$ for which item (3) in the previous theorem is satisfied. The proof of this equivalence is postponed to Section \ref{s:frs} and we first establish applications.

	\section{Strichartz estimates for Schr\"{o}dinger equation: Proof of Theorem \ref{t:main2}}\label{s:schro}
	Consider the Schr\"{o}dinger Equation on an $H$-type group $G$:
	\begin{equation}\label{eq:sch-heis}
		\begin{cases}
			i\partial_tu-\Delta_{H} u=0 \\
			u|_{t=0}=u_0
		\end{cases}
	\end{equation}
	where we recall that given an orthonormal basis $\{X_{j}\}_{j=1}^{2d}$ the sub-Laplacian is defined by \begin{equation}
		\Delta_{H}=\sum_{j=1}^{2d} X_j^2
	\end{equation}
	For Schwartz initial data, taking the Fourier transform and using the fact that for any $f\in\cS(G)$
	\begin{align*}
		\cF_G(\Delta_{H} f)(\lambda, \alpha, \beta)=-|\lambda| (2|\alpha|+d)\cF_G(f)(\lambda, \alpha, \beta), 
	\end{align*}
	for all $\lambda\in \fv^*\mz$ and $\alpha,\beta\in\N^d$, 
	we obtain
	\[
	\begin{cases}
		i\partial_t\cF_G(u_t)(\lambda, \alpha, \beta)=-|\lambda|(2|\alpha|+d)\cF_G(u_t)(\lambda, \alpha, \beta) \\
		\cF_G(u_t)|_{t=0}=\cF_G(u_0).
	\end{cases}
	\]
	By integration we obtain for all $t\in \R$
	\begin{align*}
		\cF_G(u_t)(\lambda, \alpha, \beta)=e^{i|\lambda|(2|\alpha|+d)t}\cF_G(u_0)(\lambda, \alpha, \beta).
	\end{align*}
	Applying the inverse Fourier transform, we obtain an explicit expression for the solution to the Schr\"{o}dinger equation. 
	 
	\begin{align} \label{eq:postotto2}
		u_t(x, z)
		&=
		(2\pi)^{-d-m}\int_{\fv^*} \sum_{\alpha, \beta\in \N^d}e^{i|\lambda|(2|\alpha|+d)t}\cF_G(u_0)(\lambda, \alpha, \beta)\overline{E_{\alpha\beta}^{\lambda}(x, z)}| \lambda|^d d\lambda 
	\end{align}
	We can rewrite this as an integral over  $\Sigma\subset \bDhat$ as follows
	\begin{align} \label{eq:postotto3}
		u_t(x, z)
		&=
		\sum_{k\in \N}\int_{\bDhat} e^{i\mu t}\sum_{\substack{\alpha, \beta\in \N^d \\|\alpha|=k}}\cF_G(u_0)(\lambda, \alpha, \beta)\overline{E_{\alpha\beta}^{\lambda}(x, z)}d\sigma_k(\mu, \lambda)
	\end{align}
	where  $d\sigma_k$ is the scalar measure supported on $\Sigma_k$ defined in Section~\ref{sec:dual_parab}.

	We need the following lemma.
	\begin{lemma} \label{l:postotto}
		For $u_0\in \cS(G)$, we have 
		\begin{align*}
			\trace{\pi_\lambda(x, z)^*\cF_G u_0(\lambda)P_k(\lambda)}=\sum_{\substack{\alpha, \beta\in \N^d\\|\alpha|=k}}\cF_G(u_0)(\lambda, \alpha, \beta)\overline{E_{\alpha\beta}^{\lambda}(x, z)}.
		\end{align*}
	\end{lemma}
	Lemma~\ref{l:postotto}, whose proof is in Section~\ref{s:frp}, permits to write \eqref{eq:postotto2} as the inverse Fourier of some function defined on the Heisenberg fan.

	\begin{lemma}\label{l:solution-extension}
		Given $u_0\in \cS(G)$, the product $\cF_G(u_0)d\Sigma$ is a trace-class operator-valued measure on $\Sigma$. We have
		\begin{align}\label{eq:iftsigma}
			u_{t}(x, z)=\cF_\mathbb{G}^{-1}\{\cF_G(u_0)d\Sigma\}(t, x, z).
		\end{align}
		Moreover, for every Borel function $\vartheta:(0, \infty)\to \R$ the function
		\begin{equation}\label{eq:w-tau}
			w_\vartheta(t, x, z):=\sum_{k\in \N}\int_{\widehat G}e^{it\vartheta(|\lambda|(2k+d))}\trace{\pi_\lambda(x, z)^*\cF_Gu_0(\lambda)P_k(\lambda)}\,d\sigma_{\widehat G}(\lambda)
		\end{equation}
		is given by an absolutely convergent series of integrals and is bounded and continuous on $\R\times G$. If $\vartheta(\mu)=\mu$ then $w_\vartheta(x, z)=u_t(x, z)$.
	\end{lemma}
	The meaning of the right hand side in \eqref{eq:iftsigma} is explained in Section~\ref{sec:dual_parab}.
	\begin{proof}
Set $L=-\Delta_H$, $A(\lambda)=\cF_Gu_0(\lambda)$,
$E_k(\lambda)=|\lambda|(2k+d)$ and $D_k=\operatorname{rank}P_k(\lambda)$.
For an integer $N$ with $2N>d+m$, the trace norm and Cauchy--Schwarz give
\begin{align*}
 \sum_k\int_{\widehat G}\|A(\lambda)P_k(\lambda)\|_1\,d\sigma_{\widehat G}
 &\leq
 \left(\sum_k\int_{\widehat G}D_k(1+E_k(\lambda))^{-2N}\,d\sigma_{\widehat G}\right)^{1/2}
 \|(1+L)^Nu_0\|_2.
\end{align*}
The first factor is finite, since its square is a constant multiple of
\[
 \left(\sum_k\frac{D_k}{(2k+d)^{d+m}}\right)
 \int_0^\infty \mu^{d+m-1}(1+\mu)^{-2N}\,d\mu.
\]
Thus the above  trace norm of $A\,d\Sigma$ is finite, and the following inversion is absolutely convergent.

		Unraveling the definition of the equation above, we have
		\begin{align}\label{eq:sch-prop0}
			\cF_\mathbb{G}^{-1}\{\cF_G(u_0)d\Sigma\}(t, x, z)
			:=&
			\int_{\bDhat} e^{i\mu t}\trace{\pi_\lambda(x, z)^*\cF_Gu_0(\lambda)d\Sigma(\mu, \lambda)}.\\
			=&
			\sum_{k\in \N}\int_{\bDhat} e^{i\mu t}\trace{\pi_\lambda(x, z)^*\cF_Gu_0(\lambda)P_k(\lambda)}d\sigma_k(\mu, \lambda).
		\end{align}
		which in view of Lemma \ref{l:postotto} and Equation \eqref{eq:postotto3} is equal to $u_t(x, z)$. Finally, since $\pi_\lambda(x, z)$ is unitary, the integrand in \eqref{eq:w-tau} is bounded by $\|A(\lambda)P_k(\lambda)\|_1$, whose sum over $k$ and integral over $\widehat G$ is finite by the bound above, for every $\vartheta$; and it is continuous in $(t, x, z)$, since $\pi_\lambda$ is strongly continuous and $A(\lambda)P_k(\lambda)$ is trace class. Dominated convergence shows that $w_\vartheta$ is bounded and continuous, and \eqref{eq:sch-prop0} is $w_\vartheta$ with $\vartheta(\mu)=\mu$.
	\end{proof}

\subsection{Proof of the Strichartz estimate.} We set $X':=L^{r'}_\fv L^{q'}_tL^{p'}_\fh$ with $$2\leq p'\leq q'\leq \infty,\quad 2\leq r'\leq\infty,\qquad 1/r'\leq(m-1)/(2(m+1))$$ with $p'\leq r'$ and $(m, p')\neq (1, 2)$, so that $(p, q, r)$ satisfies the assumptions of Proposition~\ref{prop:TTstarintro2}. Let $\chi$ and $\Xi_j$ be as in Appendix~\ref{s:app}, and pick $\psi$ in \eqref{eq:newlab33} with $\psi=1$ on $\supp\chi$.
	
	 Suppose first that $u_0\in \cS(G)$ is spectrally supported\footnote{	
	Hulanicki's functional calculus \cite[Theorem~2.4]{Hulanicki1984} gives a Schwartz convolution kernel for every multiplier in $C_c^\infty(0,\infty)$ for the positive sub-Laplacian $L=-\Delta_H$. In particular, $\Xi_j\cS(G)\subset\cS(G)$.
} in $\supp\chi$ (cf.\ Definition~\ref{d:spectral}), i.e.,  $u_0=\mathbf 1_{\supp\chi}(L)u_0$. Then  item (2) of Lemma~\ref{l:spectral} with $\phi=\mathbf 1_{\supp\chi}$ gives $\cF_G(u_0)(\lambda)P_k(\lambda)=0$ unless $|\lambda|(2k+d)\in\supp\chi$, where $\psi=1$. Thus $\cF_G(u_0)d\Sigma=\cF_G(u_0)d\Sigma_\psi$ in \eqref{eq:iftsigma}, and the computation of Lemma~\ref{l:postotto2} gives $\|\cF_G(u_0)\|_{L^2(d\Sigma_\psi)}=\|\cF_G(u_0)\|_{L^2(\widehat G)}$. Combining Theorem \ref{tm:TTstar}, Proposition \ref{prop:TTstarintro2} and  the Plancherel Theorem for  $\cF_G$, it holds that
	\begin{equation}\label{eq:freq-loc}
		\|u\|_{L^{r'}_{\fv} L^{q'}_tL^{p'}_\fh}\leq C\|\cF_G(u_0)\|_{L^2(d\Sigma_\psi)}=C\|u_0\|_{L^2(G)}.
	\end{equation}
	 Now let $u_0=\Xi_jf$ with $f\in \cS(G)$ and $j\in \Z$ (by Hulanicki's theorem, we have $\Xi_jf\in \cS(G)$), and set $\Lambda:=2^{j}$. By Lemma~\ref{l:spectral} (1) and (3), the function
	\begin{equation}
		u_{0, \Lambda}=u_0\circ\delta_{\Lambda}^{-1},\qquad \delta_{\Lambda}(x,z)=(\Lambda x,\Lambda^{2}z).
	\end{equation}
	 equals $\Xi_0(f\circ \delta_\Lambda^{-1})$ and is spectrally supported in $\supp\chi$. If $u_\Lambda$ is the solution to the Schr\"{o}dinger equation \eqref{eq:sch-heis} with initial data $u_{0, \Lambda}$ in place of $u_0$, then 
	\begin{equation}
		u_\Lambda(t, x, z)=u(\Lambda^{-2}t, \Lambda^{-1}x, \Lambda^{-2}z). 
	\end{equation}
	Observe the rescaling properties
	\begin{equation}
		\|u_\Lambda\|_{L^{r'}_\fv L^{q'}_t L^{p'}_\fh}=\Lambda^{\frac{2m}{r'}+\frac{2}{q'}+\frac{2d}{p'}}\|u\|_{L^{r'}_\fv L^{q'}_t L^{p'}_\fh}, \qquad 
		\|u_{0, \Lambda}\|_{L^2(G)}=\Lambda^{d+m}\|u_0\|_{L^2(G)}.
	\end{equation}
	Set $\sigma':=\frac{Q}{2}-\frac{2m}{r'}-\frac{2}{q'}-\frac{2d}{p'}$. Thus
	\begin{equation}\label{eq:scaled-schro}
		\|u\|_{L^{r'}_\fv L^{q'}_t L^{p'}_\fh}\leq C\Lambda^{d+m-\frac{2m}{r'}-\frac{2}{q'}-\frac{2d}{p'}}\|u_0\|_{L^2(G)}= C\Lambda^{\sigma'}\|u_0\|_{L^2(G)}.
	\end{equation}
	
	Denote  $\cU(t)$ the Schr\"{o}dinger propogator so that $u=\cU(\cdot)u_0$. With $\Lambda=2^{j}$,  \eqref{eq:scaled-schro} reads
	\begin{equation}\label{eq:dyadic-schro}
		\|\cU(\cdot)\Xi_jf\|_{X'}\leq C2^{j\sigma'}\|\Xi_jf\|_{L^2(G)}, \qquad j\in \Z,
	\end{equation}
	with $C$ independent of $j$ and $f\in S(G)$. The dyadic pieces are recombined as follows.

	\begin{lemma}\label{l:recombination}
		Let $2\leq p, q<\infty$, $2\leq r\leq \infty$, $A\geq 1$, and let $(\theta_j)_{j\in \Z}$ be positive numbers with $\theta_{j+1}\geq 2\theta_j$. There is $C>0$, depending only on $p$, $q$ and $A$, such that for every finite set $J\subset\Z$ and every family $(u_j)_{j\in J}$ of bounded continuous functions on $\R\times G$ whose Fourier transforms in $t$ are supported in $\{\tau\in \R: A^{-1}\theta_j\leq |\tau|\leq A\theta_j\}$, one has
				\begin{equation}\label{eq:recombination}
			\Big\|\sum_{j\in J} u_j\Big\|_{L^r_\fv L^q_tL^p_\fh}\leq C\Big(\sum_{j\in J}\|u_j\|^2_{L^r_\fv L^q_tL^p_\fh}\Big)^{1/2}.
		\end{equation}
	\end{lemma}

\begin{proof}
Set $X=L_v^rL_t^qL_h^p$, $Y=L_t^qL_h^p$, and
$F=\sum_{j\in J}u_j$. We may assume that the right-hand side of \eqref{eq:recombination} is finite.
Choose a fixed $\varphi\in C_c^\infty(\mathbb R\setminus\{0\})$ such that
\[
 \operatorname{supp}\varphi\subset\{1/2\leq|\tau|\leq2\},
 \qquad
 \sum_{k\in\mathbb Z}\varphi(2^{-k}\tau)=1\quad(\tau\ne0),
\]
Let $\Delta_k^t$ be the temporal Fourier multiplier with symbol
$\varphi(2^{-k}\tau)$, namely
\[
\mathcal F_t(\Delta_k^t H)(\tau,x)
= \varphi(2^{-k}\tau)\,\mathcal F_t H(\tau,x).
\]
The mixed-norm Littlewood--Paley inequality
(see for instance \cite[p.~633, (7)]{BM22}) applied in the temporal variable with values
in $L_h^p$, states that
\[
 \|H\|_Y\leq C_{p,q}
 \bigg\|\bigg(\sum_{k\in\mathbb Z}|\Delta_k^tH|^2\bigg)^{1/2}
 \bigg\|_Y.
\]
Although initially stated for Schwartz functions, this inequality also
holds for $H\in Y$ with compact temporal Fourier support away from zero.\footnote{Indeed, approximate $H$ in $Y$ by Schwartz functions and apply a fixed
smooth temporal Fourier cutoff equal to one near its spectrum and
supported away from zero. The resulting approximants have only
finitely many nonzero blocks, in a fixed set of indices, so Young's
inequality allows passage to the limit.}

For almost every $z$, the function $H_{z}=F(\cdot,\cdot,z)$ satisfies these
assumptions, since $J$ is finite. Taking the $L_z^r$ norm and applying
Minkowski's inequality, using $p,q,r\geq2$, gives
\begin{equation}\label{eq:na2}
\big\|\sum_{j\in J}u_j\big\|_{X}= \|F\|_X\leq C_{p,q}
 \bigg\|\bigg(\sum_{k\in \Z}|\Delta_k^tF|^2\bigg)^{1/2}\bigg\|_X
 \leq C_{p,q}\bigg(\sum_{k\in \Z}\|\Delta_k^tF\|_X^2\bigg)^{1/2}.
\end{equation}
Define now the set
\[ 
 I_k=\{j\in J:2^{k-1}/A\leq\theta_j\leq A2^{k+1}\}.
\]
The assumption on the support of the Fourier transform of $u_{j}$ implies that
\begin{equation}\label{eq:na7}
\sum_{j\in J}\Delta_k^tu_j=\sum_{j\in I_k}\Delta_k^tu_j.
\end{equation}
Moreover, one can check that for $N_A=\lceil2\log_2A\rceil+4$ one has
\begin{equation}\label{eq:na}
 \#I_k\leq N_A,
 \qquad
 \#\{k\in\mathbb Z:j\in I_k\}\leq N_A.
\end{equation}
The first bound in \eqref{eq:na} follows from 
$\log_2\theta_{j+1}-\log_{2}\theta_{j}\geq 1$, while the second counts the integers $k$ in an
interval of length $2+2\log_2A$.

The convolution kernel of $\Delta_k^t$ is
$2^k\check\varphi(2^kt)$; hence Young's inequality in $t$ gives
\begin{equation}\label{eq:nay}
\|\Delta_k^tu_j\|_X\leq C_\varphi\|u_j\|_X
\end{equation} uniformly in $k$, where
$C_\varphi=\|\check\varphi\|_{L^1}$.
Consequently, using \eqref{eq:na7}-\eqref{eq:nay}, Cauchy--Schwarz and the two bounds in \eqref{eq:na} yield (recall that $F=\sum_{j\in J}u_j$)
 \begin{align}\label{eq:na9}
 \sum_{k\in \Z}\|\Delta_k^tF\|_X^2
 &\leq C_\varphi^2\sum_{k\in \Z}
       \bigg(\sum_{j\in I_k}\|u_j\|_X\bigg)^2\\
 &\leq C_\varphi^2N_A\sum_{k\in \Z}\sum_{j\in I_k}\|u_j\|_X^2\leq C_\varphi^2N_A^2\sum_{j\in J}\|u_j\|_X^2.
 \end{align}
Combining the estimates \eqref{eq:na2} and \eqref{eq:na9} proves the claim. All constants are independent
of $r$, and the argument also applies when $r=\infty$.
\end{proof}

	\section{Strichartz estimates for wave equation: Proof of Theorem \ref{t:main3}}\label{s:wave}
	Consider the wave equation on an $H$-type group $G$
	\begin{equation}\label{eq:wave-homogeneous}
		\begin{cases}
			\partial^{2}_tu-\Delta_{H} u=0 \\
			(u, \partial_{t}u)|_{t=0}=(u_0,v_{0})\ 
		\end{cases}
	\end{equation}
	First take $u_0,v_0$ finite-band Schwartz  functions $\cS_0(G)$ defined in Appendix~\ref{s:app}. Taking the Fourier transform, and using the fact that for any $f\in\cS(G)$
	\begin{align*}
		\cF_G(\Delta_{H} f)(\lambda, \alpha, \beta)=-|\lambda| (2|\alpha|+d)\cF_G(f)(\lambda, \alpha, \beta), 
	\end{align*}
	for all $\lambda\in \fv^*\mz$ and $\alpha,\beta\in\N^d$, 
	we obtain
	\[
	\begin{cases}
		\partial^{2}_t\cF_G(u)(\lambda, \alpha, \beta)=-|\lambda|(2|\alpha|+d)\cF_G(u)(\lambda, \alpha, \beta) \\
		(\cF_G(u),\cF_G(\partial_{t}u))|_{t=0}=(\cF_G(u_0),\cF_G(v_0)).
	\end{cases}
	\]
	By integration we obtain for all $t\in \R$
	\begin{align*}
		\cF_G(u_t)(\lambda, \alpha, \beta)=\sum_{\pm}e^{\pm it \sqrt{|\lambda|(2|\alpha|+d)}}\cF_G(\gamma_{\pm})(\lambda, \alpha, \beta)
	\end{align*}
	where we have defined
	\begin{align}\label{eq:gammapmwave}
		\cF_G(\gamma_{\pm})=\frac{1}{2}\left( \cF_G(u_{0}) \pm \frac{1}{i \sqrt{|\lambda|(2|\alpha|+d)}}\cF_G(v_{0})\right).
	\end{align}
	
	Then $\gamma_\pm\in\cS(G)$, since the multiplier $L^{-1/2}$ is smooth on the compact spectral support away from zero.
	Again by the Fourier inversion theorem 
	we have
	\begin{align}
		u_t(x, z)
		&=
		(2\pi)^{-d-m}\int_{\fv^*} \sum_{\substack{\alpha, \beta\in \N^d \\ \pm}}e^{\pm it \sqrt{|\lambda|(2|\alpha|+d)}}\cF_G(\gamma_{\pm})(\lambda, \alpha, \beta)\overline{E_{\alpha\beta}^{\lambda}(x, z)}| \lambda|^d d\lambda,\nonumber\\
		&=\sum_{\pm}\sum_{k\in \N}\int_{\widehat G}e^{\pm it\sqrt{|\lambda|(2k+d)}}\trace{\pi_\lambda(x, z)^*\cF_G(\gamma_\pm)(\lambda)P_k(\lambda)}\,d\sigma_{\widehat G}(\lambda),\label{eq:wave-trace}
	\end{align}
	since $\trace{\pi_\lambda(x, z)^*\theta P_k(\lambda)}=\sum_{|\alpha|=k}\sum_{\beta\in \N^d}\theta(\lambda, \alpha, \beta)\overline{E^\lambda_{\alpha\beta}(x, z)}$ for every trace-class operator $\theta$ on $L^2(\R^d)$, and $d\sigma_{\widehat G}(\lambda)=(2\pi)^{-d-m}|\lambda|^dd\lambda$.
	Define now  $\Sigma^{w}=\cup_{\pm} \Sigma^{\pm}$ by
	\begin{align*}
		\Sigma^{\pm}&:=\{(\mu, \lambda)\in \widehat{\R}\times (\fv^*\mz)\mid  \mu^{2}=|\lambda| (2k+d), \,\pm \mu>0 \ \text{for some }k\in \N\},
	\end{align*}
	and the measures $d\Sigma^{\pm}$ are defined in Section~\ref{s:wavedisint}.
	We can rewrite this as an integral over  $\Sigma^w\subset \bDhat$ as follows
	\begin{align} \label{eq:wave-fan-solution}
		u_t(x, z)
		&=
		\sum_{\substack{k\in \N\\ \pm}}\int_{\bDhat} e^{i\mu t}\sum_{\substack{\alpha, \beta\in \N^d \\|\alpha|=k}}\cF_G(\gamma_{\pm})(\lambda, \alpha, \beta)\overline{E_{\alpha\beta}^{\lambda}(x, z)}d\sigma_k^\pm(\mu, \lambda)
	\end{align}
	where $d\sigma_k^\pm:=(2\pi)^{-d-m}\delta_0(\mu\mp\sqrt{|\lambda|(2k+d)})|\lambda|^d\,d\lambda\,d\mu$ is the pullback of the Plancherel measure to $\Sigma_k^\pm$, as defined in Section~\ref{s:wavedisint}.
	
	Similarly to the previous case,  $\cF_G(\gamma_{\pm})d\Sigma^{\pm}$ is a trace-class operator valued measure. 
	Using again Lemma~\ref{l:postotto}
we obtain
\begin{align}\label{eq:sch-propw}
	u_t(x, z)=\sum_{\pm}\int_{\bDhat} e^{i\mu t}\trace{\pi_\lambda(x, z)^*\cF_G(\gamma_\pm)(\lambda)d\Sigma^{\pm}(\mu, \lambda)}.
\end{align}
Consequently, we can interpret Equation \eqref{eq:sch-propw} for the solution to the wave equation as the inverse Fourier transform in $\mathbb{G}$ of  the trace-class operator-valued measure $\cF_G(\gamma_{\pm})d\Sigma^{\pm}$.
\begin{align*}
	u_{t}(x, z)=\sum_{\pm}\cF_\mathbb{G}^{-1}\{\cF_G(\gamma_{\pm})d\Sigma^{\pm}\}(t, x, z).
\end{align*}

\begin{remark} Theorem~\ref{tm:TTstar} is stated for the specific set $\Sigma$ involved in the proof for the Schrodinger equation. One can easily check that the same equivalence holds true for  $\Sigma^{\pm}$. The corrisponding kernel $\kappa_{\Sigma^{\pm}}$ satisfies an analogous integral decomposition in terms of the spectral projections of the sub-Laplacians, whose details are in Section~\ref{s:wavedisint}.
\end{remark}

Let $p', q', r'$, $X'$, $\chi$ and $\Xi_j$ be as in Section~\ref{s:schro}, and choose $\psi\in C_c^\infty((0,\infty),[0,1])$ with $\psi(\nu^{1/2})=1$ for $\nu\in \supp\chi$. On the signed wave surfaces we use $d\Sigma_\psi^\pm=\psi(|\mu|)\,d\Sigma^\pm$. Suppose first that $u_0, v_0\in \cS(G)$ are spectrally supported in $\supp\chi$ (Definition~\ref{d:spectral}). Then so are $\gamma_\pm$, and as in Section~\ref{s:schro} the cutoff $\psi(|\mu|)$ does not affect $\cF_G(\gamma_\pm)d\Sigma^\pm$.
Using the observations of Section~\ref{s:wavedisint}, we obtain
\begin{equation}\label{eq:wave-freq-loc}
	\|u\|_{L^{r'}_{\fv} L^{q'}_tL^{p'}_\fh}\leq C\sum_{\pm}\|\cF_G(\gamma_{\pm})\|_{L^2(d\Sigma_\psi^{\pm})}=C\sum_{\pm}\|\gamma_{\pm}\|_{L^2(G)}.
\end{equation}
Now let $u_0=\Xi_jf$ and $v_0=\Xi_jg$ with $f, g\in \cS(G)$ and $j\in \Z$, and set $\Lambda:=2^{j}$. Then $\gamma_\pm$ are spectrally supported in $2^{2j}\supp\chi$, since the spectral multiplier $(-\Delta_H)^{-1/2}$ preserves this support. Since $\delta_\Lambda^{-1}=\delta_{\Lambda^{-1}}$, Lemma~\ref{l:spectral} (3) therefore shows that $\gamma_{\pm,\Lambda}:=\gamma_\pm\circ\delta_\Lambda^{-1}$ are spectrally supported in $\supp\chi$. Define the function
\begin{equation}
	u_{0, \Lambda}=u_0\circ\delta_{\Lambda}^{-1}.
\end{equation}
If $u_\Lambda$ is the solution to the wave equation \eqref{eq:wave-homogeneous} with initial data $(u_{0, \Lambda}, \Lambda^{-1}v_0\circ\delta_\Lambda^{-1})$ in place of $(u_0, v_0)$, for which \eqref{eq:gammapmwave} gives $\gamma_{\pm, \Lambda}$ in place of $\gamma_\pm$, then the correct rescaling (notice the difference with respect to the Schr\"{o}dinger case)
\begin{equation}
	u_\Lambda(t, x, z)=u(\Lambda^{-1}t, \Lambda^{-1}x, \Lambda^{-2}z). 
\end{equation}
Observe that 
\begin{equation}
	\|u_\Lambda\|_{L^{r'}_\fv L^{q'}_t L^{p'}_\fh}=\Lambda^{\frac{2m}{r'}+\frac{1}{q'}+\frac{2d}{p'}}\|u\|_{L^{r'}_\fv L^{q'}_t L^{p'}_\fh}
\end{equation}
and 
\begin{equation}
	\|\gamma_{\pm, \Lambda}\|_{L^2(G)}=\Lambda^{d+m}\|\gamma_{\pm}\|_{L^2(G)}.
\end{equation}
Thus recalling $Q=2d+2m$,
\begin{equation}\label{eq:scaled-wave}
	\|u\|_{L^{r'}_\fv L^{q'}_t L^{p'}_\fh}\leq C\Lambda^{\frac{Q}{2}-\frac{2m}{r'}-\frac{1}{q'}-\frac{2d}{p'}}\sum_{\pm}\|\gamma_{\pm}\|_{L^2(G)}.
\end{equation}

Denote by $\cW(t)(f, g)$ the solution of the wave equation with data $(u_0, v_0)=(f, g)$, and set $\sigma':=\frac{Q}{2}-1-\frac{2m}{r'}-\frac{1}{q'}-\frac{2d}{p'}$. By \eqref{eq:gammapmwave}, $\gamma_\pm=\frac12\big(\Xi_jf\pm (i\sqrt{-\Delta_H})^{-1}\Xi_jg\big)$, and on $2^{2j}\supp\chi$ we have $\Lambda\leq 2\mu^{1/2}$ and $\Lambda\mu^{-1/2}\leq 2$, so that $$\Lambda\sum_\pm\|\gamma_\pm\|_{L^2(G)}\leq 2\big(\|\Xi_j(-\Delta_H)^{1/2}f\|_{L^2(G)}+\|\Xi_jg\|_{L^2(G)}\big)$$ by the spectral theorem. With $\Lambda=2^{j}$, \eqref{eq:scaled-wave} therefore reads
\begin{equation}\label{eq:dyadic-wave}
	\|\cW(\cdot)(\Xi_jf, \Xi_jg)\|_{X'}\leq C2^{j\sigma'}\big(\|\Xi_j(-\Delta_H)^{1/2}f\|_{L^2(G)}+\|\Xi_jg\|_{L^2(G)}\big), \qquad j\in \Z.
\end{equation}
By \eqref{eq:wave-trace}, $\cW(t)(\Xi_jf, \Xi_jg)$ is the sum of the functions \eqref{eq:w-tau} with $u_0=\gamma_\pm\in \cS(G)$ and $\vartheta(\mu)=\pm\mu^{1/2}$. Hence Lemma~\ref{l:solution-extension} shows that $\cW(\cdot)(\Xi_jf, \Xi_jg)$ is bounded and continuous on $\R\times G$. By \eqref{eq:wave-trace} and Lemma~\ref{l:spectral} (2) with $\phi=\chi(2^{-2j}\,\cdot)$, for every $(x, z)\in G$ the function $t\mapsto \cW(t)(\Xi_jf, \Xi_jg)(x, z)$ has its Fourier transform in $t$ supported in $\{\pm\sqrt{\mu}:\mu\in 2^{2j}\supp\chi\}\subset\{\tau: 2^{j-1}\leq |\tau|\leq 2^{j+1}\}$. Suppose $p', q'<\infty$. Lemma~\ref{l:recombination} (applied with $(p', q', r')$ in place of $(p, q, r)$, so that its space $X$ is our $X'$) with $\theta_j=2^{j}$ and $A=2$, \eqref{eq:dyadic-wave} and \eqref{eq:LP-sobolev}, together with $\|(-\Delta_H)^{1/2}f\|_{\dot H^{\sigma'}(G)}=\|f\|_{\dot H^{\sigma'+1}(G)}$ (spectral theorem), bound the partial sums $$\Big\|\sum_{|j|\leq N}\cW(\cdot)(\Xi_jf, \Xi_jg)\Big\|_{X'}\leq C(\|f\|_{\dot H^{\sigma'+1}(G)}+\|g\|_{\dot H^{\sigma'}(G)})$$   uniformly in $N$. Set $S_N:=\sum_{|j|\leq N}\Xi_j$. For every $T>0$, the spectral theorem and $|\sin(t\sqrt{\mu})|/\sqrt{\mu}\leq |t|$ give
\[
 \sup_{|t|\leq T}\|\cW(t)(S_Nf,S_Ng)-\cW(t)(f,g)\|_2
 \leq \|(S_N-I)f\|_2+T\|(S_N-I)g\|_2\longrightarrow0.
\]
Thus a diagonal subsequence converges almost everywhere on $\R\times G$, and Fatou's lemma gives the following estimate.
\begin{equation}\label{eq:wave-homog}
	\|\cW(\cdot)(f, g)\|_{X'}\leq C\big(\|f\|_{\dot H^{\sigma'+1}(G)}+\|g\|_{\dot H^{\sigma'}(G)}\big)
\end{equation}
for $f, g\in S(G)$ and for all $2\leq p'\leq q'<\infty$ and $2\leq r'\leq\infty$ and $1/r'\leq(m-1)/(2(m+1))$ with $p'\leq r'$, $(m, p')\neq (1, 2)$ and $\sigma'\geq 0$. For $f\in \cS_0(G)\subset\cS(G)$ 
$$\|f\|_{\dot H^{\sigma'+1}(G)}\leq \|L^{1/2}f\|_{H^{\sigma'}(G)}\leq C\|\nabla_Hf\|_{H^{\sigma'}(G)}.$$ The first inequality holds since $\mu^{\sigma'}\leq (1+\mu)^{\sigma'}$, and the second follows from \cite[Proposition (4.1), Theorem (4.9) and the proof of Theorem (4.10)]{follandSubellipticEstimatesFunction1975} which give $L^{1/2}f=-\sum_{j}(X_jf)*K_j$ with kernels $K_j$ of type $0$, and such kernels act boundedly on $H^{\sigma'}(G)$. Hence, for $f, g\in\cS(G)$,
\begin{equation}\label{eq:wave-homog1}
\|\cW(\cdot)(f, g)\|_{X'}\leq C\big(\|\nabla_Hf\|_{H^{\sigma'}(G)}+\|g\|_{H^{\sigma'}(G)}\big).
\end{equation}
Since also $\|\nabla_Hf\|_{H^{\sigma'}(G)}\leq C\|f\|_{H^{\sigma'+1}(G)}$ \cite[Theorem (4.10)]{follandSubellipticEstimatesFunction1975}, the right-hand side of \eqref{eq:wave-homog1} is continuous in $(f, g)\in H^{\sigma'+1}(G)\times H^{\sigma'}(G)$. By density of $\cS(G)$ in $H^{\sigma'+1}(G)\times H^{\sigma'}(G)$ (Appendix~\ref{s:app}), together with the same local-in-time $L^2$ continuity, this estimate extends to $f\in H^{\sigma'+1}(G)$ and $g\in H^{\sigma'}(G)$. This is the homogeneous part of Theorem~\ref{t:main3}. The inhomogeneous part is proved in Appendix~\ref{s:appino}. 

If $q'=\infty$ or $p'=\infty$, for $f, g\in \cS_0(G)$ the triangle inequality in $X'$, \eqref{eq:dyadic-wave} and $\|\Xi_j(-\Delta_H)^{1/2}f\|_{L^2(G)}\leq 2^{j+1}\|\Xi_jf\|_{L^2(G)}$ give instead
\begin{equation}\label{eq:wave-besov}
	\|\cW(\cdot)(f, g)\|_{X'}\leq C\big(\|f\|_{\dot B^{\sigma'+1}_{2,1}(G)}+\|g\|_{\dot B^{\sigma'}_{2,1}(G)}\big).
\end{equation}
The endpoint solution map extends by continuity from $\cS_0(G)$ to the homogeneous Besov completions specified in Appendix~\ref{s:app}.

\section{The Fourier restriction: Proof of Theorem \ref{tm:TTstar}}\label{s:frs}\label{s:frp}
In what follows $G$ is an $H$-type group and $\Delta_{H}$ denotes the corresponding sub-Laplacian. For simplicity in this section we denote $\Delta_{H}$ simply by $\Delta$.

\subsection{Convolution Kernel of the Spectral Projector}
We start with some considerations on the spectral projectors. We report here a result from \cite{casarinoRestrictionTheoremMetivier2013}, specialized to the case of $H$-type groups.

Recall that, for $\rho>0$, $\Lambda_k^\rho$ is the orthogonal projector onto the eigenspace of the Heisenberg twisted Laplacian $\Delta_H^\rho$ with eigenvalue $-\rho(2k+d)$. Choose $T_{\rho\omega}=T_\omega$ for $\rho>0$ and $\omega\in S$. Define the corresponding projector on $\fh$ by
\begin{equation}\label{eq:directional-projector}
 \Lambda_{k,\omega}^{\rho}g
 :=\big[\Lambda_k^\rho(g\circ T_\omega)\big]\circ T_\omega^{-1}.
\end{equation}
Since $T_\omega$ is orthogonal and \eqref{eq:diagJ} identifies the twisted Laplacians,
\[
 \Delta_H^{\rho\omega}\Lambda_{k,\omega}^{\rho}
 =-\rho(2k+d)\Lambda_{k,\omega}^{\rho},\qquad
 \|\Lambda_{k,\omega}^{\rho}\|_{L^p\to L^{p_1}}
 =\|\Lambda_k^\rho\|_{L^p\to L^{p_1}}.
\]

\begin{theorem}\label{th:spectral-resolution}\cite[Theorem 4.8]{casarinoRestrictionTheoremMetivier2013}
	Let $G$ be an H-type group. If $f\in \cS(G)$, then
	\begin{equation}
		f(x, z)=\int_0^\infty \cP_\mu f(x, z)d\mu
	\end{equation}
	where $\cP_\mu=\cP_\mu^{-\Delta}$ is given by any of the following equivalent formulas
	\begin{equation}\label{eq:spec-proj-Htype}
		\cP_\mu f(x, z)= 
		\frac{\mu^{m-1}}{(2\pi)^m}\sum_{k=0}^\infty\frac{1}{(2k+d)^{m}}\int_Se^{i\mu_k\omega(z)}\Lambda_{k,\omega}^{\mu_k} f^{\mu_k\omega}(x)d\sigma_S(\omega) 
	\end{equation}
	\begin{align}\label{eq:spec_proj_alt}
		\cP_\mu f(x, z)=\frac{\mu^{d+m-1} 	}{(2\pi)^{d+m}}\sum_{k\in \N}\frac{1}{(2k+d)^{d+m}}\int_{S} f*e_k^{\mu_k\omega}(x, z) d\sigma_S(\omega)
	\end{align}
	where $\mu_k:=\mu/(2k+d)$ and  $d\sigma_{S}$ is the measure induced on the $m-1$ unit sphere $S\subset \fv^*$ from the metric on $\fv$ used in the definition of the $H$-type group $G$.  Moreover
	\begin{equation}
		-\Delta (\cP_\mu f)=\mu\, \cP_\mu f.
	\end{equation}
\end{theorem}

\begin{remark}
	Note that, compared with the notation in \cite[Theore 4.8]{casarinoRestrictionTheoremMetivier2013}, some powers of $\mu$ and $(2k+d)$ have been subsumed into the projector $\Lambda_k^{\mu_k}$.
\end{remark}

Let $\cP_\mu=\cP_\mu^{-\Delta}:  \cS(G)\to C^\infty(G)$ be the spectral projector defined in \eqref{eq:spec_proj_alt} corresponding to the spectral value $\mu$ of $-\Delta_H$ on $G$. 
\begin{definition}\label{d:spectral-kernel}
	Denote by $\kappa_\mu\in \cS'(G)$ the right convolution kernel of the spectral projector $\cP_\mu$, i.e., the tempered distribution for which
	\begin{equation}
		\cP_\mu f(x, z)= f*_G \kappa_\mu(x, z),
	\end{equation}
	where convolution is in the sense of distributions.
\end{definition}
The existence of the right convolution kernel $\kappa_\mu\in \cS'(G)$ is guaranteed by the  Schwartz kernel theorem and the left-invariance of $\cP_\mu$. 

The following is an analog of \cite[Proposition~4.2]{BBG21}. 
 \begin{proposition}\label{def:GSlice}
	For all $\mu>0$ the tempered distribution $\kappa_\mu$ is a bounded smooth function on $G$ and is given by
	\begin{align}\label{eq:GSlice}
		\kappa_\mu(x, z)
		&=
		\frac{\mu^{d+m-1}}{(2\pi)^{d+m}}\sum_{\substack{k\in \N }}\frac{1}{(2k+d)^{d+m}}\int_{S\subset\fv^*}e_k^{ \mu\omega/{(2k+d)}}(x, z)d\sigma_S(\omega).
	\end{align}
	where $e_k^{\rho\omega}(x, z)=\sum_{|\alpha|=k}E_{\alpha\alpha}^{\rho\omega}(x, z)$ and $d\sigma_S$ is the volume form on the unit sphere $S\subset \fv^*$. 
\end{proposition} 
\begin{remark}
	Also note here that
	$e^{\rho\omega}_k(x, z)=e_k^{\rho}(T_\omega^{-1}x, \omega\cdot z)$, where $T_{\lambda}$ denotes the linear map diagonalizing $J_{\lambda}$ as defined in \eqref{eq:diagJ}.
\end{remark}
\begin{proof}
	The expression for $\kappa_\mu$ with convergence in $\cS'(G)$ in the proposition is a direct consequence of the definition of $\cP_\mu$ in Equation \eqref{eq:spec_proj_alt}. 
	To prove that $\kappa_\mu$ is a bounded smooth function, first note that by the Cauchy Schwartz Inequality, 
	\begin{align*}
		|E_{\alpha\alpha}^\lambda(x, z)|=|\langle \pi_\lambda(x, z)\Phi_\alpha^\lambda, \Phi_\alpha^\lambda\rangle|\leq 1
	\end{align*}
	so
	\begin{align*}
		|e_k^\lambda(x, z)|\leq \sum_{|\alpha|=k}1=\frac{(k+d-1)!}{k!(d-1)!}.
	\end{align*}
	Therefore
	\begin{align*}
		|\kappa_{\mu}(x, z)|\leq \frac{\textrm{vol}(S)\mu^{d+m-1}}{(2\pi)^{d+m}(d-1)!}\sum_{k\in \N}\frac{(k+d-1)!}{k!(2k+d)^{d+m}}<\infty,
	\end{align*}
	whence the series in the definition $\kappa_{\mu}$ converges uniformly to a bounded continuous function on $G$. Each summand is an eigenfunction of $-\Delta_H$ with eigenvalue $\mu$. Uniform convergence implies convergence in $\cS'(G)$, so continuity of $-\Delta_H-\mu$ on $\cS'(G)$ gives
\[
 (-\Delta_H-\mu)\kappa_\mu=0.
\]
By H\"ormander's theorem \cite[Theorem~1.1]{hormanderHypoellipticSecondOrder1967}, the operator $-\Delta_H-\mu$ is hypoelliptic. Hence $\kappa_\mu\in C^\infty(G)$.
\end{proof}
 The goal of this section is to prove Theorem \ref{tm:TTstar}.

\subsection{Extended Space}
Let $\mathbb{G}=\R\times G$ be the extended space  and $\widehat{\mathbb{G}}=\mathbb{R}\times (\mathfrak{v}^*\setminus \{0\})$. 
Here $\mathcal{L}_2$ denotes the space of Hilbert--Schmidt operators on $L^2(\R^d)$ with the norm
$$\|A\|_{\cL_2}:=\left(\trace{A^*A}\right)^{\frac{1}{2}}.$$

\begin{definition}\label{d:dual-L2}
	Let $L^2(\widehat{\mathbb{G}})=L^2(\widehat{\mathbb{G}}, \mathcal{L}_2)$ be the space of all measurable families of operators $\Theta(\mu, \lambda)\in \cL(L^2(\R^d))$ for almost all $(\mu, \lambda)\in \widehat{\mathbb{G}}$ such that 
 	\begin{equation}
		\|\Theta\|_{L^2(\bDhat)}^2:=\frac{1}{2\pi}\int_{-\infty}^\infty \int_{\mathfrak{v}^*}\|\Theta (\mu, \lambda)\|_{\cL_2}^2d\sigma_{\widehat{G}}(\lambda)d\mu<\infty.
	\end{equation}
	Notice that $L^2(\bDhat)$ is a Hilbert space with the obvious inner product. 
\end{definition}
Then the Plancherel theorem takes the form
\begin{equation}
	\langle\cF_\D(f), \cF_\D(g)\rangle_{L^2(\bDhat)}=\langle f, g\rangle_{L^2(\D)}.
\end{equation}
\begin{remark}
	As stated previously, we may think of any element $\Theta\in \cS(\widehat{\mathbb{G}})$ as a family of operators $(\mu, \lambda)\mapsto \Theta(\mu, \lambda)$ parameterized by $(\mu, \lambda)\in \widehat{\R}\times (\fv^*\mz)$ or as a scalar valued function $\Theta(\mu, \lambda, \alpha, \beta)=\langle\Theta(\mu, \lambda)\Phi_\alpha^\lambda, \Phi_\beta^\lambda\rangle$ defined on $\widehat{\R}\times (\fv^*\mz)\times \N^{2d}$. Then $L^2(\widehat{\mathbb{G}})$ is isomorphic to the Hilbert space of functions $(\mu, \lambda, \alpha, \beta)\mapsto \Theta(\mu, \lambda, \alpha, \beta)$ such that 
	\begin{align*}
		\sum_{\alpha, \beta\in \N^d}\int_{-\infty}^\infty\int_{\fv^*}|\Theta(\mu, \lambda, \alpha, \beta)|^2d\sigma_{\widehat{G}}(\lambda)d\mu<\infty.
	\end{align*}
\end{remark}
\subsection{Polar coordinates}
Recall from Section~\ref{s:fheis} that we defined the Heisenberg fan by
\begin{align*}
	\Sigma:=\{(\mu, \lambda)\in \widehat{\R}\times (\fv^*\mz)\mid  \mu=|\lambda| (2k+d), \ \text{for some }k\in \N\},
\end{align*}
Recall that  we can decompose
$$\Sigma=\bigcup_{k\in \N} \Sigma _{k},\qquad \Sigma_{k}:=\{(\mu, \lambda)\in \widehat{\R}\times (\fv^*\mz)\mid  \mu=|\lambda| (2k+d)\},$$
and defining  
\begin{align}\label{eq:steve1}
	\langle d\Sigma_k, \Theta\rangle
 	=\frac{1}{(2\pi)^{d+m}}\sum_{|\alpha|=k}\int_{\fv^*} \Theta\left(|\lambda|(2k+d), \lambda, \alpha, \alpha\right)|\lambda|^d d\lambda.
\end{align}
\begin{align}\label{eq:steve2}
 	\langle d\Sigma, \Theta\rangle
	:=
	\frac{1}{(2\pi)^{d+m}}\sum_{\substack{\alpha\in \N^d }}\int_{\fv^*} \Theta\left(|\lambda|(2|\alpha|+d), \lambda,  \alpha, \alpha\right)|\lambda|^{d} d\lambda
\end{align}

Writing \eqref{eq:steve1} in polar coordinates $\lambda=\rho\omega$ gives $\langle d\Sigma_k,\Theta\rangle$ as
\begin{align*}
	\frac{1}{(2\pi)^{d+m}}\sum_{|\alpha|=k}\int_{S\subset\fv^*} \int_0^\infty\Theta\left(\rho(2k+d), \rho\omega, \alpha, \alpha\right)\rho^{d+m-1} d\rho\, d\sigma_S(\omega).
\end{align*}
where $S\subset \fv^*$ is the unit sphere. 
The rescaling $\mu=\rho(2k+d)$ yields
\begin{align*}
	\langle d\Sigma_k, \Theta\rangle
	=
	\frac{1}{(2\pi)^{d+m}}&\sum_{|\alpha|=k}\frac{1}{(2k+d)^{d+m}}\int_{S} \int_0^\infty\Theta\left(\mu, \frac{\mu\omega}{2k+d}, \alpha, \alpha\right)\mu^{d+m-1} d\mu\, d\sigma_S(\omega).
\end{align*}
which in particular gives
\begin{align*}
	\langle d\Sigma, \Theta\rangle
	=
	\frac{1}{(2\pi)^{d+m}}&\sum_{k\in\N}\sum_{|\alpha|=k}\frac{1}{(2k+d)^{d+m}}\int_{S} \int_0^\infty\Theta\left(\mu, \frac{\mu\omega}{2k+d}, \alpha, \alpha\right)\mu^{d+m-1} d\mu\, d\sigma_S(\omega).
\end{align*}

\begin{remark}\label{r:later}
	Later, for technical considerations, it will be convenient to multiply the measure $d\Sigma$ with a cutoff function $\psi\in C_c^\infty(\R, [0, 1])$ as follows
	\begin{align*}
		\langle d\Sigma_{\psi}, \Theta\rangle
		=
		\frac{1}{(2\pi)^{d+m}}&\sum_{k\in\N}\sum_{|\alpha|=k}\frac{1}{(2k+d)^{d+m}}\int_{S} \int_0^\infty\Theta\left(\mu, \frac{\mu\omega}{2k+d}, \alpha, \alpha\right)\mu^{d+m-1} \psi(\mu)d\mu\, d\sigma_S(\omega).
	\end{align*}
	
\end{remark}
\subsection{More on the Fourier side}
Recall that $\cP_\mu=\cP_\mu^{-\Delta_H}$ is the generalized spectral density for $-\Delta_H$. Its Fourier-side identity is understood after integration in $\mu$: for $f\in\cS(G)$ and $a\in C_c^\infty((0,\infty))$, the functional calculus~\cite[Corollary~4.1.16]{FR16} gives,\footnote{In \cite{FR16} the group Fourier transform is defined with respect to $\pi(x)^*$, so by there convention the multiplier acts on the left.} for almost every $\lambda$,
\begin{align*}
 \cF_G\left(\int_0^\infty a(\mu)\cP_\mu f\,d\mu\right)(\lambda)
 &=\cF_G(f)(\lambda)a(H(\lambda))\\
 &=\sum_{k\geq0}a\bigl(|\lambda|(2k+d)\bigr)\cF_G(f)(\lambda)P_k(\lambda).
\end{align*}
The integral on the left is $a(-\Delta_H)f$ in $L^2(G)$.
For $(\mu,\lambda)\in\Sigma$, let $\cP_\mu^{H(\lambda)}:L^2(\R^d)\to\Pi(\mu,\lambda)$ be the ordinary orthogonal projector onto the $\mu$-eigenspace of $H(\lambda)$. As an consequence of the equation above, 
$$\cP_\mu^{H(\lambda)}=\sum_{\mu=|\lambda|(2k+d)} P_k(\lambda)$$ Thus, the operator-valued measure $d\Sigma$ can be written neatly as
\begin{equation}\label{e:dSigmaAlt}
	d\Sigma(\mu,\lambda)=\cP_\mu^{H(\lambda)}\,d\sigma(\mu,\lambda),
\end{equation}
where $d\sigma$ is the scalar measure on $\Sigma$ defined above.

For any bounded operator $A\in \cL(L^2(\R^d))$, let 
$$|A|:=\sqrt{A^*A}. $$

Define $L^2(d\Sigma)$ to be the set of all measurable  families of operators $\Theta(\mu, \lambda): \Pi(\mu, \lambda)\to L^2(\R^d)$ for $(\mu, \lambda)\in \Sigma\subset \widehat{\mathbb{G}}$  such that 
\begin{equation}\label{def:lebesgue_parab}
	\|\Theta\|_{L^2( d\Sigma)}^2:=\int_{\Sigma}\trace{|\Theta(\mu, \lambda)|^2 d\Sigma(\mu, \lambda)}<\infty.
\end{equation}

 \begin{remark}
	We may think of an element $\Theta\in L^2(d\Sigma)$ as a scalar valued function $(\mu, \lambda, \alpha, \beta)\mapsto \Theta(\mu, \lambda, \alpha, \beta)$ on the subset  $\Sigma\times \N^{2d}$ with $\mu=|\lambda|(2|\alpha|+d)$. Then $L^2(d\Sigma)$ is isomorphic to the Hilbert space of all such functions  such that
	\begin{align*}
		\sum_{\substack{\alpha, \beta\in\N^d }}\int_{\fv^*}|\Theta(|\lambda|(2|\alpha|+d), \lambda, \alpha, \beta)|^2 d\sigma_{\widehat{G}}(\lambda)<\infty.
	\end{align*}
\end{remark}

\begin{lemma} \label{l:postotto2}
	Let ${\rm pr}_2: \bDhat\to \widehat{G}$ given by $(\mu, \lambda)\mapsto \lambda$ and  let $\theta: \widehat{G}\to \cL(L^2(\R^d))$. If $\Theta: \Sigma\subset \bDhat\to \cL(L^2(\R^d))$ is defined by $\Theta =\theta\circ{\rm pr}_2|_{\Sigma}$, for all $k\in \N$, then 
	\begin{equation}
		\|\Theta\|_{L^2(d\Sigma)}=\|\theta\|_{L^2(\widehat{G})}.
	\end{equation}
\end{lemma}

\begin{proof}[Proof of Lemma \ref{l:postotto2}] We have  
	\begin{align*}
		\int_{\Sigma}\trace{|\Theta(\mu, \lambda)|^2 d\Sigma(\mu, \lambda)}
		&=\sum_{k=0}^\infty\int_{\Sigma_k}\trace{|\Theta(\mu, \lambda)|^2 d\Sigma_k(\mu, \lambda)}\\
		&=\sum_{k=0}^\infty\int_{\Sigma_k}\trace{|\Theta(\mu, \lambda)|^2 P_k(\lambda)}d\sigma_k(\mu, \lambda) \\
		&=\sum_{k=0}^\infty\int_{\fv^*}\trace{|\Theta(|\lambda|(2k+d), \lambda)|^2 P_k(\lambda))|}d\sigma_{\widehat{G}}(\lambda) \\
		&=\sum_{k=0}^\infty\int_{\fv^*}\trace{|\theta(\lambda)|^2 P_k(\lambda))}d\sigma_{\widehat{G}}(\lambda) \\
		&=\int_{\fv^*}\trace{|\theta(\lambda)|^2 }d\sigma_{\widehat{G}}(\lambda) \qedhere
	\end{align*}
\end{proof}

\subsection{Fourier transform $\cF_{\mathbb{G}}$ on tempered distributions}

Let $\cF^t_\D:\cS(\bDhat)\to\cS(\D)$ be the transpose of the Fourier transform $\cF_\D$ with respect to the bilinear pairing $\langle\cdot, \cdot\rangle: \cS(\D)\times\cS(\D)\to \C$ and $\langle\cdot, \cdot\rangle: \cS(\bDhat)\times\cS(\bDhat)\to \C$. That is, for all $f\in \cS(\D)$ and  $\Theta\in \cS(\bDhat),$
\begin{equation}
	\langle f, \cF^t_\D\Theta\rangle:=\langle \cF_\D f, \Theta\rangle.
\end{equation}
The group Fourier transform is defined on tempered distributions as follows:
\begin{equation}\label{eq:Fourier_temp_D}
	\cF_\D:
	\begin{cases}
		\cS'(\D)\longrightarrow \cS'(\bDhat) \\
		u\longmapsto \left(\Theta\mapsto \langle u, \cF_\D^t\Theta\rangle\right).
	\end{cases}
\end{equation}

\begin{lemma}\label{lem:Fourier_tp_D}
	For any $\Theta\in\cS(\bDhat)$, the inverse Fourier transform of $\Theta$ is related to the transpose via
	\begin{align*}
		\cF_{\mathbb{G}}^{t}\Theta=\iota^*\circ \cF^{-1}_{\mathbb{G}}\Theta
	\end{align*}
	where $\iota: \mathbb{G}\to \mathbb{G}$ is the involution
	$\iota(t, x, z)=(-t, -x, -z)$.
\end{lemma}
\begin{proof}
With the bilinear trace pairing on the dual, including the Plancherel measure,
\[
 \cF_\D^t\Theta(t,w)
 =\int_{\widehat\D}e^{-i\mu t}\tr\!\big(\pi_\lambda(w)\Theta(\mu,\lambda)\big)
 \,\frac{d\mu}{2\pi}\,d\sigma_{\widehat G}(\lambda)
 =\cF_\D^{-1}\Theta(-t,w^{-1}).
\]
This also specifies the domain and codomain of the transpose.
\end{proof}

Considering the cutoff measure $d\Sigma_{\psi}$ as in Remark~\ref{r:later}, we write in the sense of tempered distributions
\begin{align}
	\kappa_{\Sigma_\psi}:=\cF_\D^{-1}(d\Sigma_\psi), \qquad \text{or}\qquad  \cF_\D(\kappa_{\Sigma_\psi})=d\Sigma_{\psi}.
\end{align}
The presence of the cutoff is used to guarantee the boundedness of the convolution kernel.
\begin{proposition}\label{prop:GSigma}
	The tempered distribution $\kappa_{\Sigma_\psi}$ is in $L^\infty(\D)\cap C(\D)$ and is given by
	\begin{equation}
		\kappa_{\Sigma_\psi}(t, x, z)=
		\frac{1}{(2\pi)^{d+m}}\sum_{k\in \N}\frac{1}{(2k+d)^{d+m}}\int_0^\infty\int_{S} e^{i\mu t}e_k^{\mu\omega/(2k+d)}(x, z)\psi(\mu)\mu^{d+m-1}d\sigma_S(\omega)d\mu 
	\end{equation}
	where $	e_k^{\lambda}:=\sum_{|\alpha|=k}E_{\alpha\alpha}^{\lambda}$ and $d\sigma_S$ is the intrinsic surface measure on the unit sphere $S\subset \fv^*$. 
\end{proposition}
\begin{proof}
By the definition of the inverse Fourier transform of the cutoff measure,
\begin{align*}
 \kappa_{\Sigma_\psi}(t,w)
 &=(2\pi)^{-d-m}\sum_{k\in\N}\int_{\fv^*}
 e^{i|\lambda|(2k+d)t}
 \tr\!\big(\pi_\lambda(w)^*P_k(\lambda)\big)
 \psi(|\lambda|(2k+d))|\lambda|^d\,d\lambda.
\end{align*}
The series is absolutely convergent by the trace bound below. Moreover
$\tr(\pi_\lambda(w)^*P_k(\lambda))
=\overline{e_k^\lambda(w)}=e_k^{-\lambda}(w)$, since
$e_k^\lambda(x,z)=e^{i\lambda(z)}\varphi_k^{|\lambda|}(T_\lambda^{-1}x)$
and the Laguerre factor is real and radial. The substitution $\lambda\mapsto-\lambda$ gives
\begin{align}\label{eq:GSigma1}
 \kappa_{\Sigma_\psi}(t,w)
 &=(2\pi)^{-d-m}\sum_{k\in\N}\int_{\fv^*}
 e^{i|\lambda|(2k+d)t}e_k^\lambda(w)
 \psi(|\lambda|(2k+d))|\lambda|^d\,d\lambda.
\end{align}
Taking polar coordinates $\lambda=\rho\omega$ with $\rho\in (0, \infty)$ and $\omega\in S\subset \fv^*$ the unit sphere with measure $d\sigma_S$, followed by the substitution $\mu=\rho(2k+d)$ gives the expression for $\kappa_{\Sigma_\psi}$.
	
	To prove boundedness of $\kappa_{\Sigma_\psi}$, note that by the Cauchy Schwartz Inequality, 
	\begin{align*}
		|E_{\alpha\alpha}^\lambda(x, z)|=|\langle \pi_\lambda(x, z)\Phi_\alpha^\lambda, \Phi_\alpha^\lambda\rangle|\leq 1
	\end{align*}
	so
	\begin{align*}
		|e_k^\lambda(x, z)|\leq \sum_{|\alpha|=k}1=\frac{(k+d-1)!}{k!(d-1)!}.
	\end{align*}

	Therefore
	\begin{align*}
		|\kappa_{\Sigma_\psi}(t, x, z)|\leq \frac{\textrm{vol}(S)\|(\cdot)^{d+m-1}\psi\|_{L^1(0, \infty)}}{(2\pi)^{d+m}(d-1)!}\sum_{k\in \N}\frac{(k+d-1)!}{k!(2k+d)^{d+m}}<\infty,
	\end{align*}
	whence $\kappa_{\Sigma_\psi}$ is a bounded continuous function on $\D$.
\end{proof}

\subsection{Proof of Lemma~\ref{l:postotto}} We prove that 
for any $k\in \N$, 
\begin{align*}
	\sum_{\substack{\alpha, \beta\in \N^d \\ |\alpha|=k}}\cF_Gu_0(\lambda, \alpha, \beta)\overline{E_{\alpha\beta}^\lambda(x, z)}=\trace{\pi_\lambda(x, z)^*\cF_Gu_0(\lambda)P_k(\lambda)}.
\end{align*}
\begin{proof} For $k\in \N$, $(x, z)\in G$, $\lambda\in \fv^*\mz$ and $u_0\in \cS(G)$, 
	\begin{align*}
		\trace{\pi_\lambda(x, z)^*\cF_Gu_0(\lambda)P_k(\lambda)}
		&=
		\sum_{\alpha\in \N^d}\langle \pi_\lambda(x, z)^*\cF_Gu_0(\lambda)P_k(\lambda)\Phi_\alpha^\lambda, \Phi_\alpha^\lambda\rangle\\
		&=
		\sum_{\substack{\alpha\in \N^d\\ |\alpha|=k}}\langle \pi_\lambda(x, z)^*\cF_Gu_0(\lambda)\Phi_\alpha^\lambda, \Phi_\alpha^\lambda\rangle\\
		&=
		\sum_{\substack{\alpha, \beta\in \N^d\\ |\alpha|=k}}\langle \cF_Gu_0(\lambda)\Phi_\alpha^\lambda, \Phi_\beta^\lambda\rangle \langle\Phi_\beta^\lambda, \pi_\lambda(x, z)\Phi_\alpha^\lambda\rangle\\
		&=
		\sum_{\substack{\alpha, \beta\in \N^d\\ |\alpha|=k}}\cF_Gu_0(\lambda, \alpha, \beta) \overline{E_{\alpha\beta}^\lambda(x, z)}.\qedhere
	\end{align*}
\end{proof}

\subsection{Restriction and extension operators}\label{sec:Restriction/Extension}
For $f\in \cS(\D)$, the restriction operator introduced in Section~\ref{s:fr} can be written as follows
$$\cR_{\Sigma}f:=\cF_{\mathbb{G}}f|_{\Sigma}.$$

The meaning of this expression is the restriction of the map $(\mu, \lambda)\mapsto \cF_{\mathbb{G}}f(\mu, \lambda)$ to $(\mu, \lambda)\in \Sigma\subset \widehat{\mathbb{G}}$ followed by restriction of the operator $\cF_{\mathbb{G}}f(\mu, \lambda): L^2(\R^d)\to L^2(\R^d)$ to the subspace $\Pi(\mu, \lambda)\subset L^2(\R^d).$

More explicitly, we have
\begin{align*}
	\cR_{\Sigma}f(\mu, \lambda):=\cF_{\mathbb{G}}f(\mu, \lambda)\cP_\mu^{H(\lambda)}, \quad (\mu, \lambda)\in \Sigma.
\end{align*}
\begin{remark}
	The restriction operator $\cR_{\Sigma}f$ corresponds to restriction of the scalar-valued Fourier transform $(\mu, \lambda, \alpha, \beta)\mapsto \cF_{\mathbb{G}}f(\mu, \lambda, \alpha, \beta)$ to the set $$\{(\mu, \lambda, \alpha, \beta)\in \widehat{\R}\times (\fv^*\mz)\times\N^{2d}:\quad \mu=|\lambda|(2|\alpha|+d)\}.$$   
\end{remark}

{\bf Notation.} In this subsection $L^2(d\Sigma)$ abbreviates $L^2(d\Sigma_\psi)$ and $\kappa_\Sigma$ denotes the cutoff kernel. In explicit integrals, $d\sigma_k$ and $P_k\,d\sigma_k$ are the uncut measures, and the weight $\psi(\mu)$ is included exactly once. Thus $\cF_\D^{-1}(\Theta d\Sigma)$ below means $\cF_\D^{-1}(\Theta d\Sigma_\psi)$.

\smallskip

Let us also introduce the $k$-th restriction operator defined analogously for $\Sigma_k$:
$$\cR_{\Sigma_k}f:=\cF_{\mathbb{G}}f|_{\Sigma_k}.$$
In particular, 
\begin{align*}
	\cR_{\Sigma_k} f(\mu, \lambda)=\cF_{\mathbb{G}}f(\mu, \lambda)\circ P_k(\lambda), \quad (\mu, \lambda)\in \Sigma_k,
\end{align*}

	and
	\begin{align*}
		\cR_{\Sigma}f=\sum_{k\in \N}\cR_{\Sigma_k}f.
	\end{align*}

	For $\Theta\in L^2(d\Sigma)$ define the Fourier extension operator by
	\begin{equation}
		\cE_{\Sigma}\Theta(t, x, z):=\cF_\D^{-1}(\Theta d\Sigma).
	\end{equation}
	Similarly, define the $k$-th Fourier extension operator by
	\begin{equation}
		\cE_{\Sigma_{k}}\Theta(t, x, z):=\cF_\D^{-1}(\Theta d\Sigma_k).
	\end{equation}
	Explicitly,
\[
 \cE_{\Sigma_k}\Theta(t,w)=
 \int_{\Sigma_k}\psi(\mu)e^{i\mu t}
 \tr\!\big(\pi_\lambda(w)^*\Theta(\mu,\lambda)P_k(\lambda)\big)
 \,d\sigma_k(\mu,\lambda),
 \qquad
 \cE_\Sigma\Theta=\sum_{k\in\N}\cE_{\Sigma_k}\Theta.
\]

\begin{lemma}\label{lem:extention_parab}
For all $k\in\N$, $\cE_{\Sigma_k}$ is the formal adjoint of $\cR_{\Sigma_k}$ in the $L^2(d\Sigma_\psi)$ and space-time pairings: for $\Theta\in L^2(d\Sigma_\psi)$ and $f\in\cS(\D)$,
\[
 \langle\Theta,\cR_{\Sigma_k}f\rangle_{L^2(d\Sigma_\psi)}
 =\int_\D\overline{\cE_{\Sigma_k}\Theta(t,w)}\,f(t,w)\,dt\,dw.
\]
Consequently $\cE_\Sigma=\cR_\Sigma^*$ in the same sense.
\end{lemma}
\begin{proof}
Write $w=(x,z)$. Since $\Theta$ acts on the range of $P_k$ on $\Sigma_k$, cyclicity of the trace gives
\begin{align*}
 \langle\Theta,\cR_{\Sigma_k}f\rangle
 &=\int_{\Sigma_k}\psi(\mu)
   \tr\!\big(\Theta(\mu,\lambda)^*\cF_\D f(\mu,\lambda)P_k(\lambda)\big)
   \,d\sigma_k(\mu,\lambda)\\
 &=\int_\D f(t,w)\int_{\Sigma_k}\psi(\mu)e^{-i\mu t}
   \tr\!\big(\Theta(\mu,\lambda)^*\pi_\lambda(w)P_k(\lambda)\big)
   \,d\sigma_k(\mu,\lambda)\,dt\,dw\\
 &=\int_\D f(t,w)\,
   \overline{\int_{\Sigma_k}\psi(\mu)e^{i\mu t}
   \tr\!\big(\pi_\lambda(w)^*\Theta(\mu,\lambda)P_k(\lambda)\big)
   \,d\sigma_k(\mu,\lambda)}\,dt\,dw.
\end{align*}
The inner integral is $\cE_{\Sigma_k}\Theta(t,w)$. The cutoff measure $d\Sigma_\psi$ has finite trace, as shown in Proposition~\ref{prop:GSigma}. Therefore by the Cauchy--Schwarz  justifies Fubini and summation over $k$ for $f\in L^1(\D)$ therefore.
\end{proof}

\begin{proposition}\label{prop:restriction-convolution}
For $f\in\cS(\D)$,
\begin{equation}\label{eq:TT*-ker}
 \cR_\Sigma^*\cR_\Sigma f=f*_\D\kappa_\Sigma.
\end{equation}
Here $\kappa_\Sigma$ is the cutoff convolution kernel in Proposition~\ref{prop:GSigma}.
\end{proposition}
\begin{proof}
For $w,v\in G$, $\pi_\lambda(w)^*\pi_\lambda(v)
=\pi_\lambda(v^{-1}w)^*$. Thus the preceding lemma and Fubini give
\begin{align*}
 \cR_\Sigma^*\cR_\Sigma f(t,w)
 &=\int_{\R\times G}f(s,v)
 \sum_{k\in\N}\int_{\Sigma_k}\psi(\mu)e^{i\mu(t-s)}
 \tr\!\big(\pi_\lambda(w)^*\pi_\lambda(v)P_k(\lambda)\big)
 \,d\sigma_k(\mu,\lambda)\,ds\,dv\\
 &=\int_{\R\times G}f(s,v)\kappa_\Sigma(t-s,v^{-1}w)\,ds\,dv.
\end{align*}
Absolute integrability follows from $f\in L^1(\D)$ and the finite total trace of the cutoff measure. This is precisely right convolution by $\kappa_\Sigma$.
\end{proof}

\begin{proof}[Completion of the proof of Theorem~\ref{tm:TTstar}]
Write $X=L^r_\fv L^q_tL^p_\fh$ and
$X'=L^{r'}_\fv L^{q'}_tL^{p'}_\fh$.  For every locally integrable $g$ on $\D$,
\[
 \|g\|_{X'}=\sup\Big\{\Big|\int_\D\overline{g}\,f\Big| : f\in C_c^\infty(\D),\ \|f\|_X\leq 1\Big\}\in[0,\infty].
\]
This holds for all $p, q, r\in[1,\infty]$. If an exponent is infinite,  $X'$ is strictly smaller than the Banach dual of $X$, but the formula still holds. Together with Hilbert-space duality and Lemma~\ref{lem:extention_parab}, this proves $(1)\Leftrightarrow(2)$. If $(3)$ holds, then for $f\in\cS(\D)$,
\[
 \|\cR_\Sigma f\|_{L^2(d\Sigma_\psi)}^2
 =\int_\D\overline{\cE_\Sigma\cR_\Sigma f}\,f
 \leq \|f*_\D\kappa_\Sigma\|_{X'}\|f\|_X
 \leq C\|f\|_X^2,
\]
so $(1)$ follows. Conversely, $(1)$ and $(2)$, followed by Proposition~\ref{prop:restriction-convolution}, give $(3)$ by composition.
\end{proof}

\subsection{Wave Equation}\label{s:wavedisint}
	In the case of the wave equation, we consider the sets
	\begin{equation}
		\Sigma^{\pm}=\{(\mu, \lambda)\in \bDhat: |\mu|^2=|\lambda|(2k+d), \, \pm\mu>0 \, \quad k\in \N\}
	\end{equation}
	With respect to the surfaces $\Sigma^{\pm}$ (or after applying the cutoff $\psi(|\mu|)$) the map $\cR_{\Sigma^{\pm}}f=\cF_{\mathbb{G}}f|_{\Sigma^{\pm}}$  is defined by  the restriction of the map $(\mu, \lambda)\mapsto \cF_{\mathbb{G}}f(\mu, \lambda)$ to $(\mu, \lambda)\in \Sigma^{\pm}$ followed by the restriction of the operator $\cF_{\mathbb{G}}f(\mu, \lambda)\in \cL(L^2(\R^d))$ to the finite dimensional eigenspace $\Pi(|\mu|^2, \lambda)=\ker(\mu^2I-H(\lambda))\subset L^2(\R^d)$. 
	Each component $\Sigma_k^\pm\subset \Sigma^\pm$ is equipped with the scalar measure $d\sigma_k^\pm=\pr|_{\Sigma_k^\pm}^*d\sigma_{\widehat{G}}$ on $\Sigma_k^\pm$. As in the Schr\"{o}dinger case, we define
	\begin{equation}
		d\Sigma^\pm_k(\mu, \lambda):=P_k(\lambda)\otimes d\sigma_k^\pm(\mu, \lambda).
	\end{equation}
	Fix $\psi\in C_c^\infty((0,\infty),[0,1])$ and set $d\Sigma_\psi^\pm:=\psi(|\mu|)\sum_{k\in\N}d\Sigma_k^\pm$. In this subsection the adjoints of $\cR_{\Sigma^\pm}$ are taken with respect to $L^2(d\Sigma_\psi^\pm)$; we suppress $\psi$ in the restriction and kernel notation.
Following the same arguments as in Section \ref{sec:Restriction/Extension}, we obtain 
	\begin{proposition}\label{p:stev5}
		For $f\in \cS(\mathbb{G})$,  we have
		$$\cR_{\Sigma^{\pm}}^*\cR_{\Sigma^{\pm}}f=f*_{\mathbb{G}}\kappa_{\Sigma^\pm}$$
		where the convolution kernel $\kappa_{\Sigma^{\pm}}\in L^\infty(\mathbb G)\cap C(\mathbb G)$ is given by 
		\begin{align}\label{eq:k_wave}
			\kappa_{\Sigma^{\pm}}(t, x, z)=2\int_0^\infty e^{\pm i\mu t}\kappa_{\mu^2}(x, z)\psi(\mu)\mu\,d\mu.
		\end{align}
		with $\kappa_\mu$  given in Proposition \ref{def:GSlice}. 
	\end{proposition}
	\begin{proof}
		Since 
		$\cR_{\Sigma^\pm}(f)d\Sigma_\psi^\pm=\cF_\mathbb{G}(f)d\Sigma_\psi^\pm$, we have 
		\begin{align*}
			\cR_{\Sigma^\pm}^*\cR_{\Sigma^{\pm}}f
 			&=
			\cF_{\mathbb{G}}^{-1}(\cF_{\mathbb{G}}(f)d\Sigma_\psi^\pm)\\
			&=
			f*_\mathbb{G}\cF_{\mathbb{G}}^{-1}(d\Sigma_\psi^\pm)
  		\end{align*}
		where convolution is in the sense of functions and distributions. 
		Then unraveling definitions, 
		\begin{align*}
			\kappa_{\Sigma^\pm}(t, x, z)
			&=
			\cF_{\mathbb{G}}^{-1}(d\Sigma_\psi^\pm)(t, x, z)\\
			&=
			\int_{\widehat{\mathbb{G}}}e^{i\mu t}\trace{\pi_\lambda(x, z)^*d\Sigma_\psi^{\pm}(\mu, \lambda)}\\
			&=
			\sum_{k\in \N}\int_{\widehat{\mathbb G}}e^{i\mu t}\trace{\pi_\lambda(x, z)^*P_k(\lambda)}\psi(|\mu|)d\sigma_k^\pm(\mu, \lambda).
		\end{align*}
		We then use the fact
		\begin{equation}
			\trace{\pi_\lambda(x, z)P_k(\lambda)}=\sum_{|\alpha|=k}E_{\alpha\alpha}^\lambda(x, z)=e_k^\lambda(x, z)
		\end{equation}
		and the definition of $d\sigma_k^\pm$ to obtain the formula below. Here $\trace{\pi_\lambda(x,z)^*P_k(\lambda)}=\overline{e_k^\lambda(x,z)}=e_k^{-\lambda}(x,z)$, and the change of variables $\lambda\mapsto-\lambda$ removes the conjugation under the integral.
		\begin{align*}
			\kappa_{\Sigma^\pm}(t, x, z)
			&=
			\frac{1}{(2\pi)^{d+m}}\sum_{\substack{k\in \N}}\int_{\fv^*}e^{\pm i\sqrt{|\lambda|(2k+d)} t}e_k^\lambda(x, z)\psi\left(\sqrt{|\lambda|(2k+d)}\right)|\lambda|^dd\lambda.
		\end{align*}
		Taking spherical coordinates $\lambda=\rho\omega$ on $\fv^*$, the above expression becomes
		\begin{align*}
			\frac{1}{(2\pi)^{d+m}}\sum_{\substack{k\in \N }}\int_0^\infty\int_{S}e^{\pm i\sqrt{\rho(2k+d)} t}e_k^{\rho\omega}(x, z)\psi\left(\sqrt{\rho(2k+d)}\right)d\sigma_S(\omega)\rho^{d+m-1}d\rho.
		\end{align*}
		The substitution $\rho=\mu^2/(2k+d)$ leads to 
		\begin{align*}
			\kappa_{\Sigma^\pm}(t, x, z)
			&=
			\frac{2}{(2\pi)^{d+m}}\sum_{\substack{k\in \N}}\frac{1}{(2k+d)^{d+m}}\int_0^\infty\int_{S}e^{\pm i\mu t}e_k^{\frac{\mu^2\omega}{2k+d}}(x, z)\psi\left(\mu\right)d\sigma_S(\omega)\mu^{2d+2m-1}d\mu.
		\end{align*}
		A comparison with the formula for $\kappa_\mu$ in Proposition \ref{def:GSlice} gives Equation \eqref{eq:k_wave}, since $\kappa_{\mu^2}$ carries the factor $\mu^{2d+2m-2}$ and $\mu^{2d+2m-1}=\mu\cdot \mu^{2d+2m-2}$. The proof that $\kappa_{\Sigma\pm}\in L^\infty(\mathbb{G})\cap C(\mathbb{G})$ is identical to the proof of Proposition \ref{prop:GSigma}. 
	\end{proof}

	In the following proposition and its proof, $\kappa_{\Sigma^\pm}$ continues to denote the kernel associated with the weighted measure $d\Sigma_\psi^\pm$. 
	
	\begin{proposition}\label{prop:TTstarWave}
		Let $\mathbb{G}=\R\times G$, and $\kappa_{\Sigma^\pm}$ be defined in Proposition~\ref{p:stev5}. For any $p, q\in [1, 2]$ and $1\leq r \leq 2(m+1)/(m+3)$  satisfying  $q\leq p$ and $r\leq p$   with $(m,p)\neq (1,2)$, there exists $C>0$ such that  for all $f\in \cS(\mathbb{G})$    %
		\begin{equation}\label{eq:wave-restriction-estimate}
			\|f*_{\mathbb{G}} \kappa_{\Sigma^\pm}\|_{ L^{r'}_\fv L^{q'}_tL^{p'}_\fh}\leq C \left(\sum_{k=0}^\infty\frac{\|\Lambda_k\|_{L^p\to L^{p'}}}{(2k+d)^{m(\frac{1}{r}-\frac{1}{r'})+d(\frac{1}{p}-\frac{1}{p'})}} \right)\|f\|_{ L^r_\fv L^q_tL^p_\fh},
		\end{equation}
		where the series on the right hand side is convergent.
	\end{proposition}
	
	\begin{proof}[Proof of Proposition \ref{prop:TTstarWave}]

		By Equation \eqref{eq:k_wave} and by arguing as in Corollary \ref{cor:GP_link}, we have

		\begin{equation}\label{eq:wave-kernelH} 
			f*_\mathbb{G} \kappa_{\Sigma^{\pm}}(t, x, z)=2\int_0^\infty e^{\pm i\mu t}\cP_{\mu^2}(f^{\pm\mu})(x, z)\psi(\mu)\mu\,d\mu.
		\end{equation}
		where  
		$\cP_\mu=\cP_\mu^{-\Delta}$ is the generalized spectral projector given in Equation \eqref{eq:spec-proj-Htype}  for spectral value $\mu$ of the (minus) sub-Laplacian $-\Delta_{H}$ of $G$. We can rewrite identity \eqref{eq:wave-kernelH} as 
		\begin{equation}\label{eq:wave-kernel3}
			f*_\mathbb{G} \kappa_{\Sigma^\pm}(t, x, z)=2\mathcal{F}_{\mu\mapsto \mp t}\left[\cP_{\mu^2}(f^{\pm\mu})(x, z)\psi_+(\mu)\mu\right]
		\end{equation}
		where $\cF_{\mu\mapsto \mp t}$ denotes the Euclidean Fourier transform in the $\mu$ variable evaluated at $\mp t$.
		Under the same assumptions, since $q\leq p\leq 2\leq p'$, we may apply the anisotropic Hausdorff-Young inequality (Lemma~\ref{l:an-haus-young}) to Equation \eqref{eq:wave-kernel3} to  obtain   
		\begin{equation}\label{eq:estimateWave1}
			\|f*_{\mathbb{G}} \kappa_{\Sigma^\pm}\|_{L^{r'}_\fv L^{q'}_tL^{p'}_\fh}
			\leq 
			C_q\|\cP_{\mu^2}(f^{\pm\mu})\psi_+(\mu)\mu\|_{ L^{r'}_\fv L^{q}_\mu L^{p'}_\fh}.
		\end{equation}
		where $\widehat{\R}_\mu$ is the dual of $\R_t$.
		Then by an application of the Minkowski Inequality,  the term on the right hand side in \eqref{eq:estimateWave1} is less than or equal to 
		\begin{equation*}
			C_q\|\cP_{\mu^2}(f^{\pm\mu})\psi_+(\mu)\mu\|_{L^{q}_\mu L^{p'}_\fh L^{r'}_\fv}.
		\end{equation*}
		By Theorem \ref{th:bound_spec_proiec_Htype}, there exists $C>0$ such that, denoting
		\begin{equation}\label{eq:ai}
			\eta:=2m\left(\frac{1}{r}-\frac{1}{r'}\right)+2d\left(\frac{1}{p}-\frac{1}{{{p'}}}\right)-1
		\end{equation}
		we have
		\begin{align}\label{eq:spec-proj-estimateWave}
			\|\cP_{\mu^2}(f^{\pm\mu})\psi_+(\mu)\mu\|_{L^{q}_\mu L^{p'}_\fh L^{r'}_\fv}
			&\leq
			C\bigg\|\|f^{\pm\mu}\|_{L^p_\fh L^r_\fv}\psi_+(\mu) \mu^{\eta}\bigg\|_{L^q_\mu} \\
			&\times\left(\sum_{k=0}^\infty(2k+d)^{-m(\frac{1}{r}-\frac{1}{r'})-d(\frac{1}{p}-\frac{1}{{{p'}}})}\|\Lambda_k\|_{L^p_\fh\to L^{p'}_\fh}\right) 
		\end{align}
		whenever the right side is finite. Note that  the series in \eqref{eq:spec-proj-estimateWave} is convergent. 
		Then, applying the H\"{o}lder Inequality on the first term on the right, for $\frac{1}{b}+\frac{1}{q'}=\frac{1}{q}$, 
		we have
		\begin{equation}\label{eq:wave-holder}
			\bigg\|\|f^{\pm\mu}\|_{L^{p}_{\fh} L^{r}_{\fv}}\psi_+(\mu)\mu^{\eta}\bigg\|_{L^q_\mu}
			\leq
			\|f^{\pm\mu}\|_{L^{q'}_\mu L^p_\fh L^r_\fv }\|\psi_+(\mu)\mu^{\eta}\|_{L^b_\mu}.
		\end{equation}
		
		The second term on the right hand side of \eqref{eq:wave-holder} is clearly finite. As in the proof of Proposition \ref{prop:TTstarintro2}, the first term is bounded as follows:
		\begin{align*}
			\|f^{\pm\mu}\|_{L^{q'}_\mu L^p_\fh L^r_\fv}
			&\leq 
			C_q\|f\|_{L^r_\fv  L^q_tL^p_\fh}, \quad \text{since } p\geq r, q.
		\end{align*}
		The result then follows once the convergence of the series in in \eqref{eq:spec-proj-estimateWave} is proved, which is done in Proposition \ref{tm:TT*-Htype}  thanks to $(m,p)\neq (1,2)$. 
	\end{proof}
	
	\begin{remark} Notice that the statement of Proposition~\ref{prop:TTstarWave} for the wave kernel coincides with the Schrodinger one. The main difference in the proof is the value of the exponent $\eta=\eta(m,d,r,p)$ as defined in \eqref{eq:ai}, which is different due to the different rescaling in \eqref{eq:wave-kernelH}. This does not affect the convergence of the series in \eqref{eq:spec-proj-estimateWave}.
		One may notice that a similar argument would work to get restriction estimates associated to different surfaces like those associated to other powers of the sub-Laplacian $(-\Delta_{H})^{s}$ and their corresponding kernels.
	\end{remark}
	
	\appendix
	
	\section{A few technical results} \label{s:app}
	This appendix contains a miscellanea of results.
	\subsection{Hausdorff-Young Inequality}
	\begin{lemma}[Hausdorff-Young Inequality for Mixed Lebesgue Spaces] \label{l:an-haus-young}
		Let $X$ and $Z$ be vector spaces and $f\in L^1(X\times Z)$. For $\lambda\in Z^*$, let 
		\begin{equation*}
			\cF_{z\to \lambda}f(x)=\int_{Z}e^{-i\langle\lambda, z\rangle}f(x, z)dz
		\end{equation*} be the Fourier transform of $f(x, z)$ in the last variable. Suppose $a, b\in [1, \infty]$ are such that  $b\leq \min\{a, a'\}$.
		Then there exists $C>0$ such that
		\begin{equation}\label{eq:an-haus-young}
			\| \cF_{z\to \lambda}f\|_{L^{b'}(Z^*)L^{a'}(X)}\leq C \|f\|_{L^b(Z)L^{a'}(X)}.
		\end{equation}
	\end{lemma}
	\begin{proof} It follows from 
		\begin{align*}
			\|\mathcal{F}_{z\to \lambda}f\|_{L^{b'}(Z^*)L^{a'}(X)}
			&\leq 
			\|\mathcal{F}_{z\to \lambda}f\|_{L^{a'}(X)L^{b'}(Z^*)}, \quad \text{since } a'\leq b'\\
			&\leq
			C\|f\|_{L^{a'}(X)L^{b}(Z)}, \quad \text{by the usual Hausdorff-Young,}\\
			&\leq
			C\|f\|_{L^{b}(Z) L^{a'}(X)},\quad \text{since } b\leq a',\qedhere
		\end{align*}
	\end{proof}
	
	\subsection{Spectral localization}
	Here $G$ is an $H$-type group.

	Let $L:=-\Delta_H$. On $C_c^\infty(G)$ the operator $L$ is essentially self-adjoint, since the vector fields $X_j$ are skew-adjoint on $L^2(G)$ and $G$ is complete for its sub-Riemannian distance \cite[\S12]{strichartzSubRiemannianGeometry1986}; we denote its closure again by $L$. It is a nonnegative self-adjoint operator on $L^2(G)$, and for a Borel function $\phi$ on $[0, \infty)$ the operator $\phi(L)$ is defined by the spectral theorem. For $s\in \R$ set
	\begin{equation}\label{eq:sobolev-def}
		\|f\|_{\dot H^s(G)}:=\|L^{s/2}f\|_{L^2(G)}, \qquad \|f\|_{H^s(G)}:=\|(1+L)^{s/2}f\|_{L^2(G)},
	\end{equation}
	so that $\|f\|_{\dot H^s(G)}\leq \|f\|_{H^s(G)}$ for $s\geq 0$. Fix $\chi\in C_c^\infty((0, \infty))$ with $0\leq \chi\leq 1$, $\supp \chi\subset [\tfrac14, 4]$ and
	\begin{equation}\label{eq:LP-partition}
		\sum_{j\in \Z}\chi(2^{-2j}\mu)=1, \qquad \mu>0,
	\end{equation}
	and set $\Xi_j:=\chi(2^{-2j}L)$ for $j\in \Z$. Write $\cS_0(G):=\operatorname{span}\{\Xi_jg:j\in\Z,\ g\in\cS(G)\}$. For $s\in\R$, the homogeneous Besov space $\dot B^s_{2,1}(G)$ means the completion of $\cS_0(G)$ in the norm
	\begin{equation}\label{eq:besov-def}
		\|f\|_{\dot B^{s}_{2,1}(G)}:=\sum_{j\in \Z}2^{js}\|\Xi_jf\|_{L^2(G)}.
	\end{equation}
The solution operators on these homogeneous spaces are the continuous extensions from $\cS_0(G)$ supplied by the Strichartz estimates. For $s\geq0$, $H^s(G)=D((1+L)^{s/2})$, and $\cS_0(G)$ is dense in $H^s(G)$: first apply $\sum_{|j|\leq N}\Xi_j$ and then approximate the underlying $L^2$ datum by Schwartz functions.
\begin{definition}\label{d:spectral}
		Let $S\subset [0, \infty)$ be closed. A function $f\in L^2(G)$ is \textit{spectrally supported in $S$} if $\mathbf{1}_S(L)f=f$.
	\end{definition}
	
We need the following lemma, which for $H$-type groups can be find in \cite{GS12} (cf.\ \cite{BGX00,BG01} where analogous estimates are proved in Heisenberg groups).
	\begin{lemma}\label{l:spectral}
		Let $f\in L^2(G)$, $j\in \Z$ and $s\in \R$.
		Then we have
		\begin{enumerate}
			\item  $f=\sum_{k\in \Z}\Xi_{k}f$ in $L^2(G)$, and $\Xi_jf$ is spectrally supported in $$2^{2j}\supp\chi\subset [2^{2j-2}, 2^{2j+2}].$$
			\item For every bounded Borel function $\phi$ on $[0, \infty)$, one has $$\cF_G(\phi(L)f)(\lambda, \alpha, \beta)=\phi\big(|\lambda|(2|\alpha|+d)\big)\cF_G(f)(\lambda, \alpha, \beta)$$ for almost every $\lambda\in \fv^*\mz$ and all $\alpha, \beta\in \N^d$;   equivalently, for all $k\in \N$ $$\cF_G(\phi(L)f)(\lambda)P_k(\lambda)=\phi\big(|\lambda|(2k+d)\big)\cF_G(f)(\lambda)P_k(\lambda).$$
			\item Let $\delta_\Lambda(x, z)=(\Lambda x, \Lambda^2 z)$ for $\Lambda>0$. For every bounded Borel function $\phi$ on $[0, \infty)$ and $g\in L^2(G)$, we have $$\phi(L)(g\circ\delta_\Lambda)=(\phi(\Lambda^{2}L)g)\circ \delta_\Lambda.$$ Consequently, if $g$ is spectrally supported in $S$, then $g\circ\delta_\Lambda$ is spectrally supported in $\Lambda^2S$, and $(\Xi_jf)\circ \delta_{2^{j}}^{-1}=\Xi_0(f\circ \delta_{2^{j}}^{-1})$.
			\item There is $C_s\geq 1$, depending only on $s$ and $\chi$, such that
			\begin{gather}\label{eq:LP-sobolev}
				C_s^{-1}\|f\|_{\dot H^s(G)}^2\leq \sum_{k\in \Z}2^{2ks}\|\Xi_{k}f\|_{L^2(G)}^2\leq C_s\|f\|^2_{\dot H^s(G)}, \\ 2^{js}\|\Xi_jf\|_{L^2(G)}\leq 2^{|s|}\|L^{s/2}\Xi_jf\|_{L^2(G)}.
			\end{gather}
		\end{enumerate}
	\end{lemma}
	\subsection{Failure of Sobolev estimates at endpoint}\label{a:estfalse}

	 (i) We first recall that 
	 \begin{equation}\label{eq:hq}
	 H^{Q/2}(G)\not\subset L^\infty(G)
	 \end{equation} so that the Sobolev form of the estimate of Theorem~\ref{t:main2} with $p=q=r=\infty$, $\sigma=Q/2$ and $f=0$, which at $t=0$ reads $\|u_0\|_{L^\infty(G)}\leq C\|u_0\|_{H^{Q/2}(G)}$, fails.
	
	To prove \eqref{eq:hq}, fix $0<a<b$ and $R>1$ with $R^2a>b$, and let $\varphi\in \cS(G)$ be\footnote{for instance $\varphi=c\,\chi_0(-\Delta_H)g$ with $\chi_0\in C_c^\infty((a, b))$, $g\in \cS(G)$ and $c\in \C$} spectrally supported in $[a, b]$ in the sense of Definition~\ref{d:spectral} with $\varphi(e)=1$ at the identity $e\in G$. For $N\geq 1$ set
		\[
			f_N:=N^{-1/2}\sum_{j=1}^N\varphi\circ \delta_{R^j}, \qquad \delta_{\Lambda}(x,z)=(\Lambda x,\Lambda^{2}z).
		\]
		The function $\varphi\circ\delta_{R^j}$ is spectrally supported in $R^{2j}[a, b]$ (by Lemma~\ref{l:spectral} (3)), these intervals are pairwise disjoint, and $\|h\circ \delta_\Lambda\|_{\dot H^{s}(G)}=\Lambda^{s-\frac Q2}\|h\|_{\dot H^{s}(G)}$. Hence the summands are orthogonal in $\dot H^s(G)$ for every $s\geq 0$, $\|f_N\|_{\dot H^{Q/2}(G)}=\|\varphi\|_{\dot H^{Q/2}(G)}$ and $\|f_N\|_{L^2(G)}\leq C\|\varphi\|_{L^2(G)}$ with $C$ independent of $N$, while $f_N(e)=\sqrt N$, which proves the claim.   
		
	(ii)	The same data with $v_0=0$ show the failure of the Sobolev form of Theorem~\ref{t:main3} at $p=q=r=\infty$, where $\sigma+1=Q/2$. 
	For $p=\infty$ or $q=\infty$ the triangle inequality still gives the estimates with the data measured in $\dot B^{\sigma}_{2,1}(G)$, as stated in the theorems.	
	
	\section{The twisted (isotropic) Laplacian}\label{twisted-lap}

	Given $f\in S(\mathbb{H}_d)$ we denote $f^{\lambda}$ the Fourier transform in the vertical variable, namely
	\begin{equation}\label{eq:fourier_trans_central}
		f^{\lambda}(x)=\int_{\mathbb{R}}f(x, z)e^{-i\lambda z}dz.
	\end{equation}
	The twisted Laplacian is defined by the action of the sub-Laplacian on the $z$-Fourier modes of a function $f=f(x, z)$ on the Heisenberg group:
	$$(\Delta_{H} f)^{\lambda}=\Delta_{H}^{\lambda }f^{\lambda}$$
 	Precisely, (where we let $y_i:=x_{i+d}$ for $i=1,\ldots, d$)
	\begin{align*}
		\Delta_{H}^{\lambda}:=\Delta_{\R^{2d}}-i\lambda\sum_{j=1}^{d} \left(y_j\frac{\partial}{\partial x_j}-x_j\frac{\partial}{\partial y_j}\right)-\frac{1}{4}|\lambda|^2\sum_{j=1}^d(x_j^2+y_j^2).
	\end{align*}
	
	The twisted Laplacian $\Delta^\lambda$ is  essentially self-adjoint and elliptic  on $S(\mathbb{C}^d)\subset L^2(\mathbb{C}^d)$ with a pure point spectrum consisting of eigenvalues $\{-|\lambda|(2k+d)\mid k\in \N\}$. In the sequel, we construct the corresponding eigenfunctions for $\Delta^\lambda$.  
	
	Following \cite[pg. 22]{thangaveluHarmonicAnalysisHeisenberg1998} let $	\Phi_\alpha^{\lambda}$ be the (multi) Hermite function on $\R^d$ defined by
	\begin{align}\label{eq_Hermite_func}
		\Phi_\alpha^{\lambda}(\xi):=|\lambda|^{\frac{d}{4}}h_{\alpha_1}(|\lambda|^{\frac{1}{2}}\xi_1)\cdots h_{\alpha_d}(|\lambda|^{\frac{1}{2}}\xi_d), \quad\xi\in \mathbb{R}^d,
	\end{align}
	where $h_k$ shall denote the $k$-th Hermite function on $\mathbb{R}$ normalized as follows
	\begin{align*}
		h_k(t)=(2^k\sqrt{\pi}k!)^{-\frac{1}{2}}H_k(t)e^{-\frac{1}{2}t^2}, \quad H_k(t)=(-1)^k\left(e^{t^2}\frac{d^k}{dt^k}\{e^{-t^2}\}\right).
	\end{align*}
	For $\alpha, \beta\in\N^d$ and $\lambda>0$, define the special Hermite functions let $\Phi_{\alpha\beta}^\lambda$ by
	\begin{equation}\label{def:spec-herm}
		\Phi_{\alpha\beta}^{\lambda}(x)=(2\pi)^{-\frac{d}{2}}\lambda^{\frac{d}{2}}\langle \pi_\lambda(x, 0)\Phi_\alpha^\lambda, \Phi_\beta^\lambda\rangle, \quad x\in \mathbb{C}^{d}.
	\end{equation}
 	
	The special Hermite functions form a complete orthonormal system for $L^2(\C^d) $ \cite[pg.18]{thangaveluHarmonicAnalysisHeisenberg1998}, and are eigenfunctions of the twisted Laplacian $\Delta^{\lambda}$ with 
	\begin{equation}
		\Delta^\lambda \Phi_{\alpha\beta}^\lambda=-|\lambda|(2|\alpha|+d)\Phi_{\alpha\beta}^\lambda.
	\end{equation}
	Let $\Lambda_k^\lambda$ be the orthogonal projector onto the eigenspace of $\Delta^\lambda$  for the eigenvalue $-|\lambda|(2k+d)$:
	\begin{equation}
		\Lambda_k^\lambda g:=\sum_{|\alpha|=k}\sum_{\beta\in\N^d}\langle g, \Phi_{\alpha\beta}^\lambda\rangle\Phi_{\alpha\beta}^\lambda, \quad g\in L^2(\C^d).
	\end{equation}
	\begin{remark}
		The definition of $\Lambda_k^\lambda$ given above, differs from the orthogonal projector $\Lambda_k^\lambda$  defined in \cite{casarinoRestrictionTheoremMetivier2013} by a factor of $|\lambda|^d$. The normalization here guarantees 
		$$\Lambda_k^\lambda\circ\Lambda_k^\lambda=\Lambda_k^\lambda.$$
	\end{remark}
	The projection $\Lambda_k^\lambda$ can be written in a more explicit form as a twisted convolution with a Laguerre function.
	For $\lambda>0$ let $\times_\lambda$ be the $\lambda$-twisted convolution defined (as in \cite[pg. 16]{thangaveluHarmonicAnalysisHeisenberg1998}) by 
	$$g\times_\lambda h(x):=\int_{\C^d}g(x-w)h(w)e^{\frac{i\lambda}{2}\textrm{Im}(x\cdot \overline{w})}dw.$$
	Moreover, let $\varphi_k^\lambda$ be  the Laguerre function defined (as in  \cite[page 52]{thangaveluHarmonicAnalysisHeisenberg1998})  for $\lambda>0$ by
	\begin{equation}
		\varphi_k^\lambda(x)=L_k^{d-1}(\tfrac{1}{2}\lambda|x|^2)e^{-\frac{1}{4}\lambda|x|^2},
	\end{equation}
	where $L_k^{d-1}$ denotes the $k$-th Laguerre polynomial of generalized type $d-1$. 
	Note the difference of a factor of $\lambda^d$ with \cite[page 9]{niedorfSpectralMultiplierTheorem2024}. Then, since 
	\begin{equation}
		\Phi_{\alpha\beta}^\lambda(x)=\lambda^{\frac{d}{2}}\Phi^1_{\alpha\beta}(\lambda^{\frac{1}{2}}x),
	\end{equation}
	by definition of $\Phi_{\alpha\beta}^\lambda$,  equations
	(1.3.41) and (1.3.42) of  \cite[page 30]{thangaveluLecturesHermiteLaguerre1993} imply
	\begin{align*} 
		\varphi_k^\lambda=(2\pi)^{\frac{d}{2}}\lambda^{-\frac{d}{2}}\sum_{|\alpha|=k}\Phi_{\alpha\alpha}^\lambda.
	\end{align*}
	Hence by (2.1.5) of \cite[page 30]{thangaveluLecturesHermiteLaguerre1993}
	\begin{equation}
		\Lambda_k^\lambda g := (2\pi)^{-d}|\lambda|^d g\times_\lambda \varphi_k^\lambda
	\end{equation}

	The fact that $\Lambda_k^\lambda$ is a projection can be verified by the following identity (see (1.4.30) in \cite[page 21]{thangaveluLecturesHermiteLaguerre1993}) for $\lambda>0$
 	$$\varphi_j^{\lambda}\times_\lambda\varphi_k^{\lambda}=(2\pi/\lambda)^d\delta_{jk}\varphi_j^{\lambda}.$$

	\subsection{Estimate on Twisted Spectral Projectors}
	In view of the following identity  $\Phi_{\alpha\beta}^\lambda(x)=\lambda^{\frac{d}{2}}\Phi_{\alpha\beta}(\lambda^{\frac{1}{2}}x)$ we have the homogeneity property
	\begin{equation}\label{homo-twist}
		\|\Lambda_k^\lambda\|_{p\to r}=\lambda^{d(\frac{1}{p}-\frac{1}{r})}\|\Lambda_k\|_{p\to r}.
	\end{equation}
	Theorem 1 of \cite{kochSpectralProjectionsTwisted2004} states the following:
	\begin{theorem}\label{th:koch-ricci} For $2\leq r\leq\infty$ we have the estimate
		\begin{equation}\label{eq:twisted-estimate}
			\|\Lambda_k\|_{L^2\to L^r}\lesssim_n (2k+d)^{\frac{1}{2}\varrho(r)},
		\end{equation}
		where
		\begin{equation} \label{eq:rho}
			\varrho(r)=\begin{cases}
				\frac{1}{r}-\frac{1}{2}, & \text{if } 2\leq r\leq \frac{2(2d+1)}{2d-1},\\
				2d\left(\frac{1}{2}-\frac{1}{r}\right)-1, & \text{if }\frac{2(2d+1)}{2d-1}\leq r\leq \infty
			\end{cases}.
		\end{equation}
	\end{theorem}
	\begin{remark}
		Note that
		\begin{align*}
			2d\left(\frac{1}{2}-\frac{1}{r}\right) =d\left(1-\frac{2}{r}\right)=d\left(\frac{1}{r'}-\frac{1}{r}\right).
		\end{align*}
	\end{remark}
	By duality, we also obtain for $1\leq p\leq 2$, 
	\begin{equation}\label{eq:twisted-estimate-dual}
		\|\Lambda_k\|_{p\to 2}\lesssim_d (2k+d)^{\frac{1}{2}\varrho(p')}.
	\end{equation}
	Using the fact that $\Lambda_k^2=\Lambda_k$ we combine \eqref{eq:twisted-estimate} and \eqref{eq:twisted-estimate-dual} to obtain the following corollary.
	\begin{corollary} \label{c:twisted-estimate-comb}
		For $1\leq p\leq 2\leq r\leq\infty$,
		\begin{equation}\label{eq:twisted-estimate-comb}
			\|\Lambda_k\|_{p\to r}\lesssim_d (2k+d)^{\frac{1}{2}(\varrho(p')+\varrho(r))}.
		\end{equation}
	\end{corollary}
	\begin{remark}
		Note that when $r=p'$, the exponent becomes $\frac{1}{2}(\varrho(p')+\varrho(r))=\varrho(p')$.
	\end{remark}
	
	\begin{corollary}\label{c:twisted-large-range}
		Suppose $1\leq p \leq 2(2d+1)/(2d+3)$. Then
		\begin{equation}\label{eq:spec-proj-bound2}
			\|\Lambda_k\|_{p\to p'}\lesssim_d (2k+d)^{d(\frac{1}{p}-\frac{1}{p'})-1}.
		\end{equation}
	\end{corollary}
	\begin{proof}
		If $1\leq p \leq 2(2d+1)/(2d+3)$ then $2(2d+1)/(2d-1)\leq p'\leq \infty$.
		Taking $r=p'$ in Equation \eqref{eq:twisted-estimate-comb} above yields the result.\qedhere

		\end{proof}

				\section{Generalities on Spectral Projectors}\label{s:app2}

		\subsection{Spectral Projector for the sub-Laplacian on $H$-type groups}
		
		For $\rho\in\R\setminus\{0\}$, use the Heisenberg representation
\[
 \pi_\rho(x,y,z)\Phi(\xi)
 =e^{i\rho(z+\xi\cdot y+\frac12x\cdot y)}\Phi(\xi+x).
\]
With $\cF_\heis f(\rho)=\int_{\heis^d}f(w)\pi_\rho(w)\,dw$ and
$f^\rho(x,y)=\int_\R f(x,y,z)e^{-i\rho z}\,dz$, we have
\[
 \cF_\heis f(\rho)=\int_{\R^{2d}}f^{-\rho}(x,y)\pi_\rho(x,y,0)\,dx\,dy.
\]
The integral kernel, defined by
$(\cF_\heis f(\rho)\Phi)(\xi)=\int K_\rho^f(\xi,\eta)\Phi(\eta)\,d\eta$, is
\[
 K_\rho^f(\xi,\eta)=\cF_{y,z}f
 \left(\eta-\xi,-\frac{\rho}{2}(\xi+\eta),-\rho\right),
\]
where $\cF_{y,z}$ uses the negative exponential. Integrating its diagonal gives
\[
 \trace{\cF_\heis f(\rho)}=(2\pi)^d|\rho|^{-d}f^{-\rho}(0,0).
\]
For $G$, the representation $\pi_\lambda=\pi_{|\lambda|}\circ\alpha_\lambda$ gives
\[
 \cF_Gf(\lambda)=\int_{\R^{2d}}f^{-\lambda}(T_\lambda v)
 \pi_{|\lambda|}(v,0)\,dv,\qquad
 \trace{\cF_Gf(\lambda)}=(2\pi)^d|\lambda|^{-d}f^{-\lambda}(0).
\]
Euclidean inversion in the central variables therefore yields
\[
 f(0,0)=(2\pi)^{-d-m}
 \int_{\fv^*}\trace{\cF_Gf(\lambda)}|\lambda|^d\,d\lambda.
\]
Apply this identity to $\delta_{w^{-1}}*_Gf$ to obtain
\[
 f(w)=(2\pi)^{-d-m}\int_{\fv^*}
 \trace{\pi_\lambda(w)^*\cF_Gf(\lambda)}|\lambda|^d\,d\lambda.
\]
Unitarity gives
$\overline{E_{\alpha\alpha}^\lambda(v^{-1}w)}
=E_{\alpha\alpha}^\lambda(w^{-1}v)$, whence
\[
 \trace{\pi_\lambda(w)^*\cF_Gf(\lambda)}
 =\sum_{\alpha\in\N^d}(f*_G\overline{E_{\alpha\alpha}^\lambda})(w).
\]
For the sum over one Hermite eigenspace,
\[
 e_k^\lambda(x,z)=e^{i\lambda(z)}
 \varphi_k^{|\lambda|}(T_\lambda^{-1}x).
\]
Since the Laguerre function is real and radial and $T_\lambda$ is orthogonal,
$\overline{e_k^\lambda}=e_k^{-\lambda}$. Grouping by $k$ and changing
$\lambda$ to $-\lambda$ recovers the kernels in \eqref{eq:spec_proj_alt}.

\begin{proof}[Proof of Theorem~\ref{th:spectral-resolution}]
			Identifying the coordinate $t$ with an element of the center Lie algebra $\fv$,   let $$\cF_\fv f(x, \eta):=f^\eta(x):=\int_{\fv}f(x, z)e^{-i\eta(z)}dz, \quad \eta\in \fv^*,$$ be the Euclidean Fourier transform in the $\fv$ variable. 
			The inversion formula is 
			\begin{equation}
				f(x, z)
				=\frac{1}{(2\pi)^m}\int_{\fv^*}\cF_\fv f(x, \eta)e^{i\eta(z)}d\eta, \quad f\in \cS(G).
			\end{equation}
			Writing the inversion formula for  $\cF_\fv$ in polar coordinates gives 
			\begin{align}\label{fourier-inv-central}
				f(x, z)
				=\frac{1}{(2\pi)^m}\int_0^\infty\int_{S}\cF_\fv f(x, \rho\omega)e^{i\rho\omega(z)}d\sigma(\omega)\rho^{m-1} d\rho.
			\end{align}

			For $\eta=\rho\omega\neq0$, $\cF_\fv f(\cdot,\eta)\in\cS(\fh)$. The orthogonal spectral resolution of the twisted Laplacian, conjugated as in \eqref{eq:directional-projector}, gives
			\begin{equation}\label{spec-fourier}
				\cF_\fv f(x,\rho\omega)=\sum_{k\in\N}\Lambda_{k,\omega}^{\rho}f^{\rho\omega}(x).
			\end{equation}
			Combining \eqref{spec-fourier} and \eqref{fourier-inv-central} formally gives the following identities; their rearrangements are justified below: 
			\begin{equation}
				f(x, z)=\frac{1}{(2\pi)^m}\sum_{k=0}^\infty\int_0^\infty\left(\int_S\Lambda_{k,\omega}^{\rho} f^{\rho\omega}(x)e^{i\rho\omega(z)}d\sigma(\omega)\right) \rho^{m-1}d\rho
			\end{equation}
			Now perform the substitution $\mu=\rho(2k+d)$ to obtain
			\begin{equation}
				f(x, z)=\frac{1}{(2\pi)^m}\sum_{k=0}^\infty\frac{1}{(2k+d)^{m}}\int_0^\infty\left(\int_S\Lambda_{k,\omega}^{\mu_k} f^{\mu_k\omega}(x)e^{i\mu_k\omega(z)}d\sigma(\omega)\right) \mu^{m-1}d\mu
			\end{equation}
			with $\mu_k:=\mu/(2k+d)$. The required input intertwining follows from 
			\begin{align*}
 \Lambda_{k,\omega}^{\mu_k}\cF_{\fv}(\Delta f)(x,\mu_k\omega)
 &=\Lambda_{k,\omega}^{\mu_k}
   \big(\Delta_H^{\mu_k\omega}f^{\mu_k\omega}\big)(x)\\
 &=-\mu_k(2k+d)\Lambda_{k,\omega}^{\mu_k}f^{\mu_k\omega}(x)\\
 &=-\mu\Lambda_{k,\omega}^{\mu_k}f^{\mu_k\omega}(x).
\end{align*}
		 
			To see that  \eqref{eq:spec_proj_alt} is equivalent to \eqref{eq:spec-proj-Htype}, note that
			\begin{align*}
				(f*_G e_k^{\rho\omega})(x,z)=(2\pi)^d\rho^{-d}e^{i\rho\omega(z)}\Lambda_{k,\omega}^{\rho}f^{\rho\omega}(x),\qquad \rho>0,\ \omega\in S.
			\end{align*}
		
Set $L=-\Delta_H$ and $D_k=\binom{k+d-1}{d-1}$. Since $|e_k^\lambda|\leq D_k$ and $\sum_kD_k(2k+d)^{-d-m}<\infty$, the convolution formula \eqref{eq:spec_proj_alt} converges absolutely and uniformly for each $\mu>0$, with
\[
 \|\cP_\mu h\|_\infty\leq C\mu^{d+m-1}\|h\|_1,\qquad h\in\cS(G).
\]
The input intertwining gives $\cP_\mu L^Nf=\mu^N\cP_\mu f$. For any integer $N>d+m$, use $h=f$ for $0<\mu\leq1$ and $h=L^Nf$ for $\mu\geq1$ to obtain
\[
 \int_0^\infty\|\cP_\mu f\|_\infty\,d\mu
 \leq C\left(\|f\|_1+\|L^Nf\|_1\right)<\infty.
\]
To justify the preceding rearrangement, test against $g\in\cS(G)$. Orthogonality gives
\[
 \sum_k\left|\int_\fh
 (\Lambda_{k,\omega}^{\rho}f^{\rho\omega})(x)
 \overline{g^{\rho\omega}(x)}\,dx\right|
 \leq\|f^{\rho\omega}\|_2\|g^{\rho\omega}\|_2.
\]
The right side is integrable in $\eta=\rho\omega$ by Cauchy--Schwarz and central Plancherel. Thus Fubini and $\mu=\rho(2k+d)$ give
\[
 \int_0^\infty\int_G\cP_\mu f(w)\overline{g(w)}\,dw\,d\mu
 =\int_G f(w)\overline{g(w)}\,dw.
\]
The uniformly convergent integral therefore equals $f$ pointwise. Inserting $a(\mu)$ for $a\in C_c^\infty((0,\infty))$ in the same calculation identifies the weak $L^2$ integral with $a(L)f$.

Finally, applying $\Delta_H$ to each projected term
$e^{i\rho\omega(z)}\Lambda_{k,\omega}^{\rho}f^{\rho\omega}(x)$
applies $\Delta_H^{\rho\omega}$ to the projected factor. It follows that
$L\cP_\mu f=\mu\cP_\mu f$ in $\cS'(G)$, as well as the input identity above. H\"ormander's theorem applied to $\Delta_H+\mu$ gives $\cP_\mu f\in C^\infty(G)$.
\end{proof}
		\begin{remark}
			Note that, compared with notation of \cite[Theorem 4.8]{casarinoRestrictionTheoremMetivier2013}, some powers of $\mu$ and $(2k+d)$ have been subsumed into the projector $\Lambda_k^{\mu_k}$.
		\end{remark}

		\begin{theorem}\cite{casarinoRestrictionTheoremMetivier2013}\label{th:bound_spec_proiec_Htype}
			Suppose $1\leq r\leq p_*(m)=2\frac{m+1}{m+3}$. If the projections $\Lambda_k$ are bounded from $L^p(\fh)$ to $L^{p_1}(\fh)$ (here $\fh\cong \R^{2d}$), with $1\leq p\leq 2\leq {{p_1}}\leq \infty$,  then there exists $C=C(m, r)>0$ such that
			\begin{align*}
				\|\cP_\mu^{-\Delta_H}f\|_{L^{{p_1}}(\fh){L^{r'}(\fv)}}
				&\leq C \mu^{m(\frac{1}{r}-\frac{1}{r'})+d(\frac{1}{p}-\frac{1}{{{p_1}}})-1}\\
				&\times \left(\sum_{k=0}^\infty(2k+d)^{-m(\frac{1}{r}-\frac{1}{r'})-d(\frac{1}{p}-\frac{1}{{{p_1}}})}\|\Lambda_k\|_{L^p(\fh)\to L^{{p_1}}(\fh)}\right)\|f\|_{L^p(\fh)L^r(\fv)}.
			\end{align*}
		\end{theorem}
		\begin{remark}
			Note that in \cite{casarinoRestrictionTheoremMetivier2013} the order of the order of the Lebesgue norms in the subscript are reversed. The difference is only in notation. 
		\end{remark}
		
		\section{The inhomogeneous case}\label{s:appino}
		The argument is standard. For completeness we report it here.
		Denoting by $(\mathcal U(t))_{t\in \R} $ the solution operator of the Schr\"odinger equation, namely $\mathcal U(t)u_0$ is the solution with $f=0$ at time $t$ associated with the data~$u_0$, then similarly to the Euclidean case   $(\mathcal U(t))_{t\in \R} $  is a one-parameter group of unitary operators on $L^2$. Moreover,   the solution to the inhomogeneous  equation  
		$$ \left\{
		\begin{array}{c}
			i\partial_t u -\Delta u = f\\
			u_{|t=0} = 0\,,
		\end{array}
		\right. $$
		writes 
		\begin{equation}
			\label{estlinh} 
			u(t, x,z)= -i \int^t_0 \mathcal U(t-t') f(t', x,z) dt'  ,
		\end{equation}
		Let us check that it satisfies, for all  admissible pairs $(p,q)$ 
		\begin{equation}
			\label{estlinhst} \|u\|_{L^r_\fv L^q_tL^p_\fh} \leq  C \|f \|_{L_t^1 H^{\sigma}} 
		\end{equation}
		with $ \sigma= \frac Q2 - \frac2q-\frac{2d}p-\frac{2m}{r}$

		Let $(r, p, q)$ be as in Theorem~\ref{t:main2} and let $u$ be given by \eqref{estlinh}. For $t'\in \R$ let $E_{t'}:=\{t\in \R: 0\leq t'\leq t\}\cup\{t\in \R: t\leq t'\leq 0\}$, so that $|u(t, \cdot)|\leq \int_\R \mathbf{1}_{E_{t'}}(t)\,|\mathcal U(t-t')f(t', \cdot)|\,dt'$. Minkowski's integral inequality in $L^r_\fv L^q_tL^p_\fh$, the bound $\mathbf{1}_{E_{t'}}\leq 1$, the translation invariance of $L^q_t$ and the homogeneous Strichartz estimate (for $f=0$) gives
		\begin{align}\label{eq:inhom-minkowski}
			\|u\|_{L^r_\fv L^q_tL^p_\fh}&\leq \int_\R\big\|\mathbf{1}_{E_{t'}}\,\mathcal U(\cdot-t')f(t', \cdot)\big\|_{L^r_\fv L^q_tL^p_\fh}dt'\\
			&\leq \int_\R\big\|\mathcal U(\cdot)f(t', \cdot)\big\|_{L^r_\fv L^q_tL^p_\fh}dt'\leq C\int_\R\|f(t', \cdot)\|_{H^\sigma}dt',
		\end{align}
		which is \eqref{estlinhst}. The same argument with the Besov statement of Theorem~\ref{t:main2} for $p=\infty$ or $q=\infty$. 
		Similar arguments can be used for the inhomogeneous wave equations.

	{\bf AI disclosure.} The endpoint estimates for infinite indices stated in the first version of this paper \cite{BF25} was incorrect (cf.~Remark \ref{r:pq-finite}). In revising the proof, we understood that an adaptation of the classical dyadic decomposition and resummation argument was needed, leading to a Besov space estimate at the endpoint. The revised proof relies on Lemma~\ref{l:recombination}, whose final formulation was completed with the assistance of AI tools. AI tools were not used to develop  other mathematical content in the paper.

The authors take full responsibility for the accuracy and integrity of the paper, and the quality of its writing. \\

	{\bf Acknowledgements.} The authors thanks Isabelle Gallagher and Hajer Bahouri for useful discussions on a second version of the manuscript. 
	
	Davide Barilari and Steven Flynn acknowledge the support granted by the European Union--NextGenerationEU Project ``NewSRG--New directions in sub-Riemannian Geometry'' within the Program STARS@UNIPD 2021. Davide Barilari acknowledges the support granted by the Italian PRIN Project ``Optimal Transport: new challenges across analysis and geometry'' within the Program PRIN 2022.

		\bibliographystyle{alphaabbrv}
		\bibliography{Strichartz-H-type}

@article{FMRS23,
      title={Non-existence of radial eigenfunctions for the perturbed {H}eisenberg sublaplacian}, 
      author={Luca Fanelli and Haruya Mizutani and Luz Roncal and Nico Michele Schiavone},
      year={2023},
      eprint={2310.18050},
         JOURNAL = {arXiv preprint},
      archivePrefix={arXiv},
      primaryClass={math.AP},
      url={https://arxiv.org/abs/2310.18050}, 
}

@article {BG01,
    AUTHOR = {Bahouri, Hajer and Gallagher, Isabelle},
     TITLE = {Paraproduit sur le groupe de {H}eisenberg et applications},
   JOURNAL = {Rev. Mat. Iberoamericana},
  FJOURNAL = {Revista Matem\'atica Iberoamericana},
    VOLUME = {17},
      YEAR = {2001},
    NUMBER = {1},
     PAGES = {69--105},
      ISSN = {0213-2230},
   MRCLASS = {35L70 (22E25 35A30 35B65 35J60)},
  MRNUMBER = {1846091},
MRREVIEWER = {Tohru\ Ozawa},
       DOI = {10.4171/RMI/289},
       URL = {https://doi.org/10.4171/RMI/289},
}

@article {GS12,
    AUTHOR = {Gallagher, Isabelle and Sire, Yannick},
     TITLE = {Besov algebras on {L}ie groups of polynomial growth},
   JOURNAL = {Studia Math.},
  FJOURNAL = {Studia Mathematica},
    VOLUME = {212},
      YEAR = {2012},
    NUMBER = {2},
     PAGES = {119--139},
      ISSN = {0039-3223,1730-6337},
   MRCLASS = {22E25 (46E35)},
  MRNUMBER = {3008437},
MRREVIEWER = {Ali\ Baklouti},
       DOI = {10.4064/sm212-2-2},
       URL = {https://doi.org/10.4064/sm212-2-2},
}

@article {BM22,
    AUTHOR = {Benea, Cristina and Muscalu, Camil},
     TITLE = {Multiple vector-valued, mixed-norm estimates for
              {L}ittlewood-{P}aley square functions},
   JOURNAL = {Publ. Mat.},
  FJOURNAL = {Publicacions Matem\`atiques},
    VOLUME = {66},
      YEAR = {2022},
    NUMBER = {2},
     PAGES = {631--681},
      ISSN = {0214-1493,2014-4350},
   MRCLASS = {42B20},
  MRNUMBER = {4443750},
MRREVIEWER = {Fanghui\ Liao},
       DOI = {10.5565/publmat6622205},
       URL = {https://doi.org/10.5565/publmat6622205},
}

@article{BL25,
      title={Strichartz estimates for {B}aouendi--{G}rushin operators}, 
      author={Nicolas Burq and Mickael Latocca},
      year={2026},
      journal={Rev. Mat. Iberoam},
      volume={42},
      number={3},
      pages={903--920},
      primaryClass={math.AP},
      url={https://arxiv.org/abs/2501.04415}, 
}

@article{BF25,
      title={{R}efined {S}trichartz {E}stimates for sub-{L}aplacians in {H}eisenberg and ${H}$-type groups}, 
      author={Davide Barilari and Steven Flynn},
      year={2025},
      eprint={2501.04415},
      archivePrefix={arXiv},
      journal={Arxiv Preprint 2501.04415v1, Jan 2025},
      primaryClass={math.AP},
      url={https://arxiv.org/abs/2501.04415}, 
}

@book{FR16,
  author = {Fischer, V\'eronique and Ruzhansky, Michael},
  title = {Quantization on Nilpotent {L}ie Groups},
  series = {Progress in Mathematics},
  volume = {314},
  publisher = {Birkh\"auser/Springer, Cham},
  year = {2016},
  doi = {10.1007/978-3-319-29558-9},
}

@article{GMS23,
      title={On local dispersive and {S}trichartz estimates for the {G}rushin operator}, 
      author={Sunit Ghosh and Shyam Swarup Mondal and Jitendriya Swain},
      year={2023},
      journal={Arxiv preprint},
      eprint={2306.10298},
      archivePrefix={arXiv},
      primaryClass={math.AP},
      url={https://arxiv.org/abs/2306.10298}, 
}

@article {BBGL,
    AUTHOR = {Bahouri, Hajer and Barilari, Davide and Gallagher, Isabelle
              and L\'eautaud, Matthieu},
     TITLE = {Spectral summability for the quartic oscillator with
              applications to the {E}ngel group},
   JOURNAL = {J. Spectr. Theory},
  FJOURNAL = {Journal of Spectral Theory},
    VOLUME = {13},
      YEAR = {2023},
    NUMBER = {2},
     PAGES = {623--706},
      ISSN = {1664-039X,1664-0403},
   MRCLASS = {43A30 (31C45 43A80 53C17)},
  MRNUMBER = {4651735},
MRREVIEWER = {Rajesh\ K.\ Srivastava},
       DOI = {10.4171/jst/464},
       URL = {https://doi.org/10.4171/jst/464},
}

@book {CG,
    AUTHOR = {Corwin, Lawrence J. and Greenleaf, Frederick P.},
     TITLE = {Representations of nilpotent {L}ie groups and their
              applications. {P}art {I}},
    SERIES = {Cambridge Studies in Advanced Mathematics},
    VOLUME = {18},
 PUBLISHER = {Cambridge University Press, Cambridge},
      YEAR = {1990},
     PAGES = {viii+269},
      ISBN = {0-521-36034-X},
}

@article{FMRS24,
      title={Uniform resolvent estimates, smoothing effects and spectral stability for the {H}eisenberg sublaplacian}, 
      author={Luca Fanelli and Haruya Mizutani and Luz Roncal and Nico Michele Schiavone},
      year={2024},
      eprint={2409.11943},
              JOURNAL = {arXiv preprint},
      archivePrefix={arXiv},
      primaryClass={math.AP},
      url={https://arxiv.org/abs/2409.11943}, 
}

@article {BBG21,
    AUTHOR = {Bahouri, Hajer and Barilari, Davide and Gallagher, Isabelle},
     TITLE = {Strichartz estimates and {F}ourier restriction theorems on the
              {H}eisenberg group},
   JOURNAL = {J. Fourier Anal. Appl.},
  FJOURNAL = {The Journal of Fourier Analysis and Applications},
    VOLUME = {27},
      YEAR = {2021},
    NUMBER = {2},
     PAGES = {Paper No. 21, 41},
      ISSN = {1069-5869,1531-5851},
   MRCLASS = {35R03 (35Q40)},
  MRNUMBER = {4228908},
MRREVIEWER = {Xing\ Cheng},
       DOI = {10.1007/s00041-021-09822-5},
       URL = {https://doi.org/10.1007/s00041-021-09822-5},
}

@article {BFG16,
    AUTHOR = {Bahouri, Hajer and Fermanian-Kammerer, Clotilde and Gallagher,
              Isabelle},
     TITLE = {Dispersive estimates for the {S}chr\"odinger operator on
              step-2 stratified {L}ie groups},
   JOURNAL = {Anal. PDE},
  FJOURNAL = {Analysis \& PDE},
    VOLUME = {9},
      YEAR = {2016},
    NUMBER = {3},
     PAGES = {545--574},
      ISSN = {2157-5045,1948-206X},
   MRCLASS = {35R03 (35B40 35J10)},
  MRNUMBER = {3518529},
MRREVIEWER = {Wayne\ M.\ Eby},
       DOI = {10.2140/apde.2016.9.545},
       URL = {https://doi.org/10.2140/apde.2016.9.545},
}

@article {DH05,
    AUTHOR = {Del Hierro, Martin},
     TITLE = {Dispersive and {S}trichartz estimates on {H}-type groups},
   JOURNAL = {Studia Math.},
  FJOURNAL = {Studia Mathematica},
    VOLUME = {169},
      YEAR = {2005},
    NUMBER = {1},
     PAGES = {1--20},
}

@article {BGX00,
    AUTHOR = {Bahouri, Hajer and G\'erard, Patrick and Xu, Chao-Jiang},
     TITLE = {Espaces de {B}esov et estimations de {S}trichartz
              g\'en\'eralis\'ees sur le groupe de {H}eisenberg},
   JOURNAL = {J. Anal. Math.},
  FJOURNAL = {Journal d'Analyse Math\'ematique},
    VOLUME = {82},
      YEAR = {2000},
     PAGES = {93--118},
      ISSN = {0021-7670,1565-8538},
   MRCLASS = {58J45 (22E25)},
  MRNUMBER = {1799659},
MRREVIEWER = {Tohru\ Ozawa},
       DOI = {10.1007/BF02791223},
       URL = {https://doi.org/10.1007/BF02791223},
}

@article{ABB19,
	author = {Agrachev, Andrei and Barilari, Davide and Boscain, Ugo},
	title = {A {C}omprehensive {I}ntroduction to {S}ub-{R}iemannian {G}eometry},
	year = {2019},
	journal = {Cambridge University Press},
	pages = {xvii+746 pp.},
	doi = {10.1017/9781108677325},
	type = {Book},
	publication_stage = {Final},
	source = {Scopus},
}

@article{hormanderHypoellipticSecondOrder1967,
  title = {Hypoelliptic Second Order Differential Equations},
  author = {Hörmander, Lars},
  date = {1967},
  journaltitle = {Acta Mathematica},
  shortjournal = {Acta Math.},
  volume = {119},
  number = {0},
  pages = {147--171},
  issn = {0001-5962},
  doi = {10.1007/BF02392081},
  url = {http://projecteuclid.org/euclid.acta/1485889533},
  urldate = {2022-03-22},
  langid = {english}
}

@book{thangaveluHarmonicAnalysisHeisenberg1998,
  title = {Harmonic {{Analysis}} on the {{Heisenberg Group}}},
  author = {Thangavelu, Sundaram},
  year = {1998},
  publisher = {Birkh\"auser Boston},
  location = {Boston, MA},
  doi = {10.1007/978-1-4612-1772-5},
  url = {http://link.springer.com/10.1007/978-1-4612-1772-5},
  urldate = {2022-08-01},
  isbn = {978-1-4612-7275-5 978-1-4612-1772-5},
  langid = {english}
}

@article{mullerRestrictionTheoremHeisenberg1990,
  title = {A {{Restriction Theorem}} for the {{Heisenberg Group}}},
  author = {Muller, Detlef},
  date = {1990},
  journal = {Annals of Mathematics},
  volume = {131},
  number = {3},
  pages = {567--587},
  publisher = {Annals of Mathematics}
}

@article{strichartzSubRiemannianGeometry1986,
  title = {Sub-{{Riemannian}} Geometry},
  author = {Strichartz, Robert S.},
  date = {1986-01},
  journaltitle = {Journal of Differential Geometry},
  volume = {24},
  number = {2},
  pages = {221--263},
  publisher = {Lehigh University},
  issn = {0022-040X},
  doi = {10.4310/jdg/1214440436},
  url = {https://projecteuclid.org/journals/journal-of-differential-geometry/volume-24/issue-2/Sub-Riemannian-geometry/10.4310/jdg/1214440436.full},
  urldate = {2022-11-15}
}

@article{follandSubellipticEstimatesFunction1975,
  title = {Subelliptic Estimates and Function Spaces on Nilpotent {{Lie}} Groups},
  author = {Folland, G. B.},
  date = {1975-12},
  journaltitle = {Arkiv för Matematik},
  shortjournal = {Ark. Mat.},
  volume = {13},
  number = {1-2},
  pages = {161--207},
  issn = {0004-2080},
  doi = {10.1007/BF02386204},
  url = {http://projecteuclid.org/euclid.afm/1485896426},
  urldate = {2023-05-24},
  langid = {english}
}

@article{casarinoRestrictionTheoremMetivier2013,
  title = {A Restriction Theorem for {{M\'etivier}} Groups},
  author = {Casarino, Valentina and Ciatti, Paolo},
  year = {2013},
  journal = {Advances in Mathematics},
  shortjournal = {Advances in Mathematics},
  volume = {245},
  pages = {52--77},
  issn = {00018708},
  doi = {10.1016/j.aim.2013.06.015},
  url = {https://linkinghub.elsevier.com/retrieve/pii/S0001870813002247},
  urldate = {2024-02-20},
  langid = {english}
}

@article{niedorfSpectralMultiplierTheorem2024,
  title = {An ${L}^{p}$-Spectral Multiplier Theorem with Sharp $p$-Specific Regularity Bound on {H}eisenberg Type Groups},
  author = {Niedorf, Lars},
  date = {2024},
  journal= {Journal of Fourier Analysis and Applications},
  volume = {30},
  number = {2},
  pages = {22},
  issn = {1069-5869, 1531-5851},
  doi = {10.1007/s00041-024-10075-1},
  url = {http://arxiv.org/abs/2203.01307},
  urldate = {2024-06-01},
  langid = {english}
}

@book{thangaveluLecturesHermiteLaguerre1993,
  title = {Lectures on {{Hermite}} and {{Laguerre}} Expansions},
  author = {Thangavelu, Sundaram},
  year = {1993},
  series = {Mathematical Notes},
  number = {42},
  publisher = {Princeton University Press},
  location = {Princeton, New Jersey},
  isbn = {978-0-691-00048-0},
  langid = {english},
  pagetotal = {195}
}

@ARTICLE{kochSpectralProjectionsTwisted2004,
	author = {Koch, Herbert and Ricci, Fulvio},
	title = {Spectral projections for the twisted Laplacian},
	year = {2007},
	journal = {Studia Mathematica},
	volume = {180},
	number = {2},
	pages = {103 -- 110},
}

@article{jeongSharpEstimateFof2022,
  title = {Sharp ${L}^{p}$-${L}^{q}$ Estimate For the Spectral Projection Associated with the Twisted {{Laplacian}}},
  author = {Jeong, Eunhee and Lee, Sanghyuk and Ryu, Jaehyeon},
  date = {2022},
  journal = {Publicacions Matematiques},
  volume = {66},
  pages = {831--855},
  issn = {0214-1493},
  doi = {10.5565/PUBLMAT6622210},
  url = {http://mat.uab.cat/pubmat/articles/view_doi/10.5565/PUBLMAT6622210},
  urldate = {2024-06-04},
  langid = {english}
}

@article{Hulanicki1984,
  author = {Hulanicki, Andrzej},
  title = {A functional calculus for {Rockland} operators on nilpotent Lie groups},
  journal = {Studia Mathematica},
  volume = {78},
  number = {3},
  year = {1984},
  pages = {253--266},
  doi = {10.4064/sm-78-3-253-266}
}
		
	\end{document}